\documentclass{amsbook}
\usepackage[cp1251]{inputenc}
\usepackage[english,russian]{babel}
\pdfoutput=1
\usepackage{chapterbib}
\usepackage{bbm}
\usepackage[a5paper, total={113mm,178mm}, left=2cm, includehead]{geometry}

\usepackage{amsfonts,amssymb,amsmath,amstext,mathrsfs}

\usepackage{latexsym}

\usepackage{euscript}
\usepackage{eufrak}

\usepackage{graphicx}

\makeatletter

\renewcommand{\chapter}{\@startsection{chapter}{1}{0pt}{-3.25ex plus -1ex minus-.2ex}{1.5ex plus .2ex}{\bfseries}}

\renewcommand{\section}{\@startsection{section}{1}{0pt}{-3.25ex plus -1ex minus-.2ex}{1.5ex plus .2ex}{\bfseries}}

\renewcommand{\subsection}{\@startsection{subsection}{2}{0pt}{-3.25ex plus -1ex minus-.2ex}{-1em}{\bfseries}}

\renewcommand{\@evenhead}%
{\raisebox{0pt}[\headheight][0pt]{%
\vbox{\hbox to\textwidth{%
\thepage\hfil\strut\leftmark}\hrule}}%
}
\renewcommand{\@oddhead}%
{\raisebox{0pt}[\headheight][0pt]{%
\vbox{\hbox to\textwidth{%
\rightmark\hfil\strut\thepage}\hrule}}%
}
\renewcommand{\@evenfoot}{}
\renewcommand{\@oddfoot}{}

\makeatother

\def\R{\mathbb{R}}
\def\N{\mathbb{N}}
\def\Z{\mathbb{Z}}
\def\C{\mathbb{C}}
\def\K{\mathbb{K}}
\def\l{\mathbf{l}}

\newcommand{\cF}{{\mathcal F}}

\newcommand{\cK}{{\mathcal K}}

\newcommand{\cN}{{\mathcal N}}

\newcommand{\cV}{{\mathcal V}}

\def\0{\boldsymbol{0}}
\def\1{\boldsymbol{1}}

\newcommand {\Rmax}	{\R_{\max}}

\newcommand {\Rmin}	{\R_{\min}}

\newcommand {\Rmaxm}	{\R_{\max,\text{m}}}

\newcommand{\Rmaxmn}	{\Rmaxm^n}

\newcommand{\Rp}	{\R_+}
\newcommand{\Rpn}	{\R_+^n}

\renewcommand{\thesection}{\arabic{section}}

\newtheorem{theorem}{Theorem}[section]
\newtheorem{proposition}{Proposition}[section]
\newtheorem{corollary}{Corollary}[section]

\theoremstyle{definition}
\newtheorem{definition}{Definition}[section]
\newtheorem{example}{Example}[section]

\theoremstyle{remark}
\newtheorem{remark}{Remark}[section]

\numberwithin{equation}{section}

\def\Bbb{\mathbb}

\newcommand{\rbar}{\overline{\R}}

\newcommand{\tr}{\mathop{\rm tr}\nolimits}

\def\states{X}
\newcommand\geo{\gamma}

\newcommand{\sK}{\mathscr{K}}
\newcommand{\sM}{\mathscr{M}}
\newcommand{\sMin}{\sM^{m}}

\newcommand{\arxiv}[1]{arXiv:#1}

\newcounter{assume}[section]
\def\theassume{(A\arabic{assume})}
\newenvironment{assumption}{\refstepcounter{assume}%
\begin{list}{}{}
\item[{\bf \theassume }]}{\end{list}}

\newcommand{\set}[2]{\{#1\mid #2\}}

\def\of#1,#2{(#1,#2)}
\def\new#1{{\em #1}}

\def\diver{{\rm div}}

\newcommand{\doi}[1]{doi:#1}

\newcommand {\ldotsn}	{\{1,\ldots,n\}}
\newcommand {\odiv}	{/}

\def\msn{\medskip\noindent}
\newcommand {\Hilb}	{\text{H}}

\def\gggg{{\bf g}}
\def\gh{\gggg_{\hbar}}

\def\Uq{U_q(sl(n))}

\def\ah{{\mathcal A}_{\hbar}}
\def\ahn{{\mathcal A}_{\hbar, \nu}}
\def\ahq{{\mathcal A}_{\hbar,q}}
\def\O{{\mathcal O}}

\def\End{{\rm End\, }}
\def\Tr{{\rm Tr}}
\def\Sym{{\rm Sym\, }}
\def\Vect{{\rm Vect\,}}
\def\spann{{\rm span\, }}

\def\SS{{\mathcal S}}
\def\MM{{\mathcal M}}
\def\Mnm {{\mathcal M}_{m,n}({\mathcal S})}
\def\M{{\mathcal M}_{n}({\mathcal S})}
\newcommand{{\per}}{\rm per \;}
\def\RR{{\mathcal R}}

\newcommand{\bw}{{\boldsymbol w}}

\def\epi{\mathop{\fam0 epi}\nolimits}
\def\cop{\mathop{\fam0 cop}\nolimits}
\def\dom{\mathop{\fam0 dom}\nolimits}

\newcommand{\dessin}[4]{
\begin{figure}[htb]
\begin{center}
\includegraphics[scale=#2]{#1}
\caption{#3}
\label{#4}
\end{center}
\end{figure}}

\newcommand{\lineika}[2]{\hbox to\textwidth{\normalsize {\em #1} \dotfill  #2}}

\newcommand{\statya}[3]{\noindent\parbox[l]{\textwidth} 
{\parbox[l]{0.9\textwidth}{\small\scshape #1}%
\par\vspace{2mm}%
\noindent\lineika{#2}{#3}}%
\vskip0.5cm}

\newcommand{\nestatya}[2]{\noindent\hbox to\textwidth{\normalsize #1\dotfill  #2}}

\newcommand{\avtor}[5]{\noindent%
\parbox[l]{0.48\textwidth}{{\bfseries #1 (p.~#5)}\\ {\upshape #2}\\ {\scshape #3}\\{\ttfamily  #4}}%
\vskip0.5cm}
\newcommand{\rusavtor}[6]{\noindent%
\parbox[l]{0.48\textwidth}{{\bfseries #1 (p.~#6)}\\ {\slshape (#2)}\\ {\upshape #3}\\ {\scshape #4}\\{\ttfamily #5}}%
\vskip0.5cm}

\newcommand{\nachalo}[2]{\par%
\pagebreak[2]\vspace{1cm plus 3mm minus 0.5mm}%
\begin{center}
{\large\rmfamily\bfseries\upshape #2\par}%
\nopagebreak%
\vspace{2mm plus 1mm}
\nopagebreak%
{\large\itshape #1\par}%
\end{center}
\nopagebreak\vspace{2mm plus 1mm}
\nopagebreak}

\newcommand{\nachaloe}[3]{\par%
\pagebreak[2]\vspace{1cm plus 3mm minus 0.5mm}%
\begin{center}%
{\large\rmfamily\bfseries\upshape #2\footnote{#3}\par}%
\nopagebreak%
\vspace{2mm plus 1mm}
\nopagebreak%
{\large\itshape #1\par}%
\end{center}%
\nopagebreak%
\vspace{2mm plus 1mm}%
\nopagebreak}

\begin{document}
\selectlanguage{english}
\thispagestyle{empty}
\vbox to \textheight{\centering
{\scshape Independent University of Moscow \\
French--Russian Laboratory ``J.-V.~Poncelet''}
\vskip6pt
\hrule

\vss 

{\large International Workshop}

\vss

{\Large
{\fontseries{b}\selectfont IDEMPOTENT AND TROPICAL MATHEMATICS AND~PROBLEMS
OF~MATHEMATICAL PHYSICS}}
\indent

\vss

{\large\slshape G.L.~Litvinov, V.P.~Maslov, S.N.~Sergeev (Eds.)}

\vss

\begin{tabular}{rl}
{\large Organizing committee:} & {\large\slshape G.L.~Litvinov, V.P.~Maslov,}\\
& {\large\slshape S.N.~Sergeev, A.N.~Sobolevski\u{\i}}\\
&\\
{\large Web-site:} & {\verb"http://www.mccme.ru/tropical07"}\\
&\\
{\large E-mail:} & {\verb"tropical07@gmail.com"}
\end{tabular}

\vss

Moscow, August 25--30, 2007 

\vss

{\large\bfseries Volume I}

\vss

Moscow, 2007
}
\newpage
\setcounter{page}{2}
\thispagestyle{empty}
\begin{flushleft}
\parbox[l]{0.8\textwidth}{
\indent {\bfseries Litvinov G.L., Maslov V.P., Sergeev S.N. (Eds.)}\\
\indent Idempotent and tropical mathematics and problems of mathematical physics 
(Vol. I) -- M.: 2007 -- 104 pages
\vskip12pt
\indent {\slshape This volume contains the proceedings of an International Workshop on Idempotent and Tropical Mathematics
and Problems of Mathematical Physics, held at the Independent University of Moscow, Russia, on August 25-30, 
2007.}
\vskip12pt
\indent {\scshape 2000 Mathematics Subject Classification: 00B10, 81Q20, 06F07, 35Q99, 49L90, 46S99, 81S99, 52B20,52A41, 14P99}
}
\end{flushleft}
\vskip 8cm
\begin{flushright}
\copyright\ 2007 by the Independent University of Moscow. All rights reserved. 
\end{flushright}
\newpage
\section*{CONTENTS}
\vskip0.5cm

\nestatya{\scshape Preface}{\pageref{preface}}
\vskip0.5cm

\statya{Representation 
of sta\-tio\-nary sol\-uti\-ons of Ham\-il\-ton-Jac\-obi-Bell\-man 
equ\-ati\-ons: a max-plus point of view}{Marianne Akian}{\pageref{aki-abs}}

\statya{On the assignment problem for a countable state space}{M.~Akian, S.~Gaubert, and V.N.~Kolokoltsov}{\pageref{kol-abs}}

\statya{Dequantization of coadjoint orbits: the case of exponential Lie groups}{Ali Baklouti}{\pageref{bak-abs}}
\thispagestyle{empty}

\statya{Quantum \mbox{Pontryagin} principle and quan\-tum Ham\-il\-ton-Jac\-obi-Bell\-man equ\-at\-ion: a 
max-plus point of view}{Viacheslav P.~Belavkin}%
{\pageref{bel-abs}}

\statya{Tropical \mbox{Pl\"ucker} functions}{V.I.~Danilov, A.V.~Karzanov, and G.A.~Koshevoy}{\pageref{dan-abs}}

\statya{Degree one homogeneous min\-plus dynamic systems and traffic applications: Part I}{N.~Farhi, M.~Goursat, and J.-P.~Quadrat}%
{\pageref{far-abs}}

\statya{Degree one homogeneous min\-plus dynamic systems and traffic applications: Part II}{N.~Farhi, M.~Goursat, and J.-P.~Quadrat}%
{\pageref{far-abs2}}

\statya{Max-plus cones and
semimodules}{F.~Faye, M.~Thiam, L.~Truffet, E.~Wagneur}{\pageref{wag-abs}}

\statya{Duality of cluster varieties}{V.V.~Fock and A.B.~Goncharov}{\pageref{fok-abs}}

\statya{From max-plus algebra to non-linear Perron-Frobenius theory: an approach to zero-sum repeated games}{St\'ephane Gaubert}%
{\pageref{gau-abs}}

\statya{Cyclic projectors and separation theorems in idempotent semimodules}{S.~Gaubert and S.~Sergeev}{\pageref{gau-ser-abs}}

\statya{Pseudo-weak convergence of the random sets defined by a pseudo integral
based on non-add\-it\-ive measure}{T.~Grbi\'c and E.~Pap}{\pageref{grb-abs}}

\statya{The stationary phase method and large deviations}{Oleg V.~Gulinsky}{\pageref{gul-abs}}
\thispagestyle{empty}

\statya{Quantization with a deformed trace}{Dmitry Gurevich}{\pageref{gur-abs}}

\statya{Transformations preserving matrix invariants over semirings}{Alexander E.~Guterman}{\pageref{gut-abs}}

\statya{Tropical geometry and enumeration of real rational curves}{I.~Itenberg, V.~Kharlamov, and E.~Shustin}{\pageref{ite-abs}}

\statya{Abstract convexity and \hbox{cone-vexing abstractions}}{Semen S.~Kutateladze}{\pageref{kut-abs}}

\statya{Interval analysis for alg\-ori\-thms of idem\-po\-tent and tro\-pi\-cal mathematics}{Grigory L.~Litvinov}{\pageref{lit-abs}}

\statya{Dequantization procedures related to the Maslov dequantization}{G.L.~Litvinov and G.B.~Shpiz}%
{\pageref{lit-shp-abs}}


\newpage
\thispagestyle{empty}
\markboth{Marianne Akian}{Preface}
\section*{PREFACE}    
\label{preface}
\vskip0.3cm

\medskip\noindent Idempotent mathematics is 
a new branch of mathematical sciences, rapidly developing  and gaining popularity over the last decade. 
It is closely related to mathematical physics. Tropical mathematics is
a very important part of idempotent mathematics. The literature on the subject is vast and includes numerous 
books and an all but innumerable body of journal papers.

\medskip\noindent An important stage of development of the subject was presented in the book {\it Idempotency}  edited by
J. Gunawardena (Publ. of the Newton Institute, vol {\bf 11},  Cambridge University Press,  Cambridge, 1998). 
This book arose out of the well-known international workshop that was held in Bristol,  England, in October  1994.

\medskip\noindent The next stage of development of idempotent and tropical mathematics was presented in the book
{\it Idempotent Mathematics and Mathematical Physics} 
edited by G.L. Litvinov and V.P. Maslov  (Contemporary Mathematics, vol. {\bf 377},
American Mathematical Society, Providence, Rhode Island, 2005).  
The book arose out of the international workshop that was held in Vienna, Austria, in February
2003.
    
\medskip\noindent The present volumes contain materials  presented for the international workshop  {\it Idempotent and
Tropical Mathematics and Problems of Mathematical Physics}
(Moscow, Russia, August  25-30, 2007).
     
\medskip\noindent It is our pleasure to thank the Independent University of  Moscow and the Poncelet Laboratory 
of  this university as well as the Russian Fund for Basic Research and  CNRS (France) for their important
support. We are grateful to a number of colleagues, especially to  L. Kryukova and M. Tsfasman of the
Poncelet Laboratory,  T.~Korobkova and Yu. Torkhov of the Independent University of Moscow, 
and A. Sobolevski\u{\i} of
the Moscow  State University, for their great help. 
We thank all the authors of the volumes and members of our ``idempotent/max-plus/tropical community'' 
for their contributions, help, and  useful
contacts.
\thispagestyle{empty}                                                             
\begin{flushright}
{\slshape
The editors\\
Moscow,  August 2007}
\end{flushright}

\newpage
\setcounter{section}{0}
\setcounter{footnote}{0}
\nachaloe{Marianne Akian}{Representation of stationary \mbox{solutions} 
of \mbox{Hamilton-Jacobi-Bellman} \mbox{equations}: a \mbox{max-plus 
point of view}}{The present work was partially supported by the RFBR/CNRS grant 05-01-02807.}
\markboth{Marianne Akian}{Stationary solutions of HJB equations}
\label{aki-abs}

This talk gathers two recent works: the first one is a joint work with
S. Gaubert and C. Walsh first presented in~\cite{AGW2},
the second one is a joint work with B. David and S. Gaubert
presented in~\cite{david}.

\section{Nonlinear eigenfunctions and stationary solutions of Ham\-il\-ton-Jac\-obi-Bell\-man equations}
Let us consider a diffusion control model on a subset $X$ of $\R^n$, that is
a stochastic process $\mathbf{x}_t$ with values in $X$ satisfying the stochastic differential equation 
\begin{equation} d\mathbf{x}_t = g(\mathbf{x}_t,\mathbf{u}_t) \ dt + \sigma(\mathbf{x}_t,\mathbf{u}_t)\ d\mathbf{b}_t \label{eds} \end{equation}
 where $(\mathbf{b}_t)$ is a $p$-dimensional brownian motion, 
$\mathbf{u}:= (\mathbf{u}_t)_{t \geq 0}$ (the control) is a stochastic process with values in a subset $U$ of $\R^p$ and adapted to the filtration of $(\mathbf{b}_t)$, and the drift $g: X \times U \rightarrow \mathbb{R}^n$ and the standard deviation $\sigma : X \times U \rightarrow \mathcal{M}_{n,p}(\mathbb{R})$ are given.

The stochastic control problem with horizon $T$ consists in maximizing over all the controls $\mathbf{u}$ the quantity 
\begin{equation} \mathbb{E}\left[ \int_0^T L(\mathbf{x}_s,\mathbf{u}_s)\ ds + \phi(\mathbf{x}_T) \right],\label{critere} \end{equation}
where $\mathbf{x}_t$ is the solution of \eqref{eds} with initial condition $\mathbf{x}_0 = x$, $L$ is the Lagrangian and $\phi$ is a final reward. Let us denote by $v^T(x)$ the value of this optimization problem, and by $S^T$ the map which associates $v^T$ to $\phi$. The familly of operators $\{S^t\}_{t\geq 0}$ is the (non linear) evolution semigroup associated to the control problem.
Moreover, each operator $S^t$ is additively homogeneous ($S^t (\mu+\phi)=\mu+S^t(\phi)$, where $(\mu+\phi)(x)=\mu+\phi(x)$), and order preserving
, thus it is nonexpansive for the sup-norm
.
If the control problem is purely deterministic then the operators $S^t$ are max-plus linear, that is $S^t(\phi\vee \psi)=S^t(\phi)\vee S^t(\psi)$.
In general, the operators $S^t$ are convex, which means that, for all $t\geq 0$ and $x\in X$, the map $\phi\mapsto S^t(\phi)(x)$ is convex on 
$\R^X$.

We say that $\lambda$ is an additive eigenvalue of the evolution semigroup if there exists a function $\phi:X\to\R$ such that for all $t \geq 0$, $S^t\phi = \lambda t + \phi$. The function $\phi$ is called an additive eigenfunction of the evolution semigroup associated to $\lambda$. If $X$ is compact and we restrict ourselves to continuous eigenfunctions, the semigroup has at most one eigenvalue. Moreover, under some regularity assumptions on $L,g,\sigma$, the eigenfunctions are exactly the viscosity solutions $\phi$ of the ergodic Hamilton-Jacobi-Bellman equation 
\begin{equation}\lambda - H(x,D\phi(x),D^2\phi(x)) = 0,\ \ \ \ x \in X,
 \label{HJB} \end{equation}
where the Hamiltonian of the problem is given by 
\begin{equation}H(x,p,A) = \max_{u \in U} ( \frac{1}{2} \tr(\sigma(x,u)\sigma(x,u)^TA) + \langle p, g(x,u)\rangle + L(x,u) ).\label{Ham-Aki} \end{equation}
In that case, $\lambda$ is the maximal mean reward by unit of time (the ergodic reward).

Given an eigenvalue $\lambda$ of $(S^t)_{t\geq 0}$, we are interested 
in characterizing the set $\mathcal{E}_\lambda$ of associated 
eigenfunctions, or of solutions  $\phi$ of~\eqref{HJB}.

\section{Related results}
In the particular deterministic case ($\sigma\equiv 0$),
the discrete-time analogue of this problem consists in the characterization 
of eigenvectors of max-plus linear operators, which has received
a considerable amount of attention, see for 
instance~\cite{
bcoq,bapat98,gondran02,AGW-s,AGW1}.
In the finite dimensional setting, eigenvectors are max-plus linear 
combinations of extremal generators which are themselves in 
bijection with the ``critical classes'' (critical classes can be seen as
the max-plus analogue of recurrent classes).
The deterministic continuous time problem itself has been
studied by Maslov, Kolokoltsov, Samborskii and other members
of the ``idempotent analysis'' school~\cite{KM-Akian,maslov92}, 
and by Rouy and Tourin~\cite{RT}
in some special cases.
More recently, it has been studied as a part of the ``weak KAM'' theory developped by Fathi \cite{Fat2,Fat3,Fat} and Fathi and Siconolfi~\cite{FatSico}. In this setting, it is shown that when $\states$ is a Riemannian manifold
and the Lagrangian has smoothness and strict-convexity
properties, an eigenfunction is uniquely determined by its restriction 
to the ``projected Aubry set''. This set can be thought of as 
a continuous analogue of the set of ``critical states'' of 
the finite dimensional max-plus spectral theory.
The case of a non-compact state space $X$ has been studied
in different settings by Contreras~\cite{Cont}, 
and by Ishii and Mitake~\cite{Ishii}.

In Section~\ref{cormac}, we present briefly the results of~\cite{AGW2}, where
we showed that general representation results hold in 
the continuous-time setting, without any regularity assumption on the 
Lagrangian.
These results were inspired by the discrete-time theory developped 
in~\cite{AGW1} and rely on a compactification of the
state space $\states$, which is the max-plus analogue of the
Martin compactification in potential theory
, and is  similar to the compactification of metric spaces by horofunctions
(generalised Busemann functions).

It is natural to ask whether an analogue of the weak KAM theory can be
developped for stochastic control problems.
Such an analogue does exist in the simpler 
finite state space and discrete time case. Indeed, it is shown in~\cite{AG}
that the additive eigenvectors 
are determined by their restriction to
a subset of ``critical states'' obtained by taking exactly
one element in each ``critical class''. Here 
critical states are defined in terms of subdifferentials, and can
be interpreted as follows:
a state is critical if there is an optimal stationary randomised strategy
for which it is recurrent, and two critical states are in the same
critical class if they are in the same recurrence class for
an optimal stationary randomised strategy.
In the continuous time stochastic case, characterization results were obtained
in the uniformly elliptic case and under various settings, 
by Bensoussan~\cite{Ben}, Akian, Sulem and Taksar~\cite{AkianTaksar}, and 
Barles and Da Lio~\cite{Bar2}: the eigenfunction is then unique up to an 
additive constant. 

In Section~\ref{david}, we present briefly the results of~\cite{david}
giving a description of the additive eigenspace similar to
the one of~\cite{AG}, in the simplest degenerate case in which there is only a finite number of "singular points" $x_1,\ldots ,x_k$ playing the role of critical states 
and classes. In the deterministic case similar results were obtained 
in~\cite{RT,KM-Akian}.

\section{Hamilton-Jacobi equations on non-compact spaces}
\label{cormac}
In~\cite{AGW2}, we consider general time continuous semigroups 
$(S^t)_{t\geq 0}$ of max-plus linear operators with kernel.
We assume that $S^t:\rbar^X\to\rbar^X$  can be written as
$ S^tg(x) =\sup_{y\in \states} \left(S^t\of x,y+g(y)\right)$.
for some function $(x,y)\mapsto S^t\of x,y\in \rbar$.
This includes the case of the evolution semigroup associated to
the deterministic optimal control problem with dynamics~\eqref{eds} 
with $\sigma\equiv 0$ and criteria~\eqref{critere}.
Without loss of generality, we assume that $\lambda=0$,
in which case an eigenfunction is called a harmonic function.
Harmonic functions may take the $-\infty$ value,
so that the set $\mathcal{E}_0$ of harmonic functions is a semimodule over the max-plus semiring $\Rmax$ (recall that this is the set $\R\cup\{-\infty\}$ endowed with $\max$ as addition and $+$ as multiplication).

We need the following assumptions~:
\begin{assumption}\label{irreducible}
$S^*\of x,y:=\sup_{t\geq 0} S^t\of x,y$ is finite for all $x$ and $y$ in $\states$.
\end{assumption}
\begin{assumption}\label{assumption}
For all $t\geq 0$ and $x,y\in\states$,
$S^t\of{x},{y} = \sup_\geo\{I(\geo)\}$, where the supremum is taken over all
paths $\geo:[0,t]\to\states$ from $x$ to $y$, and where
the reward $I(\geo)$ is defined as
\begin{align*} I(\geo):=
   \inf_{t_0,\dots,t_n}
\sum_{i=0}^{n-1} S^{t_{i+1}-t_{i}}\of{\geo(t_i)},{\geo(t_{i+1})}, 
\end{align*}
with the infimum taken over all finite increasing sequences
$(t_i)$, $i\in\{0,\dots,n\},$ in $[0,t]$ with $t_0=0$ and $t_n=t$.
\end{assumption}

The (max-plus) \new{Martin kernel} of the semigroup $(S^t)_{t\geq 0}$
with respect to the basepoint $b\in X$ is defined by:
\[
K\of x,y = S^*\of x,y- S^*\of b,y \enspace .
\]
The (max-plus) \new{Martin space} $\sM$ of $(S^t)_{t\geq 0}$
is the closure in the topology of pointwise convergence
of the set $\sK:=\set{K\of\cdot,y}{y\in \states}\subset \R^\states$. 
Any element of $\sM$ is super-harmonic, which means that it satisfies 
$S^t\phi\leq \phi$. 

For all functions $\xi: \states \to \Rmax$ and for all $\eta\in \sM$,
we set:
\begin{align*}
\mu_\xi(\eta ):= \limsup_{K\of\cdot,x\to \eta} \left( S^*\of b,x+\xi(x) \right)
\enspace ,
\end{align*}
and if $\xi\in\sM$ we set: $H\of \eta,\xi := \mu_\xi(\eta)$.
The kernel $H$ extends the kernel $S^*$ from $X\times X$ 
to $\sM\times \sM$, up to a normalization, since \[
H\of{K\of\cdot,x},{K\of\cdot,y}= S^*\of b,x + S^*\of x,y - S^*\of b,y \enspace.\]
The \new{minimal boundary} $\sM^m$ of $(S^t)_{t\geq 0}$ is the 
set of elements $\xi$ of $\sM$ that
are harmonic and satisfy $H(\xi,\xi)=0$.

\begin{theorem}[\protect{\cite[Theorem 3.11]{AGW2}}]\label{representation}
Under Assumptions \ref{irreducible}--\ref{assumption}, 
a function $h:\states\to\Rmax$ is harmonic if and only if it can be written
\begin{align}
\label{star4}
h=\sup_{w\in \sMin} (\nu(w) + w) \enspace ,
\end{align}
where $\nu$ is some upper semicontinuous function from $\sMin$ to $\Rmax$.
Moreover, $\mu_h$ is the greatest $\nu$ satisfying this equation.
\end{theorem}
Here, $\mu_h$ plays the role of the spectral measure in
the Martin representation theorem.
In~\cite{AGW2}, we also show that the elements of the minimal Martin boundary 
are precisely the extremal generators of
$\mathcal{E}_0$ normalized in such a way that $w(b)=0$,
and that they are in correspondance with almost-geodesics.

\section{A degenerate Hamilton-Jacobi-Bellman equation on the torus}
\label{david}
In~\cite{david}, we study
a simple degenerate stochastic control model on the torus
$\mathbb{T}^n = \mathbb{R}^n / \mathbb{Z}^n$.
We consider the stochastic model
given by~\eqref{eds} with the criteria~\eqref{critere} on the
set $X=\mathbb{T}^n$.
We need the following assumptions:
\begin{assumption}\label{A1}
$L$ is $\mathcal{C}^2$ on $\mathbb{T}^n \times U$.\end{assumption}
\begin{assumption}\label{A2} $g$ and $\sigma$ are Lipschitz continuous on 
$\mathbb{T}^n \times U$.\end{assumption}
\begin{assumption}\label{A3}
$L$ takes non-positive values.\end{assumption}
\begin{assumption}\label{A4}
There exists $k$ distinct points  $x_1,\ldots , x_k$ of $\mathbb{T}^n$ 
such that:
\markboth{M.~Akian, S.~Gaubert, V.N.~Kolokoltsov}{Stationary solutions of HJB equations}
\begin{itemize}
\item[(a)] $\forall i=1, \ldots  , k, \exists u_i \in U$ such that $g(x_i,u_i) = 0, L(x_i,u_i) = 0$ and $\sigma(x_i,u_i)=0$.
\item[(b)] $\forall x \in \mathbb{T}^n \smallsetminus \{x_1,\ldots ,x_k\}, \forall u \in U,$ at least one of the following properties holds:
(i) $L(x,u) < 0$ or (ii) $\sigma(x,u)\sigma(x,u)^T$ is a positive definite matrix. 
\end{itemize}\end{assumption}
\begin{assumption}\label{A5}
For all $i=1,\ldots ,k$, there exists $u^{(i)}: \mathbb{T}^n \rightarrow U$ Lipschitz continuous satisfying $u^{(i)}(x_i) = u_i$, and a continuous function $W^{(i)}:\mathbb{T}^n \rightarrow \mathbb{R}$ which vanishes on $x_i$ and is positive elsewhere, and which is a viscosity solution of 
\begin{equation} \begin{array}{l}
\displaystyle - \frac{1}{2} \operatorname{tr}(\sigma(x,u^{(i)}(x))\sigma(x,u^{(i)}(x))^TD^2W^{(i)}(x)) \\[.5ex] 
\displaystyle \quad
- \langle g(x,u^{(i)}(x)) , \nabla W^{(i)}(x) \rangle + L(x,u^{(i)}(x)) \geq 0,
\quad x \in \mathbb{T}^n.\end{array} \label{lyap} \end{equation}
\end{assumption}
Assumptions \ref{A3}--\ref{A5} ensure that the point $x_i$ is stabilized in probability by the control $u^{(i)}$.
Since $X$ is compact, the evolution semigroup has a unique eigenvalue,
and the previous assumptions imply that this eigenvalue is 0. The following result shows in particular that the set $\{x_1,\ldots ,x_k\}$ plays a role analogous to the projected Aubry set.

\begin{theorem}[\protect\cite{david}]\label{repr1}
Under Assumptions \ref{A1}--\ref{A5}, the map $\mathcal{E}_0 \rightarrow \mathbb{R}^k , v \mapsto  (v(x_1),...,v(x_k)) $ is a sup-norm isometry,
whose image is a non-empty closed convex subset $C$ of $\mathbb{R}^k$,
that is invariant by all the translations 
$(v_1,\ldots ,v_k) \mapsto (\mu+v_1,\ldots ,\mu +v_k)$ ($\mu \in \mathbb{R}$).
\end{theorem}
A more precise description of the convex set $C$ is given in~\cite{david}.


\newpage
\setcounter{footnote}{0}
\setcounter{section}{0}
\nachaloe{\mbox{M.~Akian}, \mbox{S.~Gaubert}, and \mbox{V.N.~Kolokoltsov}}%
{On the assignment problem for a countable state space}%
{Partially supported by the joint RFBR/CNRS grant 05-01-02807.}
\label{kol-abs}

\section{Introduction and formulation of the results.}

Our main results are formulated below as Theorems 1.1 - 1.3. The first
two theorems are proved in Section 2. A proof of Theorem 1.3 will be
given elsewhere. Section 3 contains the algebraic interpretation of
our results and methods.

 Let $X$ be either the set of natural
numbers $\N$ or that of all integer numbers $\Z$.
%
%
Further we work with infinite matrices
$B=(b_{ij})$, $i,j \in X$ that will always have entries in $\R\cup
\{-\infty\}$ and satisfy the following condition:

 (C) For any $i$ there is a $j$ such that $b_{ij} \neq -\infty$,
 for any $j$ there is an $i$ such that $b_{ij} \neq -\infty$,
   and $b_{ij} \to -\infty$ as $|i-j| \to \infty$.

Any such matrix $B$ defines\markboth{M.~Akian, S.~Gaubert, V.N.~Kolokoltsov}{On the assignment problem}
the mapping $B$ from the space of  functions bounded below on $X$
to the set $\R^X$ of all real valued functions on $X$, by the formula
\begin{equation} \label{eq1.1}
(Bf)_i =\sup_j \{b_{ij}-f_j\}
\end{equation}
By $B^T$ we shall denote the transpose matrix of $B$ and the
corresponding operator
\begin{equation} \label{eq1.2}
(B^Tg)_j =\sup_i \{b_{ij}-g_i\}.
\end{equation}

A 
crucial fact about $B^T$ is that the pair $(B,B^T)$ defines a
Galois connection in $\R^X$, which means in particular (see
\cite{AGK2}) that $B^T$ is a generalized inverse to $B$ in the sense
that if the equation $Bf=g$ with a given $g\in \R^X$ has a solution
$f\in \R^X$, then necessarily $\tilde f=B^Tg$ is also a solution of
this equation.

The infinite dimensional theory depends crucially on the 
class of functions, in which the solutions to the equation $Bf=g$
are sought, and on the corresponding definitions of solutions to
the assignment problem. 
Let $l_{\infty}$; $l_1$; and $l_0$ denote
respectively the spaces of functions $s=(s_i)$, $i\in X$ on $X$
such that $\sup_i |s_i|<\infty$; $\sup_n \sum_{|i|\leq n}
|s_i|<\infty$; and the (finite) limit $\lim_{|n|\to \infty} s_n$
exists. Let $\l=\l(X)$ be any of these spaces.

\begin{definition}
{\rm  A matrix $B$ (satisfying (C)) will be called
 {\em $\l$-strongly regular} if there exists a function $g\in \l$ such that (i) $f=B^T
 \in \l$, (ii) $f$ is the unique solution in $\R^X$ of the
 equation $Bf=g$ and (iii) $g$ is the unique solution in $\R^X$ of
 the equation $f=B^Tg$. In this case $g$ (respectively, $f$) is
 said to belong to the {\em $\l$-simple image} of $B$ (respectively, of $B^T$).}
\end{definition}

 Of course, it follows from this definition that $B$ is $\l$-strongly
 regular if and only if $B^T$ is strongly regular.

\begin{remark} 
One can show
(though this is not obvious) that in the case of finite $X$
our definition
coincides with the standard definition of strong regularity given by P.~Butkovi\v{c}, see \cite{But03-Kol}.
In fact, we added a crucial additional condition in our definition,
which turns out to be automatically fulfilled for finite, but not for infinite
$X$.
\end{remark}

\begin{definition}
{\rm
For any two bijections $F,G : X \mapsto X$, we define the {\em distance} between
them by}
$$
\rho (F,G)=\sup_n |F(n)-G(n)|.
$$
\end{definition}

The binary relation $F\sim G$ iff $\rho (F,G)<\infty$ is clearly an
equivalence relation on the set of bijections defining the
decomposition of this set into non-intersecting classes. We shall
say that $F$ is {\em locally bounded} if it is equivalent to the identity
map.

\begin{definition}
{\rm
 A bijection $F : X \mapsto X$
 is called a {\em (global) solution} or a
 {\em strong solution} respectively
 to the assignment problem for a matrix $B$ if}
 \begin{equation}
  \label{eq1.3}
\liminf_{n\to \infty} \sum_{|i|\le n} (b_{iF(i)}- b_{iG(i)})\ge 0
\end{equation}
{\rm for any other bijection $G: X\mapsto X$ or if}
\begin{equation}
  \label{eq1.4}
\liminf_{n\to \infty} \sum_{|i|\le n} (b_{iF(i)}- b_{iG(i)})> 0,
\end{equation}
{\rm respectively.
We say that this solution is
a {\em locally bounded $\l$-solution}, if $F$ is locally bounded and the
"solution sequence" $b_{iF(i)}$ belongs to $\l(X)$. We say that $F$
is a {\em local solution} if \eqref{eq1.3} (or \eqref{eq1.4} respectively)
holds for all $G$ 
such that the distance between $F$ and $G$ is finite.
}
\end{definition}

If a strong solution exists, then the solution to the assignment
problem is obviously unique.

\begin{definition}
{\rm A matrix $B$ is called
{\em normal} (respectively, {\em strongly normal}) if all its non-diagonal
entries
 are non-positive (respectively, negative) and $b_{ii}=0$ for all
 $i$.}
\end{definition}

This definition is literally the same as the usual finite-
dimensional one (see \cite{Bu}). The normal (respectively, strongly
normal) matrices present a class of examples, where the identity map
is an obvious locally bounded $l_1$-solution (respectively, unique
solution) to the assignment problem. As our first result will show,
this class of matrices presents natural "normal forms" for strongly
regular matrices.

\begin{definition}
{\rm Matrices $B$ and $C$ are called
 {\em (locally bounded) $\l$-similar} if
 there exist two locally bounded bijections $H : X \mapsto X$, $K : X \mapsto X$
  and two
 vectors $g$ and $f$ from $\l(X)$ such that}
 \begin{equation}
  \label{eq1.5}
 c_{ij}= b_{H(i)K(j)}-\phi_i-\psi_j.
 \end{equation}
\end{definition}

 This is also a standard definition in the case of finite $X$ (see
 e.g. \cite{Bu} and
 Section 4 below for an intuitive interpretation).
 The importance of this notion is basically due to the following
 result.

  \begin{proposition}
\label{prop1.1}
 (i) Conditions (C) and $\l$-strong regularity
for matrices with entries in
 $\R\cup \{-\infty\}$
 are all invariant under $\l$-similarity.
(ii) The property to have an $l_1$- solution
 (in particular locally bounded or strong) to
 the assignment problem for matrices with entries in
 $\R\cup \{-\infty\}$
 is invariant under $l_1$-similarity.
(iii) The property to have a locally bounded local $l_0$- solution
 (in particular strong) to
 the assignment problem for matrices with entries in
 $\R\cup \{-\infty\}$
 is invariant under $l_0$-similarity.
 \end{proposition}

\begin{proof} (i) The invariance of condition (C) is obvious.
The invariance of $\l$-strong regularity follows from the
observation that if $C$ and $B$ are related by \eqref{eq1.5} then
the equation $g=Cf$ is equivalent to the equation
$$
(g+\phi)(H^{-1})=B((f+\psi)(K^{-1})).
$$

(ii) Let a bijection $F$ be a solution to the assignment problem of
a matrix $C$. Notice that
\begin{equation}
  \label{eq1.6}
 \sum_{|i|\le n} (c_{iF(i)}- c_{iG(i)})
 = \sum_{|i|\le n} (b_{H(i)KF(i)}- b_{H(i)KG(i)})
  + \sum_{|i|\le n} (\psi_{G(i)}- \psi_{F(i)}).
  \end{equation}
  Clearly, the last sum on the r.h.s. tends to zero as $n\to \infty$
  whenever $\psi \in l_1$.
Hence,  $F$ is an $l_1$-solution (respectively, a strong $l_1$-solution)
to the assignment problem for the matrix $C$  if and only if
the mapping $KFH^{-1}$ is an $l_1$-solution
(respectively, a strong $l_1$-solution) to the assignment problem for the matrix $B$.

 (iii)
By the previous discussion, it suffices to show that the sequence
$$
 b_n= \sum_{|i|<n} (\psi_{G(i)}-\psi_{F(i)})
$$
tends to zero as $n\to \infty$ 
if and only if both $F$ and $G$ are not infinitely far from the identity map.
To this end, observe that due
to the last condition, there exists a natural number $p$ such that
for every $n$ the sets $\{ F(i): |i| <n\}$ and
 $\{ G(i): |i| <n\}$ both contain the set  $\{i: |i| <n-p\}$.
 Hence
 $$
 b_n =\psi_{G(i_1)}+\psi_{G(i_2)}+...+\psi_{G(i_p)}
  -\psi_{F(j_1)}-...-\psi_{G(j_p)},
  $$
where $i_k$ (respectively $j_k$) are such that $|G(i_k)|>n-p$
(respectively $|F(j_k)|>n-p$).
 We recall now that the function $\psi_n$ has a finite limit as
$n\to\infty$, which immediately implies that $b_n$ tends to zero (the statement that we wanted to prove).
This
was the crucial application of this assumption, which seems to be
the weakest possible to provide a link between the
solutions to the assignment problem for similar matrices.
\end{proof}

\begin{theorem}\label{th1} A matrix $B$ (satisfying (C)) is
$\l$-strongly regular if and only if it is $\l(X)$ similar to a
strongly normal matrix.
\end{theorem}

It is of course interesting to know what can be said about the
assignment problem for a regular matrix itself (not just for some of
its similar matrices). In the analysis of this question (as well as
the inverse one), an important role is played by the following
functions describing in some sense the size of the problem. Namely,
let $F$ be a (possibly local) solution to the assignment problem of
a matrix $B$. Let "the optimal distance" between the points $i,j$ be
defined as
\begin{equation}
  \label{eq1.7}
\begin{split}
&\tilde b_{ij}=\sup [b_{iF(i_1)}-b_{i_1F(i_1)}
 +b_{i_1F(i_2)}-b_{i_2F(i_2)}+...\\
&+b_{i_{n-1}F(i_n)}-b_{i_nF(i_n)}
  +b_{i_nF(j)}-b_{jF(j)}],
\end{split}   
\end{equation}
 where $\sup$ is taken over all $n=1,2,...$ and all collections
  $i_1,i_2,...,i_n$ of points from $X$, and
"the potential" and "the inverse potential" as the functions on $X$ are
given respectively by
 \begin{equation}
  \label{eq1.8}
\phi_i=\sup_j \tilde b_{ij}, \qquad \tilde \phi_i=\sup_j \tilde
b_{ji}.
   \end{equation}
 The following simple
  properties of these functions are crucial:
  \begin{enumerate}
\item the values of $\tilde b$, $\phi$ and $\tilde \phi$ do no change
 if one takes the $\sup$
only over families $i_1,...,i_n$ with pairwise disjoint points (in
fact, any cycle gives a non-positive contribution due the
assumption that $F$ is a solution to the assignment problem);
\item  $\tilde b_{ii}$, $\phi_i$ and $\tilde \phi_i$ are nonnegative
 for all $i$ (in fact, take $n=1$ and $i_1=i$
in \eqref{eq1.7});
\item  the functions $\phi$ and $-\tilde \phi$ satisfy the equation
 \begin{equation}
  \label{eq1.9}
  f_i=\sup_j [b_{iF(j)}-b_{jF(j)}+f_j],
    \quad \forall i\in X,
   \end{equation}
   or, equivalently,
\begin{equation}
  \label{eq1.10}
  f=B\psi, \qquad \psi_j=b_{F^{-1}(j)j}-f_{F^{-1}j}.
    \quad \forall i\in X,
   \end{equation}
   \item  the functions $\phi$ and $-\tilde \phi$ satisfy the equation
 \begin{equation}
  \label{eq1.11}
  f_i=\sup_j [\tilde b_{ij}+f_j],
    \quad \forall i\in X.
   \end{equation}
\end{enumerate}

Observe that if $B$ is normal then $\tilde b_{ij}\leq 0$ and
$\phi_i=\tilde \phi_i=0$ for all $i,j$.

\begin{theorem}\label{th2} (i) If a matrix $B$ (satisfying (C)) is
$l_0$-strongly regular, then it has a (necessarily unique) locally
bounded local strong $l_0$-solution to its assignment problem such
that
\begin{equation}
  \label{eq1.12}
\limsup_{i,j\to \infty} \tilde b_{ij}\leq 0,
\end{equation}
and 
that
the potentials $\phi$ and $\tilde \phi$ are bounded functions.
 (ii) If
$B$ is $l_1$-strongly regular, then this solution is also a global
$l_1$-solution.
\end{theorem}

To prove the converse to Theorem \ref{th2}, we shall use
the following additional technical assumption on a solution to the assignment
problem:

(B$(\l(X))$) Either the potential $\phi$ or the inverse potential
$\tilde \phi$ belong to $\l(X)$.

\begin{theorem}\label{th3} If $\l$ is either $l_0$ or $l_1$ and
if the assignment problem for a matrix $B$ has a (possibly local)
locally bounded strong $\l$-solution satisfying condition (B$(\l(X))$),
then $B$ is strongly $\l$-regular.
\end{theorem}

We have to
indicate
an unpleasant small gap between
the necessary condition and the sufficient condition: from strong $l_0$-regularity it
follows that the potential $\phi$ belongs to $l_{\infty}$, but in
Theorem \ref{th3} we assume that $\phi \in l_0$ (which implies
\eqref{eq1.12}). However,  this discrepancy vanishes when we consider classes of
similar matrices, as the following direct
corollary of Theorem \ref{th1} and \ref{th3} suggests.

\begin{corollary}
 Let $\l(X)$ be either $l_0$ or $l_1$. Then a
matrix $B$ is strongly $\l(X)$-regular if and only if it is
$\l(X)$-similar to a matrix having a strong solution to the
assignment problem satisfying condition (B$(\l(X))$).
\end{corollary}

\section{Coverings and sub-differentials. Proofs of Theorems \ref{th1} and \ref{th2}.}

 For the analysis of the equation $Bf=g$ (also in a more general
 setting of uncountable $X$), an important role belongs to the
 notion of (abstract) sub-differentials.

\begin{definition}
{\rm  For a matrix $B$ and $f,g \in \R^X$, the
 abstract sub-differentials (or $B$-sub-differentials) are defined
 as follows (see \cite{AGK2} and references therein)}
 $$
 \partial f(j)=\{ k\in X: (Bf)_k=\sup_l \{b_{kl}-f_l\}=b_{kj}-f_j \},
 $$
 $$
  \partial ^T g(i)
   =\{ k\in X: (B^Tg)_k=\sup_l \{b_{lk}-g_l\}=b_{ik}-g_i \}.
 $$
\end{definition}

 For a given $f$ the sub-differential is a mapping
 from $X$ to the set $P(X)$ of
 subsets of $X$. For any such mapping $G$ its inverse mapping
 $G^{-1}: X \mapsto P(X)$ is naturally defined as
  $G^{-1}(j)=\{i: j\in G(i)\}$.

We shall start with the following well known basic property of
sub-differentials that we prove here for completeness.

 \begin{proposition}\label{prop2.1}
 If $g=BB^Tg \in \R^X$, then $(\partial ^Tg)^{-1}=\partial B^Tg$.
 \end{proposition}

\begin{proof}
$$
(\partial ^T g)^{-1}(j)
   =\{ i: (B^Tg)_j=\sup_l \{b_{lj}-g_l\}=b_{ij}-g_i \},
 $$
 which is the same as $g_i=b_{ij}-(B^Tg)_j$, or equivalently
 $B(B^Tg)_i=b_{ij}-(B^Tg)_j$, and which means that $i\in \partial
 (B^Tg)(j)$.
\end{proof}

 \begin{proposition}\label{prop2.2}
Suppose that functions $g,\ B^Tg\in\R^X$ are  bounded from below. 
Then $B^Tg$ is a solution to the equation $Bf=g$ if and only
 if $\partial ^Tg(i) \neq \emptyset$ for all $i$ or, equivalently,
 if the family of the sets $(\partial ^Tg)^{-1}(j)$, $j\in X$, is
 a covering of $X$.
 \end{proposition}
\begin{proof} This is a direct consequence of a more
general Theorem 3.5 from \cite{AGK2}, where one only has to
observe that the assumption that $f=B^Tg$ is bounded from below
ensures that the set $\{j: b_{ij}-f_j\geq \beta\}$ is finite for
any $i\in X$ and $\beta \in \R$, which is the crucial condition
for the applicability of this theorem.
\end{proof}

\begin{definition} {\rm Let $G$ be a mapping from $X$ to the set of its
subsets $P(X)$ and let the family of subsets $\{G(j)\}_{j\in X}$
be a covering of $X$. An element $j\in X$ is called {\em essential}
(with respect to this covering) if
$$
\exists i\in X: i\notin \cup_{k\in X\setminus j} G(k).
$$
The covering is called {\em minimal} if all elements of $X$ are
essential.}
\end{definition}

 \begin{proposition}\label{prop2.3}
 Suppose that functions $g,\ B^Tg\in\R^X$ are bounded from below.
Then $B^Tg$ is the unique solution of equation $Bf=g$
 in $\R^X$ if and only
 if $(\partial ^Tg)^{-1}(j)$, $j\in X$, is
 a minimal covering of $X$.
 \end{proposition}

\begin{proof} This is again a consequence of a more
general Theorem 4.7 from \cite{AGK2}.
\end{proof}

The key point in proving Theorems \ref{th1} and \ref{th2} is contained
in the following statement.

\begin{proposition}\label{prop2.4}
 Suppose that  functions $g,\ B^Tg\in\R^X$ are
 bounded from below, and such that $f=B^Tg$ is the unique solution of equation $Bf=g$,
 and $g$ is the unique solution to the equation $B^Tg=f$. Then there
 exists a locally bounded bijection $F:X \mapsto X$ such that
  \begin{equation} \label{eq2.1}
j=F(i) \iff \partial f(j)=\{i\}
 \iff \partial ^Tg (i)=\{j\} \iff j=F(i).
\end{equation}
In particular,
\begin{equation} \label{eq2.2}
\begin{aligned} 
\forall & k\neq F(i)& \quad &b_{iF(i)}-f_{F(i)} > b_{ik}-f_k,\\
\forall & k\neq i &\quad &b_{iF(i)}-g_i > b_{kF(i)}-g_k.
\end{aligned}
\end{equation}
 \end{proposition}

\begin{remark}  As one easily checks, the inverse statement holds as
well: if a locally bounded bijection $F$ and bounded below functions
$f$, $g$ satisfy \eqref{eq2.1}, then $f=B^Tg$ is the unique solution
of  the equation $Bf=g$,
 and $g$ is the unique solution of the equation $B^Tg=f$.
\end{remark}

\begin{proof}[Proof of Prop.\  \ref{prop2.4}]
Applying Proposition \ref{prop2.3} to the equation $B^Tg=f$ one
concludes that for all $i$ there exists $j$ such that
$j\in (\partial f)^{-1}(i)$, but $j\notin (\partial f)^{-1}(k)$
for any $k\neq i$. In other words $(\partial f)(j)=\{i\}$,
which by Proposition \ref{prop2.1} means that
$(\partial^T g)^{-1}(j)=\{i\}$. Hence, defining the mapping
$F : X\mapsto P(X)$ by the formula
{\sloppy

} 
$$
 F(i) =\{ j : (\partial f)(j)=\{i\}\}
  =\{ j : (\partial^T g)^{-1}(j)=\{i\}\},
 $$
 one deduces that $F(i) \neq \emptyset $ for all $i$ and
 $F$ is injective in the sense that $F(i)\cap
 F(k)=\emptyset$ whenever $i\neq k$. Applying now Proposition
 \ref{prop2.3} to the equation $Bf=g$, one finds that
 for all $j$ there exists $i$ such that
 $(\partial^T g)(i)=\{j\}$. From this, one easily concludes that
 each set $F(i)$ contains precisely one point and that $F$ is
 surjective, which finally implies that $F$ is a bijection
 $X\mapsto X$ such that \eqref{eq2.1} holds.

Let us show that $F$ is bounded. In fact, since $B$ satisfies (C)
and $f$ is bounded from below, it follows that for any $C>0$ there
exists $N$ such that $b_{ij}-f_j <-C$ whenever $|i-j|>N$. On the
other hand, as $g$ is bounded from below, $b_{iF(i)}-f_{F(i)}=g_i$
is bounded from below, and hence $|i-F(i)|<N$ for large enough $N$
and all $i$.
\end{proof}

\begin{proof}[Proof of Theorem \ref{th1} .]
 Let $F$ be a bijection
constructed in Proposition \ref{prop2.4}. From the equation
$b_{iF(i)}-f_{F(i)}=g_i$ it follows that $\{b_{iF(i)}\} \in \l(X)$
whenever $f,g \in \l(X)$.
 Hence the matrix $C$ with entries
 $$
 c_{ij}= b_{iF(j)}-f_{F(j)}-(b_{iF(i)}-f_{F(i)})
 $$
 is strongly normal and is $\l$-similar to $B$.
\end{proof}

\begin{proof}[ Proof of Theorem \ref{th2} .] 
The existence of required
solution follows from Proposition \ref{prop1.1}. This solution is actually
given by the bijection $F$ constructed in Proposition \ref{prop2.4}.
 The boundedness of $\phi$ and $\tilde \phi$ follows of course from
\eqref{eq1.11}. To prove the latter, one observes that according to
the second inequality of \eqref{eq2.2}
$$
\tilde b_{ij} \le g_i-g_j,
$$
and the r.h.s. of this inequality tends to 0 as $i,j \to \infty$,
since $g \in l_0$.
\end{proof}

\section{Algebraic interpretation}

A natural algebraic language for the analysis of discrete event
systems and optimal control is supplied by the so called idempotent
algebra, in particular the $(\max,+)$-algebra (see e.g. \cite{Gu},
\cite{KM-Kolok}). The $(\max,+)$-algebra deals with the semiring $\R\cup
\{-\infty\}$ equipped with the binary operations $\oplus=\max$ and
$\otimes =+$ and with finite-dimensional semimodules over this
semiring. The main impetus to the development of this algebra (and
further its infinite-dimensional generalizations, see \cite{KM-Kolok},
\cite{LMS}) was a simple observation that the basic Bellman operator
\begin{equation} \label{eq4.1}
(Bf)_i=\sup_j(b_{ij}+f_j), \quad i\in X
\end{equation}
of the optimal control theory is linear in this structure, i.e.
$$
B(\alpha_1\otimes f_1\oplus \alpha_2\otimes f_2
 =\alpha_1\otimes B(f_1)\oplus \alpha_2\otimes B(f_2)
 $$
 for $\alpha_1, \alpha_2 \in  \R\cup
\{-\infty\}$ and $f_1,f_2 \in (\R\cup \{-\infty\})^X$.
 (Note that
previously we denoted by $Bf$ the operator which would now be
denoted by $B(-f)$; this was more convenient for the study of the
inversion of $B$.) In fact the operators of type \eqref{eq4.1} are
the most natural $(\max,+)$-linear operators, though they do not
exhaust all of them (see e.g. \cite{Ak}, \cite{Ko}, \cite{LMS} and
references therein for classical and recent results on this "kernel
type" representations).

The introduction of main notions and objects studied in this article
was
motivated by the development of the $(\max,+)$-algebra that
supplies a clear intuitive interpretation for them. For instance, the
strong regularity turns to be a linear algebraic problem connected
with the inversion of matrices (or more generally linear operators
having kernel representation). Our notion of similarity is obtained
by rewriting  the classical algebraic notion of similarity of
matrices in $(\max,+)$. Next, the solution to the assignment problem
$$
\sup_F \sum_{i\in X}b_{iF(i)} = \oplus_F \otimes_{i\in X}b_{iF(i)}
$$
turns out to be the $(\max,+)$ analogue of the classical algebraic
notion of \markboth{Ali Baklouti}{On the assignment problem}
matrix permanent.
Namely, solving the assignment
problem means finding the $(\max,+)$- permanent of a matrix (in our
case infinite dimensional). If this solution is strong, then one says that
this matrix has a strong permanent.

An important tool in algebra is given by the so called Kleene star
that for a given matrix $A$ is defined by
$$
A^{\star}=\oplus_{k=0}^{\infty} A^k
$$
(the powers $A^k$ are understood in the operations of a given
algebra). In  $(\max, +)$-algebra the elements of $A^{\star}$
clearly define the longest path on the graph associated with $A$
(see details e.g. in \cite{AGW} for finite and respectively infinite
space $X$), and its columns are natural candidates for the solution
of the eigenvalue - eigenvector equation for $A$ (equation of type
\eqref{eq1.9}). Hence the non-surprising appearance of $A^{\star}$
in our setting (our functions $\tilde b_{ij}$ and potentials
$\phi_i$ represent appropriately normalized elements of
$A^{\star}$). As was observed in \cite{Ru}, the functions $\tilde
b_{ij}$ turn out to be useful also in the analysis of the
Monge-Kantorovich mass transfer problem, which 
is a natural
analogue of the assignment problem for general measurable
(uncountable) state space $X$.

\setcounter{footnote}{0}
\setcounter{section}{0}
\nachaloe{Ali Baklouti}{Dequantization of coadjoint orbits: 
the case of exponential Lie groups}{Partially 
supported by the D.G.R.S.T Research Unity:00 UR 1501.}
\label{bak-abs}

It is well known that the unitary dual $\hat G$ of an exponential
solvable Lie group $G$ is homeomorphic to the space of coadjoint
orbits. Given a coadjoint orbit $\mathcal{O}$, the orbit method
enables us to construct the associated unitary and irreducible
representation. The inverse procedure is called
{\it Dequantization} and it consists in going backwards from an
\markboth{Viacheslav P.~Belavkin}{Dequantization of coadjoint orbits}
irreducible unitary representation $\pi:G\longrightarrow U(H)$ to
the coadjoint orbit $ \mathcal{O}$ and its associated geometric
objects. Here, $U(H)$ denotes the group of unitary operators on
some complex inner product space $H$. Towards dequantization, we
consider the Poisson characteristic variety of some
 topological unitary modules over a deformed algebra appropriately associated
 with the representation in question. In the case of nilpotent Lie
 groups, we showed that the
Poisson characteristic variety coincides with the associated
coadjoint orbit. For exponential Lie groups, we conjecture that such
a variety coincides with the Zariski closure of the orbit. In this
work, we prove such a conjecture for some restrictive classes of
exponential Lie groups.
\newpage
\setcounter{section}{0}
\setcounter{footnote}{0}
\nachalo{Viacheslav P.~Belavkin}{Quantum Pontryagin \mbox{principle} and
\mbox{quantum} \mbox{Hamilton-Jacobi-Bellman} \mbox{equation}:
a \mbox{max-plus point of view}}
\label{bel-abs}
We exploit the separation of the filtering and control aspects of quantum
feedback control to consider the optimal control as a classical stochastic
problem on the space of quantum states. We derive the corresponding
Hamilton-Jacobi-Bellman equations using the elementary arguments of
classical control theory and show that this is equivalent to a
Hamilton-Pontryagin setup. We show that, for cost functionals that are
linear in the state, the theory yields the traditional Bellman equations
treated so far in quantum feedback. A controlled qubit with a feedback is
considered as example.

\section{Introduction}

Quantum measurement, by its very nature, leads always to partial information
about a system in the sense that some quantities always remain uncertain,
and due to this the measurement typically alters the prior to a posterior
state in process. The Belavkin nondemolition principle \cite{Be80,Be87a}
states that this state reduction can be effectively treated within a
non-demolition scheme \cite{Be87a},\cite{B88} when measuring the system over
time. Hence we may apply a quantum filter for either discrete \cite{Be78} or
time-continuous \cite{Be80} non-demolition state estimation, and then
consider feedback control based on the results of this filtering. The
\markboth{Viacheslav P.~Belavkin}{Quantum Pontryagin principle}
general theory of continuous-time nondemolition estimation developed in \cite%
{B88},\cite{B90},\cite{B92} derives for quantum posterior states a
stochastic filtering evolution equation not only for diffusive but also for
counting measurements, however we will consider here the special case of
Belavkin quantum state filtering equation based on a diffusion model
described by a single white noise innovation, see e.g. \cite{B89}. Once the
filtered dynamics is known, the optimal feedback control of the system may
then be formulated as a distinct problem.

The separation of the classical world from the quantum world is, of course,
the most notoriously troublesome task faced in modern physics. At the very
heart of this issue is the very different meanings we attach to the word 
\textit{state}. What we want to exploit is the fact that the separation of
the control from the filtering problem gives us just the required separation
of classical from quantum features. By the quantum state we mean the von
Neumann density matrix which yields all the (stochastic) information
available about the system at the current time - this we also take to be the
state in the sense used in control engineering. All the quantum features are
contained in this state, and the filtering equation it satisfies may then to
be understood as classical stochastic differential equation which just
happens to have solutions that are von Neumann density matrix valued
stochastic processes. The ensuing problem of determining optimal control may
then be viewed as a classical problem, albeit on the unfamiliar state space
of von Neumann density matrices rather than the Euclidean spaces to which we
are usually accustomed. Once we get accustomed to this setting, the problem
of dynamical programming, Bellman's optimality principle, etc., can be
formulated in much the same spirit as before.

\section{Notations and Facts}

The Hilbert space for our fixed quantum system will be a complex, separable
Hilbert space $\mathfrak{h}$ . We shall use the following spaces of
operators: 
\begin{equation*}
\begin{array}{ll}
\mathcal{A}=\mathfrak{B}\left( \mathfrak{h}\right) & \text{- the Banach
algebra of bounded operators on }\mathfrak{h}; \\ 
\mathcal{A}_{\star }=\mathfrak{T}\left( \mathfrak{h}\right) & \text{- the
predual space of trace-class operators on }\mathfrak{h}; \\ 
\mathcal{S}=\mathfrak{S}\left( \mathfrak{h}\right) & \text{- the positive,
unital-trace operators (states) on }\mathfrak{h}; \\ 
\mathcal{T}_{0}=\mathfrak{T}_{0}\left( \mathfrak{h}\right) & \text{- the
tangent space of zero-trace operators on }\mathfrak{h}; \\ 
\mathcal{T}_{0}^{\star }=\mathfrak{T}_{0}\left( \mathfrak{h}\right) ^{\star }
& \text{- the cotangent space (see below)}.%
\end{array}%
\end{equation*}
The space $\mathcal{A}_{\star }$ equipped\ with the trace norm $\left\|
\varrho \right\| _{1}=\mathrm{tr}\left| \varrho \right| $ is a complex
Banach space, the dual of which is identified with the algebra $\mathcal{A}$
with usual operator norm. The natural duality between the spaces $\mathcal{A}%
_{\star }$ and $\mathcal{A}$ is indicated by 
\begin{equation}
\left\langle \varrho ,A\right\rangle :=\mathrm{tr}\left\{ \varrho A\right\} ,
\label{pairing}
\end{equation}
for each $\varrho \in \mathcal{A}_{\star },A\in \mathcal{A}$. The positive
elements, in the sense of positive definiteness $\varrho \geq 0$, form a
cone $\mathcal{T}_{+}$ of the real subspace $\mathcal{T}\subset \mathcal{A}%
_{\star }$ of all Hermitian elements $\varrho =\varrho ^{\dagger }$, and the
unit trace elements $\varrho \in \mathcal{T}_{+}$ normalized as $\left\|
\varrho \right\| _{1}=1$ are called normal states. Thus $\mathcal{S}=%
\mathcal{T}_{+}\cap \mathcal{T}_{1}$, where $\mathcal{T}_{1}=\left\{ \tau
\in \mathcal{T}:\mathrm{tr\,}\tau =1\right\} $, and the extremal elements $%
\varrho \in \mathcal{S}$ of the convex set $\mathcal{S}\subset \mathcal{T}%
_{+}$ correspond to pure quantum states. Every state $\varrho $ can be
parametrized as $\varrho \left( q\right) =\varrho _{0}-q$ by a \textit{%
tangent element} $q\in \mathcal{T}_{0}$ with respect to a given state $%
\varrho _{0}\in \mathcal{S}$. We may use the duality (\ref{pairing}) to
introduce \textit{cotangent elements} $p\in \mathcal{T}_{0}^{\star }$.
Knowledge of $\left\langle q,p\right\rangle $ for each $q\in \mathcal{T}_{0}$
will only serve to determine $p\in \mathcal{A}$ up to an additive constant
(as the $q$'s are trace-free): for this reason we should think of cotangents
elements $p$ as equivalence classes 
\begin{equation}
p\left[ X\right] =\left\{ A\in \mathcal{A}:A=X+\lambda I\text{, for some }%
\lambda \in \mathbb{R}\right\} .
\end{equation}

The symmetric tensor power $\mathcal{A}_{sym}^{\otimes 2}=\mathcal{A}\otimes
_{sym}\mathcal{A}$ of the algebra $\mathcal{A}$ is the subalgebra of $%
\mathfrak{B}\left( \mathfrak{h}^{\otimes 2}\right) $ of all bounded
operators on the Hilbert product space $\mathfrak{h}^{\otimes 2}=\mathfrak{h}%
\otimes \mathfrak{h}$, commuting with the unitary involutive operator $%
S=S^{\dagger }$ of permutations $\eta _{1}\otimes \eta _{2}\mapsto \eta
_{2}\otimes \eta _{1}$ for any $\eta _{i}\in \mathfrak{h}$.

A map $\mathfrak{L}\left( t,\cdot \right) $ from $\mathcal{A}=\mathfrak{B}%
\left( \mathfrak{h}\right) $ to itself is said to be a Lindblad generator if
it takes the form 
\begin{eqnarray}
\mathfrak{L}\left( t,X\right) &=&i\left[ H\left( t\right) ,X\right]
+\sum_{\alpha }\mathfrak{L}_{R_{\alpha }}\left( X\right) ,  \label{generator}
\\
\mathfrak{L}_{R}\left( X\right) &=&R^{\dagger }XR-\frac{1}{2}R^{\dagger }RX-%
\frac{1}{2}XR^{\dagger }R  \label{lindblad}
\end{eqnarray}%
with $H$ self adjoint, the $R_{\alpha }\in \mathcal{A}$ (and the summations
in (3) understood to be ultraweakly convergent [27] for an infinite set $%
\left\{ R_{\alpha }\right\} $). The generator is Hamiltonian if it just
takes form $i\left[ H\left( t\right) ,\cdot \right] $. The pre-adjoint $%
\mathfrak{L}^{\prime }=\mathfrak{L}_{\star }$ of a generator $\mathfrak{L}$
is defined on the pre-adjoint space $\mathcal{A}_{\star }$ through the
relation $\left\langle \mathfrak{L}^{\prime }\left( \varrho \right)
,X\right\rangle =\left\langle \varrho ,\mathfrak{L}\left( X\right)
\right\rangle $. We note that Lindblad generators have the property $%
\mathfrak{L}\left( I\right) =0$ corresponding \ to conservation of the
identity operator $I\in \mathcal{A}$ or, equivalently, $\mathrm{tr}\left\{ 
\mathfrak{L}^{\prime }\left( \varrho \right) \right\} =0$ for all $\varrho
\in \mathcal{A}_{\star }$.

In quantum control theory it is necessary to consider time-dependent
generators $\mathfrak{L}\left( t\right) $, through an integrable time
dependence of the controlled Hamiltonian $H\left( t\right) $, and, more
generally, due to a square-integrable time dependence of the coupling
operators $R_{\alpha }\left( t\right) $. We shall always assume that these
integrability conditions, ensuring existence and uniqueness of the solution $%
\varrho \left( t\right) $ to the quantum state Master equation 
\begin{equation}
\frac{d}{dt}\varrho \left( t\right) =\mathfrak{L}^{\prime }\left( t,\varrho
\left( t\right) \right) \equiv \vartheta \left( t,\varrho \left( t\right)
\right) ,  \label{master equation}
\end{equation}
for all for $t\geq t_{0}$ with given initial condition $\varrho \left(
t_{0}\right) =\varrho _{0}\in \mathcal{S}$, are fulfilled.

Let $\mathsf{F}=\mathsf{F}\left[ \cdot \right] $ be a (nonlinear) functional 
$\varrho \mapsto \mathsf{F}\left[ \varrho \right] $ on $\mathcal{A}_{\star }$
(or on $\mathcal{S}\subset \mathcal{A}_{\star }$), then we say it admits a
(Frechet) derivative if there exists an $\mathcal{A}$-valued function $%
\nabla _{\varrho }\mathsf{F}\left[ \cdot \right] $ on $\mathcal{A}_{\star }$
($\mathcal{T}_{0}^{\star }$-valued functional on $\mathcal{T}_{0}$) such
that 
\begin{equation}
\lim_{h\rightarrow 0}\frac{1}{h}\left\{ \mathsf{F}\left[ \cdot +h\tau \right]
-\mathsf{F}\left[ \cdot \right] \right\} =\left\langle \tau ,\nabla
_{\varrho }\mathsf{F}\left[ \cdot \right] \right\rangle ,  \label{nabla_Q}
\end{equation}%
for each $\tau \in \mathcal{A}_{\star }$ (for each $\tau \in \mathcal{T}_{0}$%
). In the same spirit, a Hessian $\nabla _{\varrho }^{\otimes 2}\equiv
\nabla _{\varrho }\otimes \nabla _{\varrho }$ can be defined as a mapping
from the functionals on $\mathcal{S}$ to the $\mathcal{A}_{sym}^{\otimes 2}$%
-valued functionals, via 
\begin{gather}
\lim_{h,h^{\prime }\rightarrow 0}\frac{1}{hh^{\prime }}\left\{ \mathsf{F}%
\left[ \cdot +h\tau +h^{\prime }\tau ^{\prime }\right] -\mathsf{F}\left[
\cdot +h\tau \right] -\mathsf{F}\left[ \cdot +h^{\prime }\tau ^{\prime }%
\right] +\mathsf{F}\left[ \cdot \right] \right\}  \notag \\
=\left\langle \tau \otimes \tau ^{\prime },\nabla _{\varrho }\otimes \nabla
_{\varrho }\mathsf{F}\left[ \cdot \right] \right\rangle .
\end{gather}%
and we say that the functional is twice continuously differentiable whenever 
$\nabla _{\varrho }^{\otimes 2}\mathsf{F}\left[ \cdot \right] $ exists and
is continuous in the trace norm topology.

Likewise, a functional $f:X\mapsto f\left[ X\right] $ on $\mathcal{A}$ is
said to admit an $\mathcal{A}_{\star }$-derivative if there exists an $%
\mathcal{A}_{\star }$-valued function $\nabla _{X}f\left[ \cdot \right] $ on 
$\mathcal{A}$ such that 
\begin{equation}
\lim_{h\rightarrow 0}\frac{1}{h}\left\{ f\left[ \cdot +hA\right] -f\left[
\cdot \right] \right\} =\left\langle \nabla _{X}f\left[ \cdot \right]
,A\right\rangle
\end{equation}
for each $A\in \mathfrak{B}\left( \mathfrak{h}\right) $. The derivative $%
\nabla _{X}f\left[ \cdot \right] $ has zero trace, $\nabla _{X}f\left[ A%
\right] \in \mathcal{T}_{0}$ for each $A\in \mathcal{A}$, if and only if the
functional $f\left[ X-\lambda I\right] $ does not depend on $\lambda $, i.e.
is essentially a function $f\left( p\right) $ of the class $p\left[ X\right]
\in \mathcal{T}_{0}^{\star }$.

With the customary abuses of differential notation, we have for instance 
\begin{equation*}
\nabla _{\varrho }f\left( \left\langle \varrho ,X\right\rangle \right)
=f^{\prime }\left( \left\langle \varrho ,X\right\rangle \right) X,\quad
\nabla _{X}f\left( \left\langle \varrho ,X\right\rangle \right) =f^{\prime
}\left( \left\langle \varrho ,X\right\rangle \right) \varrho ,
\end{equation*}%
for any differentiable function $f$ of the scalar $x=\left\langle \varrho
,X\right\rangle $. Typically, we shall use $\nabla _{\varrho }$ more often,
and denote it by just $\nabla $.

\section{Quantum Optimal Control}

From now on we will assume that the Hamiltonian $H$ and therefore $\vartheta 
$ (and $\upsilon $) are functions of a controlled parameter $u\in \mathcal{U}
$ depending on $t$ such that the time dependence of the generator $\mathfrak{%
L}$ is of the form $\mathfrak{L}\left( u\left( t\right) \right) $. Moreover,
we do not require at this stage the linearity of $\vartheta \left( u,\varrho
\right) $ with respect to $\varrho $, as well as the quadratic dependence $%
\sigma \left( \varrho \right) $, which means that what follows below is also
applicable to more general quantum stochastic kinetic equations 
\begin{equation*}
d\varrho _{\bullet }\left( t\right) =\vartheta \left( u\left( t\right)
,\varrho _{\bullet }\left( t\right) \right) \,dt+\sigma \left( \varrho
_{\bullet }\left( t\right) \right) \,dw\left( t\right)
\end{equation*}%
of Vlassov and Boltzmann type, with only the positivity and trace
preservation requirements $\mathrm{tr}\left\{ \vartheta \left( u,\varrho
\right) \right\} =0=\mathrm{tr}\left\{ \sigma \left( \varrho \right)
\right\} $. A choice of the control function $\left\{ u\left( r\right) :r\in %
\left[ t_{0},t\right] \right\} $ is required before we can solve the
filtering equation 
(Belavkin equation)
at the time $t$ for a given
initial state $\varrho _{0}$ at time $t_{0}$. From what we have said above,
this is required to be a $\mathcal{U}$-valued function which we take to be
continuous for the moment.

The cost for a control function $\left\{ u\left( r\right) \right\} $ over
any time-interval $\left[ t,T\right] $ is random and taken to have the
integral form 
\begin{equation}
\mathsf{J}_{\omega }\left[ \left\{ u\left( r\right) \right\} ;t,\varrho %
\right] =\int_{t}^{T}\mathsf{C}\left( u\left( r\right) ,\varrho _{\omega
}\left( r\right) \right) dr+\mathsf{G}\left( \varrho _{\omega }\left(
T\right) \right)  \label{J}
\end{equation}
where $\left\{ \varrho _{\bullet }\left( r\right) :r\in \left[ t,T\right]
\right\} $ is the solution to the filtering equation with initial condition $%
\varrho _{\bullet }\left( t\right) =\varrho $. We assume that the \textit{%
cost density} $\mathsf{C}$ and the \textit{terminal cost, or bequest
function,} $\mathsf{G}$ will be continuously differentiable in each of its
arguments. In fact, due to the statistical interpretation of quantum states,
we should consider only the linear dependence 
\begin{equation}
\mathsf{C}\left( u,\varrho \right) =\left\langle \varrho ,C\left( u\right)
\right\rangle ,\;\mathsf{G}\left( \varrho \right) =\left\langle \varrho
,G\right\rangle  \label{costs}
\end{equation}
of $\mathsf{C}$ and $\mathsf{G}$ on the state $\varrho $ as it was already
suggested in \cite{B83},\cite{B88},\cite{B99}. We will explicitly consider
this case later, but for the moment we will not use the linearity of $%
\mathsf{C}$ and $\mathsf{G}$. We refer to $C\left( u\right) \in \mathcal{A}$
as \textit{cost observable} for $u\in \mathcal{U}$ and $G\in \mathcal{A}$ as
the \textit{bequest observable}.

The feedback control $u\left( t\right) $ is to be considered a random
variable $u_{\omega }\left( t\right) $ adapted with respect to the
innovation process $w\left( t\right) $, in line with our causality
requirement, and so we therefore consider the problem of minimizing its
average cost value with respect to $\left\{ u_{\bullet }\left( t\right)
\right\} $. To this end, we define the optimal average cost on the interval $%
\left[ t,T\right] $ to be 
\begin{equation}
\mathsf{S}\left( t,\varrho \right) :=\inf_{\left\{ u_{\bullet }\left(
r\right) \right\} }\,\mathbb{E}\left[ \mathsf{J}_{\bullet }\left[ \left\{
u_{\bullet }\left( r\right) \right\} ;t,\varrho \right] \right] ,  \label{S}
\end{equation}%
where the minimum is considered over all measurable adapted control
strategies $\left\{ u_{\bullet }\left( r\right) :r\geq t\right\} $. The aim
of feedback control theory is then to find an optimal control strategy $%
\left\{ u_{\bullet }^{\ast }\left( t\right) \right\} $ and evaluate $\mathsf{%
S}\left( t,\varrho \right) $ on a fixed time interval $\left[ t_{0},T\right] 
$. Obviously that the cost $\mathsf{S}\left( t,\varrho \right) $ of the
optimal feedback control is in general smaller than the minimum of $\,%
\mathbb{E}\left[ \mathsf{J}_{\bullet }\left[ \left\{ u\right\} ;t,\varrho %
\right] \right] $ over nonstochastic strategies $\left\{ u\left( r\right)
\right\} $ only, which gives the solution of the open loop (without
feedback) quantum control problem. In the case of the linear costs (\ref%
{costs}) this open-loop problem is equivalent to the following quantum
deterministic optimization problem which can be tackled by the classical
theory of optimal deterministic control in the corresponding Banach spaces.

\subsection{Bellman \& Hamilton-Pontryagin Optimality}

Let us first consider nonstochastic quantum optimal control theory assuming
that the state $\varrho \left( t\right) \in \mathcal{S}$ obeys the master
equation (\ref{master equation}) where $\vartheta \left( u,\varrho \right) $
is the adjoint $\mathfrak{L}^{\prime }\left( u\right) $ of some Lindblad
generator for each $u$ with, say, the control being exercised in the
Hamiltonian component $i\left[ \cdot ,H\left( u\right) \right] $ as before.
(More generally, we could equally well consider a nonlinear quantum kinetic
equation.) The control strategy $\left\{ u\left( t\right) \right\} $ will be
here non-random, as will be any specific cost $\mathsf{J}\left[ \left\{
u\right\} ;t_{0},\varrho _{0}\right] $. As for $\mathsf{S}\left( t,\varrho
\right) =\inf \mathsf{J}\left[ \left\{ u\right\} ;t,\varrho \right] $ at the
times $t<t+\varepsilon <T$, one has 
\begin{align*}
\mathsf{S}&\left(t,\varrho \right) =\\
&\inf_{\left\{ u\right\} }\,\left\{
\int_{t}^{t+\varepsilon }\mathsf{C}\left( u\left( r\right) ,\varrho \left(
r\right) \right) dr+\int_{t+\varepsilon }^{T}\mathsf{C}\left( u\left(
r\right) ,\varrho \left( r\right) \right) dr+\mathsf{G}\left( \varrho \left(
T\right) \right) \right\} .
\end{align*}

Suppose that $\left\{ u^{\ast }\left( r\right) :r\in \left[ t,T\right]
\right\} $ is an optimal control when starting in state $\varrho $ at time $t
$, and denote by $\left\{ \varrho ^{\ast }\left( r\right) :r\in \left[ t,T%
\right] \right\} $ the corresponding state trajectory starting at state $%
\varrho $ at time $t$. Bellman's optimality principle observes that the
control $\left\{ u^{\ast }\left( r\right) :r\in \left[ t+\varepsilon ,T%
\right] \right\} $ will then be optimal when starting from $\varrho ^{\ast
}\left( t+\varepsilon \right) $ at the later time $t+\varepsilon $. It
therefore follows that 
\begin{equation*}
\mathsf{S}\left( t,\varrho \right) =\inf_{\left\{ u\left( r\right) \right\}
}\,\left\{ \int_{t}^{t+\varepsilon }\mathsf{C}\left( u\left( r\right)
,\varrho \left( r\right) \right) dr+\mathsf{S}\left( t+\varepsilon ,\varrho
\left( t+\varepsilon \right) \right) \right\} .
\end{equation*}%
For $\varepsilon $ small we expect that $\varrho \left( t+\varepsilon
\right) =\varrho +\vartheta \left( u\left( t\right) ,\varrho \right)
\varepsilon +o\left( \varepsilon \right) $ and provided that $\mathsf{S}$ is
sufficiently smooth we may make the Taylor expansion 
\begin{equation}
\mathsf{S}\left( t+\varepsilon ,\varrho \left( t+\varepsilon \right) \right)
=\left[ 1+\varepsilon \frac{\partial }{\partial t}+\varepsilon \left\langle
\vartheta \left( u\left( t\right) ,\varrho \right) ,\nabla \right\rangle %
\right] \mathsf{S}\left( t,\varrho \right) +o\left( \varepsilon \right) .
\label{S1}
\end{equation}%
In addition, we approximate 
\begin{equation*}
\int_{t}^{t+\varepsilon }\mathsf{C}\left( u\left( r\right) ,\varrho \left(
r\right) \right) dr=\varepsilon \mathsf{C}\left( u\left( t\right) ,\varrho
\right) +o\left( \varepsilon \right) 
\end{equation*}%
and conclude that (note the convective derivative!) 
\begin{equation*}
\mathsf{S}\left( t,\varrho \right) =\inf_{u\in U}\,\left\{ \left[
1+\varepsilon \left( \mathsf{C}\left( u,\varrho \right) +\frac{\partial }{%
\partial t}+\left\langle \vartheta \left( u,\varrho \right) ,\nabla \cdot
\right\rangle \right) \right] \mathsf{S}\left( t,\varrho \right) \right\}
+o\left( \varepsilon \right) 
\end{equation*}%
where now the infimum is taken over the point-value of $u\left( t\right)
=u\in U$. In the limit $\varepsilon \rightarrow 0$, one obtains the equation 
\begin{equation}
-\frac{\partial }{\partial t}\mathsf{S}\left( t,\varrho \right) =\inf_{u\in 
\mathcal{U}}\left\{ \mathsf{C}\left( u,\varrho \right) +\left\langle
\vartheta \left( u,\varrho \right) ,\nabla \mathsf{S}\left( t,\varrho
\right) \right\rangle \right\} ,  \label{HJB1}
\end{equation}%
where $\nabla =\nabla _{\varrho }$. The equation is then to be solved
subject to the terminal condition 
\begin{equation}
\mathsf{S}\left( T,\varrho \right) =\mathsf{G}\left( \varrho \right) .
\end{equation}

We may introduce the Pontryagin Hamiltonian function on $\mathcal{T}%
_{0}\times \mathcal{T}_{0}^{\star }$ defined by the Legendre-Fenchel
transform 
\begin{equation}
\mathcal{H}_{\vartheta }\left( q,p\left[ X\right] \right) :=\sup_{u\in 
\mathcal{U}}\left\{ \left\langle \vartheta \left( u,\varrho \left( q\right)
\right) ,\lambda I-X\right\rangle -\mathsf{C}\left( u,\varrho \left(
q\right) \right) \right\} .  \label{hamiltonian}
\end{equation}
Here we use a parametrization $\varrho \left( q\right) =\varrho _{0}-q$, $%
q\in \mathcal{T}_{0}$ and the fact that the supremum does not depend on $%
\lambda \in \mathbb{R}$ since $\left\langle \vartheta \left( u,\varrho
\right) ,I\right\rangle =0$. Therefore $\mathcal{H}$ depends on $X$ only
through the equivalence class $p\left[ X\right] \in \mathcal{T}_{0}^{\star }$
which is referred to as the \textit{co-state}. It should be emphasized that
these Hamiltonians are purely classical devices which may be called
super-Hamiltonians to be distinguished from $H$. We may then rewrite (\ref%
{HJB1}) as the (backward) \textit{Hamilton-Jacobi} equation 
\begin{equation}
-\frac{\partial }{\partial t}\mathsf{S}\left( t,\varrho \left( q\right)
\right) +\mathcal{H}_{\vartheta }\left( q,p\left[ \nabla \mathsf{S}\left(
t,\varrho \right) \right] \left( q\right) \right) =0.  \label{HJB2}
\end{equation}

Applying the derivative $\nabla _{q}=-\nabla $ to this equation to $%
q=\varrho _{0}-\varrho $ in the tangent space $\mathcal{T}_{0}$ we obtain
the dynamical equation $\dot{p}=-\nabla _{q}\mathcal{H}_{\vartheta }\left(
q,p\right) $ for the co-state $p_{t}=Q_{t}\left( T,s\right) $ of the
operator--valued function $X\left( t,\varrho \right) =\nabla \mathsf{S}%
\left( t,\varrho \right) $, where $Q_{t}\left( T,s\right) =p\left[ X\left(
t,\varrho \right) \right] $ is the solution of this equation satisfying the
terminal condition $Q_{T}\left( T,s\right) =s:=p\left[ G\right] $ with $p%
\left[ G\left( \varrho \right) \right] \left( q\right) =p\left[ G\left(
\varrho \left( q\right) \right) \right] $ for $G=\nabla \mathsf{G}$. We
remark that, if $u^{\ast }\left( q,p\left( X\right) \right) $ is an optimal
control maximizing 
\begin{equation*}
\mathcal{K}_{\vartheta }\left( u,q,p\left( X\right) \right) =\left\langle
\vartheta \left( u,\varrho \left( q\right) \right) ,\lambda I-X\right\rangle
-\mathsf{C}\left( u,\varrho \left( q\right) \right) \text{,}
\end{equation*}%
then the corresponding state dynamical equation $\frac{d}{dt}\varrho
=\vartheta \left( u^{\ast }\left( \varrho ,X\right) ,\varrho \right) $ in
terms of its optimal solution $q_{t}\equiv \varrho _{0}-\varrho _{t}\left(
t_{0},\varrho _{0}\right) $ corresponding to $\varrho _{t_{0}}\equiv \varrho
_{0}$ can be written as $\dot{q}=\nabla _{p}\mathcal{H}_{\vartheta }\left(
q,p\right) $, noting that 
\begin{equation}
\nabla _{p}\mathcal{H}_{\vartheta }\left( q,p\right) =\nabla _{p}\mathcal{K}%
_{\vartheta }\left( u^{\ast }\left( q,p\right) ,q,p\right) =-\vartheta
\left( u^{\ast }\left( q,p\right) ,\varrho \left( q\right) \right) 
\label{optim}
\end{equation}%
due to the stationarity condition $\frac{\partial }{\partial u}\mathcal{K}%
_{\vartheta }\left( u,q,p\right) =0$ at $u=u^{\ast }$. This forward equation
with $q_{0}=0$ for $\varrho ^{\ast }\left( t_{0}\right) =\varrho _{0}$
together with the co-state backward equation with $p_{T}=p\left[ G\right]
\equiv s$ is the canonical Hamiltonian system. Thus we may equivalently
consider the Hamiltonian boundary value problem 
\markboth{V.I.~Danilov, A.V.~Karzanov, G.A.~Koshevoy}{Quantum Pontryagin principle}
\begin{equation}
\left\{ 
\begin{array}{c}
\dot{q}_{t}-\nabla _{p}\mathcal{H}_{\vartheta }\left( q_{t},p_{t}\right)
=0,\;q_{0}=0 \\ 
\dot{p}_{t}+\nabla _{q}\mathcal{H}_{\vartheta }\left( q_{t},p_{t}\right)
=0,\;p_{T}=s%
\end{array}%
\right.   \label{Ham}
\end{equation}%
which we refer to as the \textit{Hamilton-Pontryagin problem}, in direct
analogy with the classical case. The solution to this problem defines the
minimal cost as the path integral 
\begin{equation*}
\mathsf{S}\left( t_{0},\varrho _{0}\right) =\int_{t_{0}}^{T}\left[
\left\langle \dot{q}_{r}|p_{r}\right\rangle -\mathcal{H}\left(
q_{r},p_{r}\right) \right] dr+\mathsf{G}\left( \varrho \left( q_{T}\right)
\right) .
\end{equation*}%
Thus the \textit{Pontryagin maximum principle} for the quantum dynamical
system is the observation that the optimal quantum control problem is
equivalent to the Hamiltonian problem for state and co-state $\left\{
q\right\} $ and $\left\{ p\right\} $ respectively, leading to optimality $%
\mathcal{K}_{\vartheta }\left( u,q,p\right) \leq \mathcal{H}_{\vartheta
}\left( q,p\right) $ with equality for $u=u^{\ast }\left( q,p\right) $
maximizing $\mathcal{K}_{\vartheta }\left( u,q,p\right) $.


\setcounter{section}{0}
\setcounter{footnote}{0}
\nachalo{\mbox{V.I.~Danilov}, \mbox{A.V.~Karzanov} and \mbox{G.A.~Koshevoy}}%
{Tropical Pl\"ucker functions}
\label{dan-abs}

\section{Introduction}
Totally positive matrices play an important role in different areas of
mathematics, from differential equations to combinatorics.
Studying parametrizations of canonical bases, 
Berenstein, Fomin and Zelevinsky  \cite{BFZ} established
the so-called Chamber ansatz for flag minors of matrices.
This result relies on Pl\"ucker relations between flag minors.
In this talk, we study functions which satisfy tropical Pl\"ucker
relations. We consider two approaches to tropicalization of
Pl\"ucker relations.
One approach is based on tropicalization of the so-called special
Pl\"ucker relation
(3-term relation) after writing it as a subtraction-free expression.
On this way
we get the class of TP-functions (on Boolean cube  $2^N$)
which can be seen as the tropicalization
of flag minors. We show that such functions are
determined by their restrictions
to the interval family of subsets of $N$, and that the class of submodular
TP-functions coincides with the class of submodular functions on the interval
family. This might be seen as the tropicalization of the Chamber ansatz.
Our proof is based on a construction of DMTP-functions via normal flows
in weighted digraphs. A 
DMTP-function
is a function that satisfies
tropical Pl\"ucker relations, and at this point we consider the tropicalization in
the sense of the second approach, meaning that we tropicalize
an algebraic formula with
subtractions in the right hand side as an inequality.

\section{Flag minors and Pl\"ucker relations}
\markboth{V.I.~Danilov, A.V.~Karzanov, G.A.~Koshevoy}{Tropical Pl\"ucker functions}
For $X$ an $n\times n$ matrix and $I$ a subset of $N=\{1,\ldots,n\}$,
denote by  $\Delta_I(X)$ the determinant of the submatrix located in the
intersection of the first $|I|$ rows and the columns indexed by
$I$. These determinants are called {\em flag minors} of $X$.
It is known (see Fulton and Harris \cite{Fulton}, p.235), that
flag minors of a matrix $X$ satisfy the Pl\"ucker relations
\[
\Delta_{A\cup I}\Delta_{A\cup J}= \sum_{i\in I}(-1)^{d(I,J,i,j)}
\Delta_{A\cup (I\setminus i)\cup j}\Delta_{A\cup (J\setminus
j)\cup i},
\]
for any pairwise disjoint $A, I, J\subset N$, $|I|\ge |J|$, any
fixed $j\in J$, and an appropriate function $d(I,J,i,j)$.

The Pl\"ucker relations with $2\ge |I|\ge
|J|\ge 1$ are of special interest.  They are given by
\[
\Delta_{A\cup ik}\Delta_{A\cup j}=\Delta_{A\cup ij}\Delta_{A\cup
k}+\Delta _{A\cup jk}\Delta _{A\cup i}
\]
with any $A$ and $i<j<k$, and
\[
\Delta _{A\cup ik}\Delta _{A\cup jl}=\Delta_{A\cup ij}\Delta
_{A\cup kl}+\Delta _{A\cup jk}\Delta_{A\cup il}
\]
with any $A$ and $i<j<k<l$.\medskip

\begin{definition} {\rm A function $F:2^N\to\mathbb R$ is said to be
{\em P-function} if $F$ satisfies the above 
Pl\"ucker relations for all
$A$ and $i<j<k$, that is}
\begin{equation}\label{3-pluker}
F(A\cup ik)F (A\cup j)=F (A\cup ij)F(A\cup k)+F (A\cup jk)F (A\cup
i).
\end{equation}
{\rm Denote by $\mathcal{PF}$ the set of $P$-functions.}
\end{definition}

We collected properties of $P$-functions in the following\medskip

\begin{theorem} 
\begin{itemize}
\item[a)] Any P-function satisfies all Pl\"ucker
relations.
\item[b)] Let $F:2^N\to\mathbb R_{\neq 0}$ be a P-function. Then there
exists a unique upper-triangular $N\times N$ matrix $X$ such that
$F(I)=\Delta_I(X)$.
\item[c)]  Let $\mathcal I\subset 2^N$ denote
the interval family constituted 
from intervals $\{i,i+1,\ldots, j\}$,
$i\le j$, and let $\operatorname{res}_I:\mathbb R^{2^N}\to
\mathbb R^{\mathcal I}$ denote the  restriction map from $2^N$ to
$\mathcal I$. Then the mapping $\operatorname{res}_I$ is a bijection between
the subspace $ PF$ ($\subset \mathbb R^{2^N}$) and $\mathbb
R^{\mathcal I}$. 
\end{itemize}
\end{theorem}

\section{Tropical Pl\"ucker functions}

We consider two ways of tropicalization 
of $P$-functions. 
Firstly, we can tropicalize (\ref{3-pluker}) 
as follows:
\begin{equation}\label{trop3-pluker}
f(A\cup ik)+f (A\cup j)=\max (f (A\cup ij)+f(A\cup k), f (A\cup
jk)+f (A\cup i)),
\end{equation}
with any $A$ and $i<j<k$.

\begin{definition} {\rm A function 
$f:2^N\to\mathbb R$ is said to be a
{\em TP-function} if $F$ satisfies 
the tropicalization
(\ref{trop3-pluker}) 
of Pl\"ucker relations for all $A$ and 
$i<j<k$.}
\end{definition}

Secondly, we can think of tropicalization 
in the form of inequality:
\begin{equation}\label{trop-pluker}
f({A\cup I})+f({A\cup J})\le\max_{i\in I} (f({A\cup (I\setminus
i)\cup j})+f({A\cup (J\setminus j)\cup i})),
\end{equation}
for any pairwise disjoint $A$, $I$, $J$, $|I|\ge |J|$, and any
fixed $j\in J$.

Specializing this to tropicalization of the 3-term
Pl\"ucker relation given by (\ref{3-pluker}), we 
have that, for all disjoint $A\subset N$ 
and $\{i,j,k\}\subset N$, 

{\em the maximum is attained 
at least twice among the three values}
\begin{equation}\label{3-double}
a=f(A\cup ik)+f (A\cup j),\, b=f (A\cup ij)+f(A\cup k),\, c= f
(A\cup jk)+f (A\cup i).
\end{equation}

Similarly, for a 4-term Pl\"ucker relation,
that is, for all disjoint $A\subset N$ and
$\{i,j,k,l\}\subset N$, we have that

{\em the maximum is attained 
at least twice among the three values}
\begin{equation}\label{4-double}
x=f(A\cup ik)+f (A\cup jl),\, y=f (A\cup ij)+f(A\cup kl),\, z= f
(A\cup jk)+f (A\cup il).
\end{equation}

\begin{definition}{\rm A function $f:2^N\to \mathbb R$ is 
called a
{\em DMTP-function} if $f$ satisfies 
(\ref{3-double}) and
(\ref{4-double}).}
\end{definition}

We will show in the following section 
that the class of
TP-functions is a subclass of DMTP-functions.

We have the following property of DMTP-functions, 
which is closely
related to valuated matroids, see Dress and 
Wenzel \cite{DW}.\medskip

\begin{theorem}  A DMTP-function $f$ satisfies 
all tropical
Pl\"ucker relations (\ref{trop-pluker}).
\end{theorem}

\section{Flows in digraphs and DMTP-functions}
Here we propose a construction of DMTP-functions.

We deal with a digraph $G=(V,E)$, 
a function $c:E\to\mathbb R$ of
weights on the edges, 
and disjoint subsets $S,T\subset V$. We
assume that $|S|=|T|$ and that $T$ is ordered:
$T=(t_1,t_2,\ldots,t_{|T|})$, and denote $\{t_1,\ldots,t_i\}$ by
$T_i$.

For $F\subseteq E$ and $X\subseteq V$, denote:

the sets of edges in $F$ leaving $X$ and entering $X$ by
$\delta^+_F(X)$ and by $\delta^-_F(X)$, respectively;

the set $\delta^+_F(X)\cup\delta^-_F(X)$ by $\delta_F(X)$;

the number $|\delta^+_F(X)|-|\delta^-_F(X)|$ by
$\diver_F(X)$.\medskip
\markboth{N.~Farhi, M.~Goursat, J.-P.~Quadrat}{Tropical Pl\"ucker functions}

\begin{definition} {\rm Let us say that $F\subset E$ is a {\em normal
flow} from $S'\subseteq S$ if}
  $$
\diver_F(v)=\left\{
        \begin{array}{rl}
     1 & \forall \ v\in S',\\
     -1 & \forall \ v\in T_{|S'|},\\
     0 & \mbox{for the other $v\in V$,}
   \end{array}
       \right.
   $$
{\rm where $\diver(v)$ stands for $\diver(\{v\})$.}
\end{definition}

This gives rise to the following important function $f=f_c$ on
$2^S$:
  \begin{equation}  \label{eq:f}
  f(S'):=\max\{c(F), \ F \mbox{ is a normal flow from}\;\; S'\}, \quad S'\subseteq S.
  \end{equation}

\begin{theorem} 
$f$ defined in (\ref{eq:f}) is a
DMTP-function.
\end{theorem}

The claim that TP-functions constitute a subslass of
DMTP-functions can be obtained as a 
consequence of the
following

\begin{theorem}  
Let $f:2^N\to\mathbb R$ be a TP-function.
Then there exists a planar digraph 
$G=(V,E)$ and a weight function
$c:E\to\mathbb R$, such that $f$ is 
defined by (\ref{eq:f}).
\end{theorem}

\begin{remark}
For a TP-function on the Boolean $2^N$, 
we can take 
the unique planar of the following from: $V:=\{(i,j)\,1\le i,j\le
n\}$, and two edges emanate from the vertex $(i,j)$, which
terminate either in $(i-1,j)$ or in $(i,j+1)$, respectively, if
both terminate points are vertices.
\end{remark}

The following property is important, see Kamnitzer \cite{Kamn}, for applications to crystal bases
construction via MV-polytopes:

\begin{theorem}  A TP-function $f:2^N\to\mathbb R$ is
submodular (that is $f(A)+f(B)\ge f(A\cup B)+f(A\cap B)$, $A$,
$B\subset N$) if and only if the function $\operatorname{res}_{\mathcal I}(f)$ is
submodular on $\mathcal I$.
\end{theorem}

\setcounter{section}{0}
\setcounter{footnote}{0}
\nachaloe{N. Farhi, M. Goursat, and J.-P. Quadrat}{Degree one homogeneous
\mbox{minplus} \mbox{dynamic} \mbox{systems}
and \mbox{traffic applications: Part I}}%
{Partially supported by the joint RFBR/CNRS  grant No. 05-01-02807.}
\label{far-abs}
We show that car traffic on a town can be modeled using
a Petri net extension where arcs have negative weights.
The corresponding minplus dynamics is not linear but 
homogeneous
of degree one. Possibly depending on the initial condition, 
homogeneous of degree 1 
minplus systems may be periodic or have
a chaotic behavior (to which corresponds a constant throughput) 
or may explode
exponentially.
In traffic systems, when this constant throughput exists it 
has the interpretation of the average car speed. 
In this first part we recall the derivation
of the 1-homogeneous dynamics of traffic system 
and show the existence of such systems with chaotic behavior and a 
constant throughput.

\section{Introduction}
At macroscopical level, the traffic on a 
road has been studied from different points of view, 
for example~:
\begin{itemize} 
\item The Lighthill-Whitham-Richards Model~\cite{LW} 
is the more standard
one
$$
\begin{cases}
\partial_t \rho+\partial_x q=0\;, \\
q=f(\rho),
\end{cases}
$$
where $q(x,t)=$ denotes the flow at time $t$ and 
position $x$ on the road, 
$\rho(x,t)=$ denotes density, 
$f$ is a given function, called the
fundamental traffic law. It plays for traffic the role of the perfect gas law for the fluid dynamics.
\item The kinetic model (Prigogine-Herman~\cite{PRI}) gives the evolution
of the density of particles $\rho(t,x,v)$ as a function of $t,x$ and $v$
the speed of particle
$$\partial_t \rho+v\partial_x \rho=C(\rho,\rho)\; ,$$
where $C(\rho,\rho)$ is an interacting term in general quadratic in $\rho$.
\end{itemize}

The second model is more costly in term of computation time and 
therefore not
used in practice. The first one supposes 
the knowledge of the function $f$.
This function usually comes from  experimental studies,
or from  theoretical studies using
simple microscopic model. Here, we will recall a way to derive a 
good approximation
of this law from a simple minplus linear system based on a Petri net.

The main purpose of this paper is to generalize this fundamental law
to the 2D cases where roads have crossings. The original minplus 
linear model on
a unique road cannot be generalized easily in term of Petri nets. 
We have
proposed in a previous paper a way to solve the difficulty 
by using Petri nets
with negative weights. 
The dynamics of these Petri nets can be written easily.
Being not  linear in minplus algebra any more, 
they are homogeneous of degree 1. 
We recall here the derivation of these 1-homogeneous dynamics.

In the first part of this paper we show that we can compute the eigenvalues for these 1-homogeneous system but that chaotic dynamics may appear. In the second
part we discuss the phases appearing in the fundamental diagram,
obtained numerically, and describe new situations where we can prove that the system is periodic.

\section{Traffic on a circular road}
Let us recall the simplest model to derive the fundamental traffic law on
a single road. The simplest way is to study the 
stationary regime on a circular
road with a given number of vehicles and then 
to consider that this stationary regime is reached locally 
when the density is given on a standard road. We present two way 
to obtain this law~: --  by logical deduction from
an exclusion process,  -- by computing 
the eigenvalue of a minplus system derived
from a simple Petri net describing the road with the vehicles.
\dessin{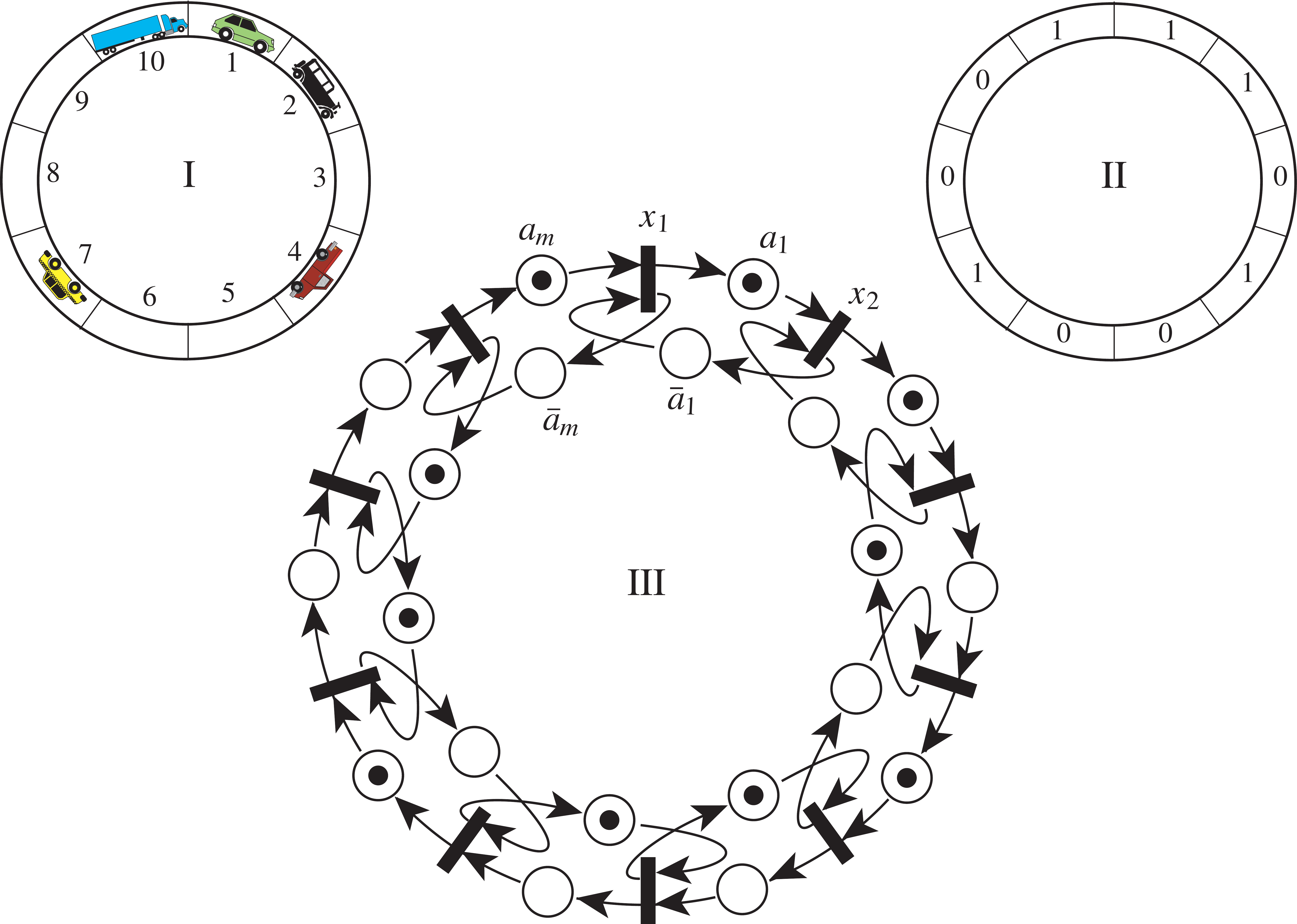}{0.13}{A circular road.}{ruta}
\subsection{Exclusion process modeling}
Following~\cite{BLA} we can consider the dynamic system defined by the rule $10\rightarrow 01$ apply to a binary word describing the car positions  on
a road cut in section (each bit representing a section 1 meaning occupied
and 0 meaning free see II in Figure~\ref{ruta}). Let us take an example~:
\begin{align*}
m_{1} &=1101001001,\quad  m_{2}=1010100101, \quad
m_{3} &=0101010011, \\
\quad m_{4} &=1010101010, \quad m_{5}=0101010101, \\
\end{align*} 

Let us define~:
-- the \emph{density} $\rho$ to be the  number  of  vehicles $n$ divided by 
number of  places $m$~: $\rho=n/m$, 
-- the  flow  $q(t)$ at time  
$t$ to be the number of vehicles going one step forward at time $t$ 
divided by the number of places. 
Then the  \emph{fundamental traffic law} gives the relation between
$q(t)$ and $d$.

If $\rho \leq 1/2$ then, after a transient period, 
all the  vehicle groups split off 
and then all the vehicles can move 
forward without other vehicles in the way, 
and we have~:
  $$q(t)=q=n/m=d\;.$$
If $\rho \geq  1/2$ then the free place groups split off
after a finite time and move backward without other free place in the way.
Then $m-n$ vehicles move forward  and we have
$$q(t)=q=(m-n)/m=1-d\;.$$
Therefore : $$\exists T:\;\forall t\geq T\quad
q(t)=q=\begin{cases}
\rho & \mathrm{if} \; \;\rho\leq 1/2\;, \\
1-\rho & \mathrm{if} \; \;\rho\geq 1/2\; .
\end{cases}
$$
\dessin{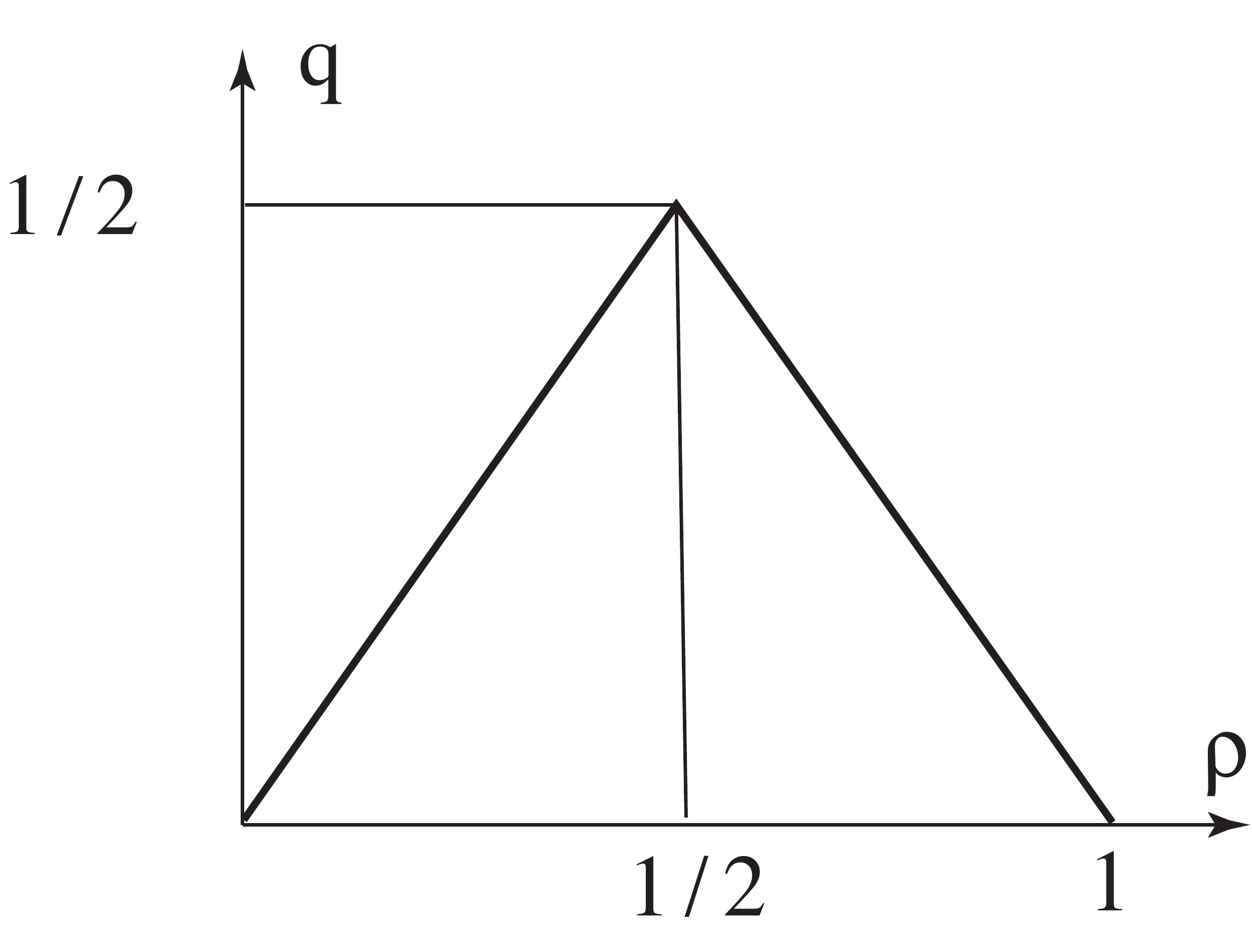}{0.2}{The fundamental traffic law.}{fundamental}
\subsection{Event Graph modeling}
Consider the Petri net given in III of Figure~\ref{ruta}
which describes in a different way the same dynamics. In fact this
Petri net is an event graph and therefore its dynamics 
is linear in minplus algebra.
The number of vehicles entered  in the place $i$ 
before time  $k$ is denoted $x_{i}^k$.
The  initial vehicle position
is given by  booleans $a_{i}$ with takes the value $1$ when 
the cell contains a  vehicle and  $0$ otherwise.

We use the notation $\bar{a}=1-a$, then the dynamics is given by~:
$$x_{i}^{k+1}=\min\{a_{i-1}+x_{i-1}^k,\bar{a}_{i}+x_{i+1}^k\}\;,$$
which can be written linearly in minplus algebra~:
$$x_{i}^{k+1}=a_{i-1}x_{i-1}^k\oplus\bar{a}_{i}x_{i+1}^k\;.$$

%
This event graph has three kinds of elementary  circuits~: 
-- the outside circuit with average mean $n/m$, 
-- the inside circuit with average mean $(m-n)/m$, 
-- the circuits corresponding to make some step forward and coming
   back, with average mean 1/2,
Therefore its eigenvalue is \[q=\min(n/m,(m-n)/m,1/2)=\min(\rho,1-\rho)\; ,\] 
which gives the average speed as a function of the car density.

\section{2D traffic}

Let us generalize the second approach to derive the fundamental
diagram to a regular town describe in Figure~\ref{town}.
\dessin{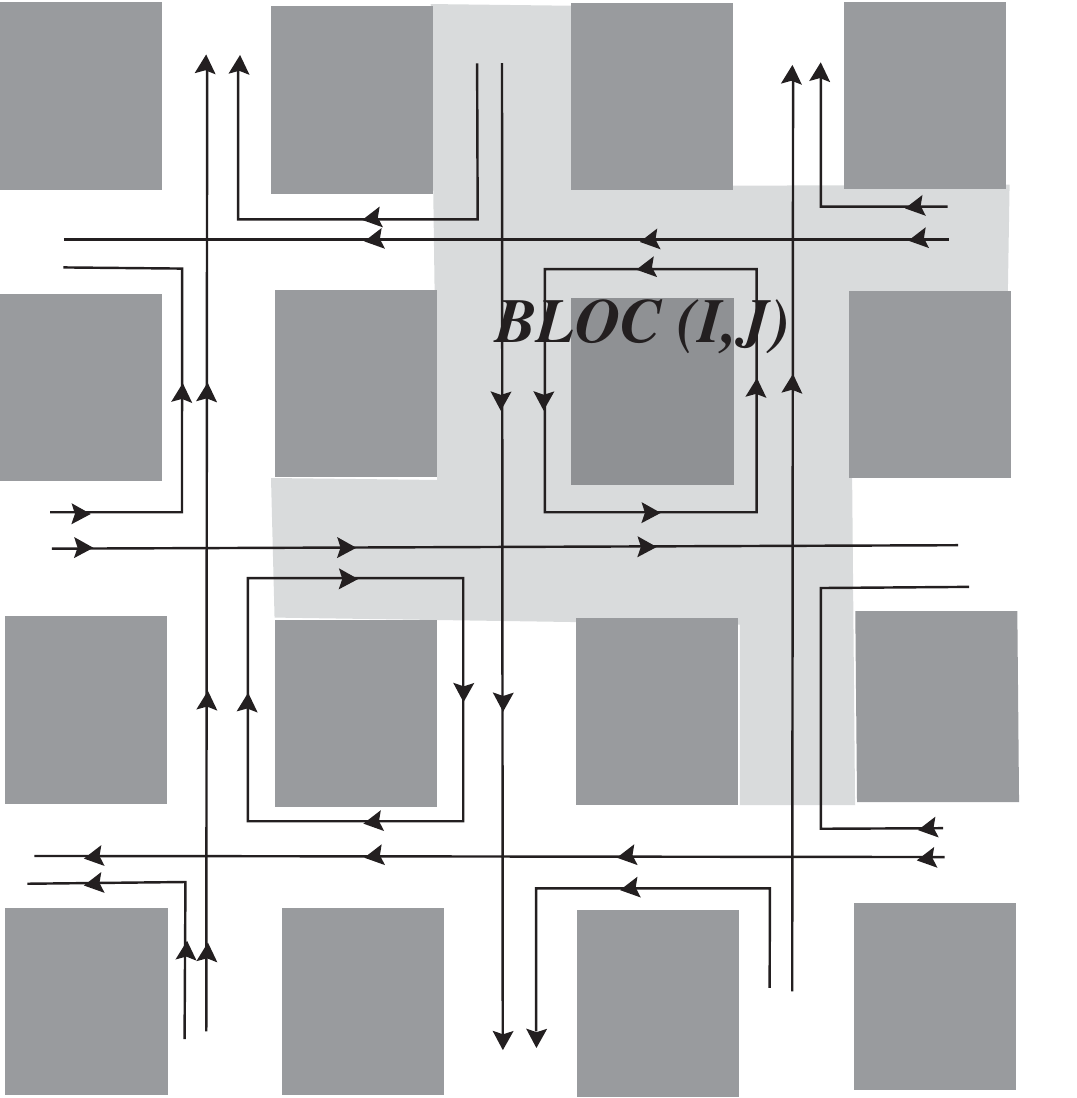}{0.3}{A town.}{town}
The complete town can be modeled as a set of subsystems corresponding
to a  unique crossing and two adjacent roads. To write the dynamics
of the town we have first to give the Petri net describing a crossing.

A first trial is to consider the Petri net given in Figure~\ref{crossingsimp}.
This Petri net is not anymore an event graph but following L. Libeaut\cite{LIB}
it is possible to write the nonlinear implicit minplus equation 
describing a general Petri net. In the case where the multipliers are
all equal to one it is ~:
\begin{equation}\min_{p\in
x^{in}}\left[a_p+\sum_{x'\in p^{in}}x'(k-1)-\sum_{x''\in
p^{out}}x''(k)\right]=0, \;\forall x\;, \forall k\label{sept}.
\end{equation}
where $x(k)$ denotes the firing number of transition $x$ and $p$ a place of the Petri Net.
\dessin{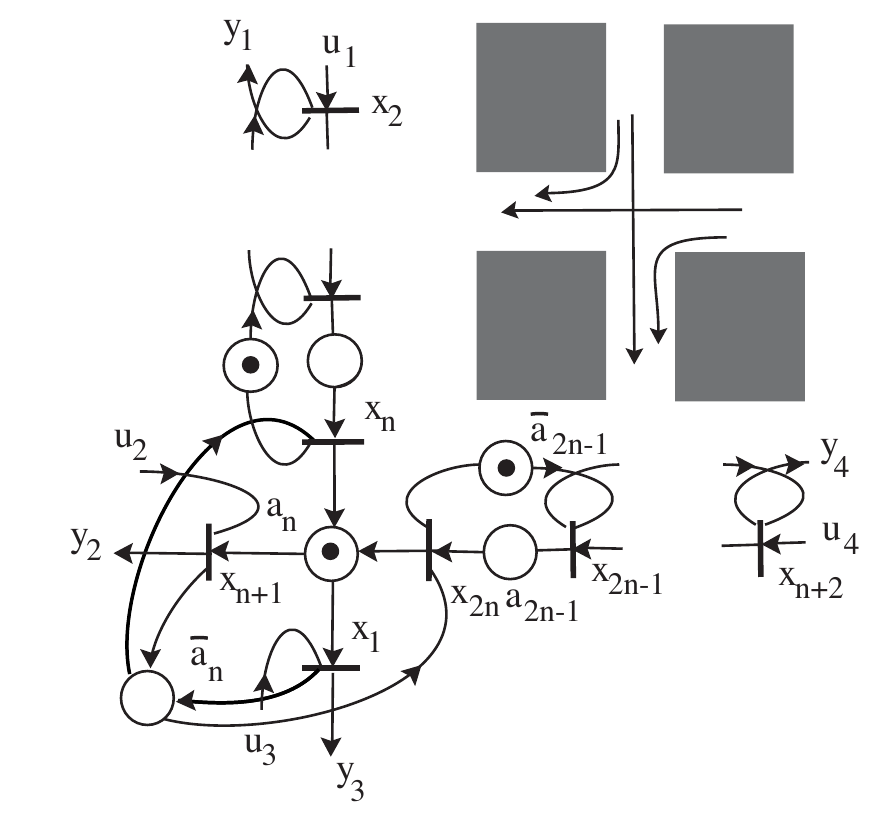}{0.7}{A simplified crossing.}{crossingsimp}
But these equations do not determine completely the dynamics
since solution to the Cauchy problem is not unique, in general.
Indeed~:
-- at place $a_n$ we may have a \emph{routing policy} giving the proportion
of cars going towards $y_2$ and the proportion going towards $y_3$
(which is not described by the Petri net~\ref{crossingsimp})
-- at place $\bar{a}_n$ we may follow \emph{the first arrived the first served} rule with the right priority if two cars arrive simultaneously at the crossing
(which is also not described by the Petri net~\ref{crossingsimp}).

Precising the dynamics of Petri net in such way that the trajectories are uniquely defined corresponds to give another Petri net having only one arc leaving each place. Let us discuss more precisely these points on a simple system given
in the first picture of Figure~\ref{homog}.
\dessin{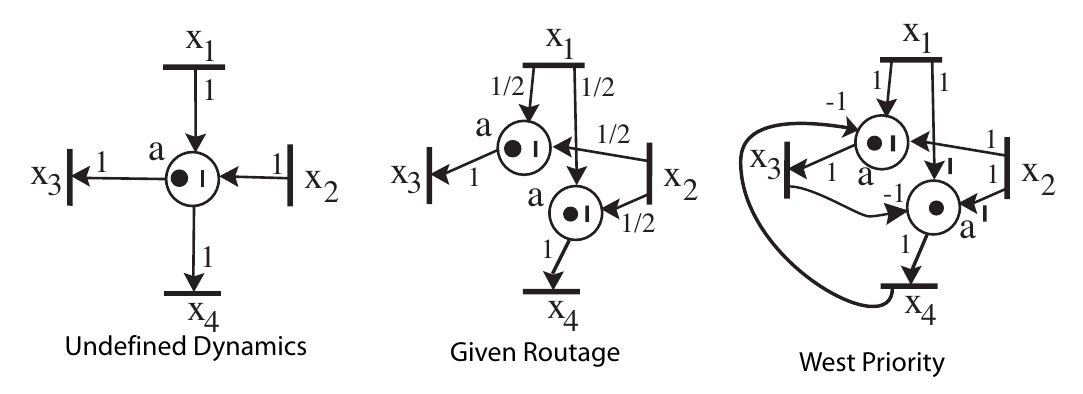}{0.9}{Dynamic Completion.}{homog}

The incomplete dynamics of this system 
can be written in minplus algebra 
$x^n_{4}x^n_{3}=ax^{n-1}_{1}x^{n-1}_{2}$. 
Clearly $x_{4}$ and $x_{3}$
are not defined uniquely. We can complete the dynamics, 
for example, in the two
following ways useful for the traffic application~:
-- by precising the routing policy 
$$x^n_{4}=x^n_{3}=\sqrt{ax^{n-1}_{1}x^{n-1}_{2}}$$
-- by choosing a priority rule
$$ \begin{cases}
x^n_{3}=ax^{n-1}_{1}x^{n-1}_{2}/x^{n-1}_{4} \\
x^n_{4}=ax^{n-1}_{1}x^{n-1}_{2}/x^n_{3}.
\end{cases} $$
In the two cases we obtain a \emph{degree one homogeneous minplus} system.

\dessin{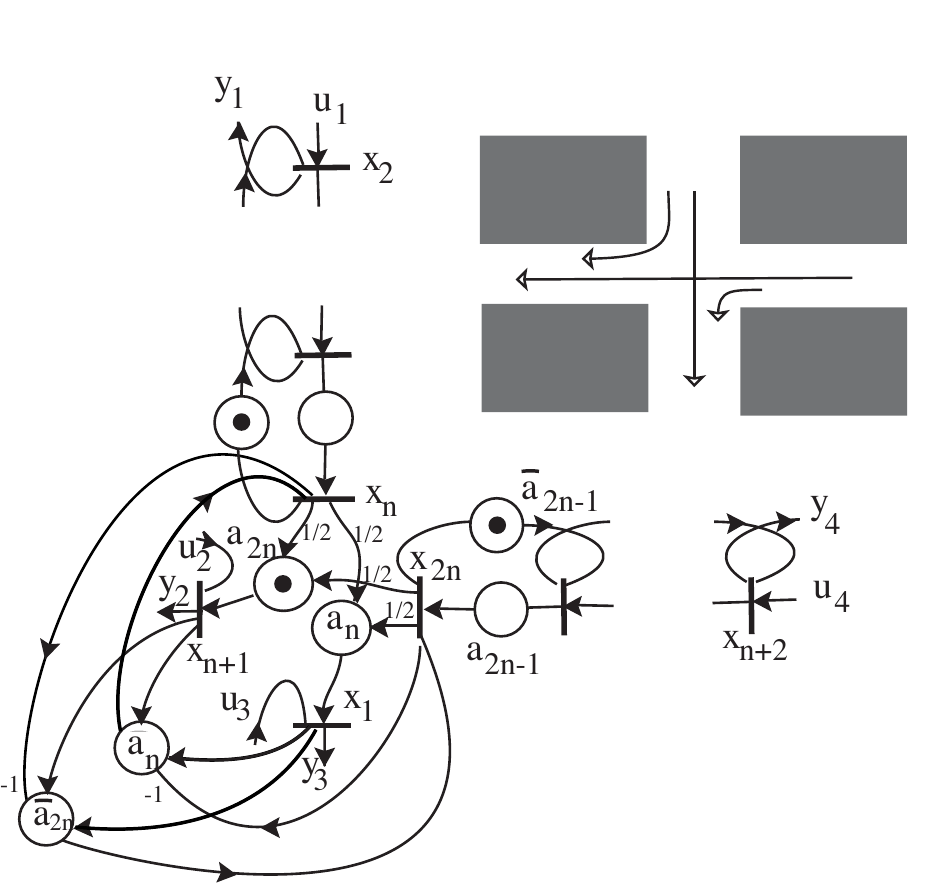}{0.70}{A Complete Crossing.}{crossingcomp}
This method can be applied to the crossing and we obtain a Petri net with
negative weights which has only one arc leaving each place 
(that we call deterministic Petri net) see Figure~\ref{crossingcomp}.

Neglecting the roundings the system can be written with minplus notations~: 
$$\begin{cases}
x_{i}/\delta=a_{i-1}x_{i-1}\oplus\bar{a}_ix_{i+1},\; \\
x_{n}/\delta=\bar{a}_nx_1x_{n+1}/x_{2n}\oplus a_{n-1}x_{n-1}\;,\\
x_{2n}/\delta=\bar{a}_{2n}x_1x_{n+1}/(x_{n}/\delta)\oplus
a_{2n-1}x_{2n-1}\;, \\
x_{1}/\delta=a_{n}\sqrt{x_{n}x_{2n}}\oplus\bar{a}_1x_2\;,\\
x_{n+1}/\delta= a_{2n}\sqrt{x_{n}x_{2n}}\oplus
\bar{a}_{n+1}x_{n+2}\;,\\
\end{cases}$$
where $\delta$ denotes the forward shifting operator acting on sequences.
It is a general degree 1 homogeneous minplus system.

Simulation of this system starting from 0 shows that 
$$\lim_{k}x^k_{i}/k=\lambda, \;\; \forall i \; .$$
The constant $\lambda$ has the interpretation of the average
speed. The fundamental diagram gives the relation between the
average speed and the vehicle density of the system. In Figure~\ref{LoiFond}
we give this law in the cases of two circular roads with one crossing
for different relative size of the two roads. We see that three phases
appear on each fundamental diagram. These phases will be discussed in the
second part of this paper.
\dessin{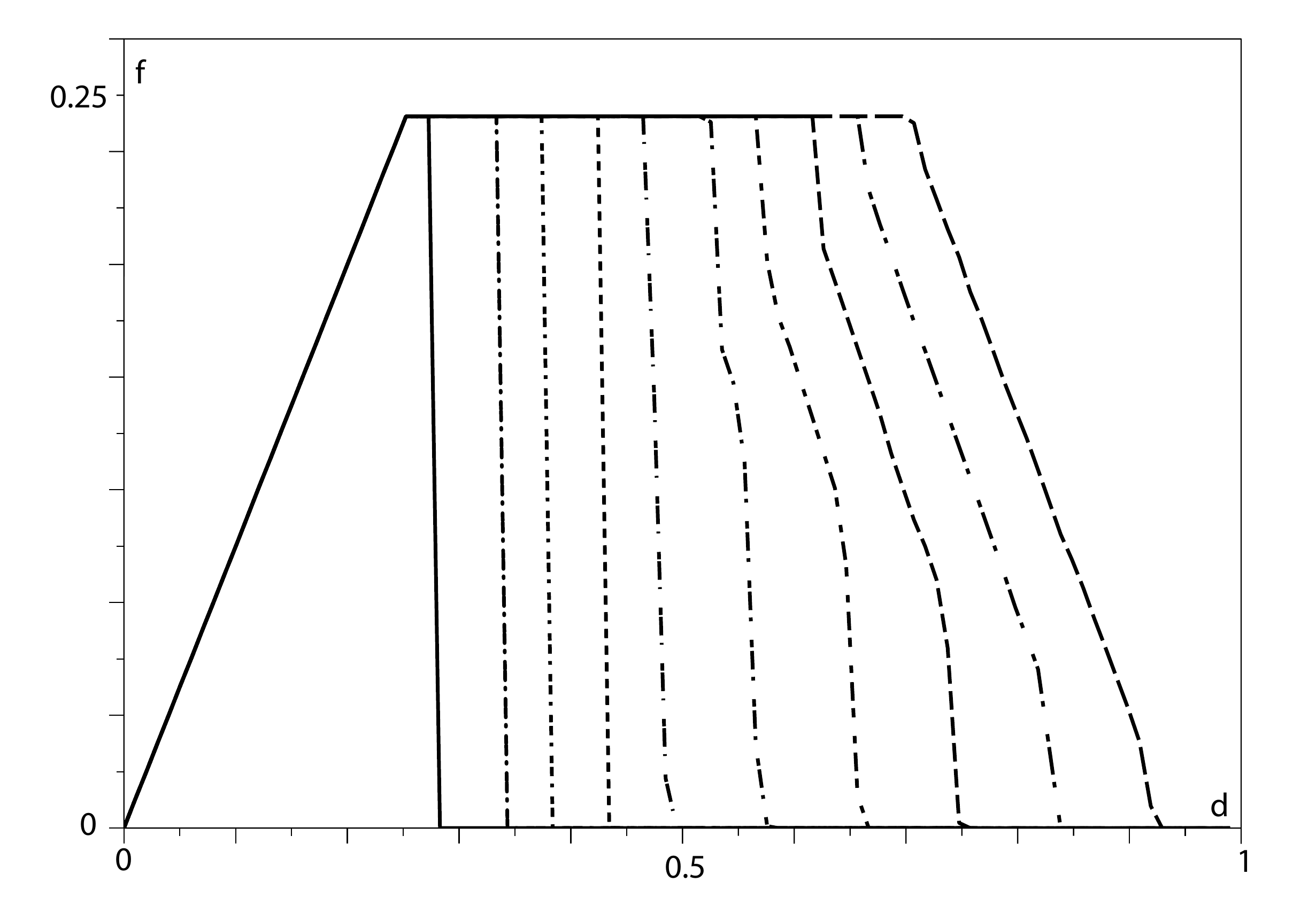}{0.20}{2D-traffic fundamental diagrams.}{LoiFond}
The experimental existence of this $\lambda$ motivates the study of the
eigenvalue of 1-homogeneous minplus system.

\section{Eigenvalues of 1-homogeneous minplus systems}
The eigenvalue problem for 1-homogeneous system $f:\Rmin^{n} \mapsto \Rmin^{n}$ can be formulated as finding $x\in \Rmin^n$ non zero, and $\lambda \in \Rmin$ such that~:
$$\lambda x=f(x)\; .$$
Since $f$ is 1-homogeneous, supposing without loss of generality
that if $x$ exists $x_{1}\neq \epsilon$, the eigenvalue problem becomes~:
$$\begin{cases}
\lambda&=f_{1}(x/x_{1})\; , \\
x_{2}/x_{1}&=(f_{2}/f_{1})(x)\; , \\
\cdots&= \cdots\\
x_{n}/x_{1}&=(f_{n}/f_{1})(x)\;,
\end{cases}$$
Denoting $y=(x_{2}/x_{1},\cdots, x_{n}/x_{1})$ the problem is
reduced to the computation of the fixed point problem
$y=g(y)$ (with $g_{i-1}(y)=(f_{i}/f_{1})(0,y)$) to compute a normalized eigenvector from which the eigenvalue is deduced by~:
$\lambda=f_{1}(0,y)$. But now $g$ is a general minplus function.

The fixed point problem has not always a solution.
There are cases where we are able to solve the problem
-- $f$ is affine in standard algebra,
-- $f$ is minplus linear,
-- $f$ is positive power function.
In the first case there is a unique
eigenvalue as soon as  $\dim(\ker(f-I))=1.$ 

In the two last cases, the problem can be reduced to the minimization
of the average cost by time unit using dynamic programming methods.
The corresponding fixed points are unique and stable. 

Moreover, since $\max(x,y)=xy/(x\oplus y)$ games
problem are also 1-homogeneous minplus systems and the solution of the corresponding eigenvalue problem is known.

In the general case we may have unstable fixed points that, nevertheless, we
can compute by Newton method (which is exactly the policy iteration) but which don't 
give the information about the asymptotic behavior of the system anymore. In this
case the asymptotic is obtained by an averaging based on invariant measure which may be difficult to compute.
Let us give an example of chaotic system which has a 1-homogeneous minplus dynamics.

\section{A Chaotic system example}
Let us consider the 1-homogeneous minplus dynamic system
$$\begin{cases}
x_{1}^{k+1}=(x_{1}^{k})^2/x_{2}^k\oplus 2(x_{2}^{k})^3/(x_{1}^{k})^2\;, \\
x_{2}^{k+1}=x_{2}^k\;.
\end{cases}$$

The corresponding eigenvalue problem is
$$\begin{cases}
\lambda x_{1}=x_{1}^2/x_{2}\oplus 2x_{2}^3/x_{1}^2, \\
\lambda x_{2}=x_{2}\; .
\end{cases}$$
\dessin{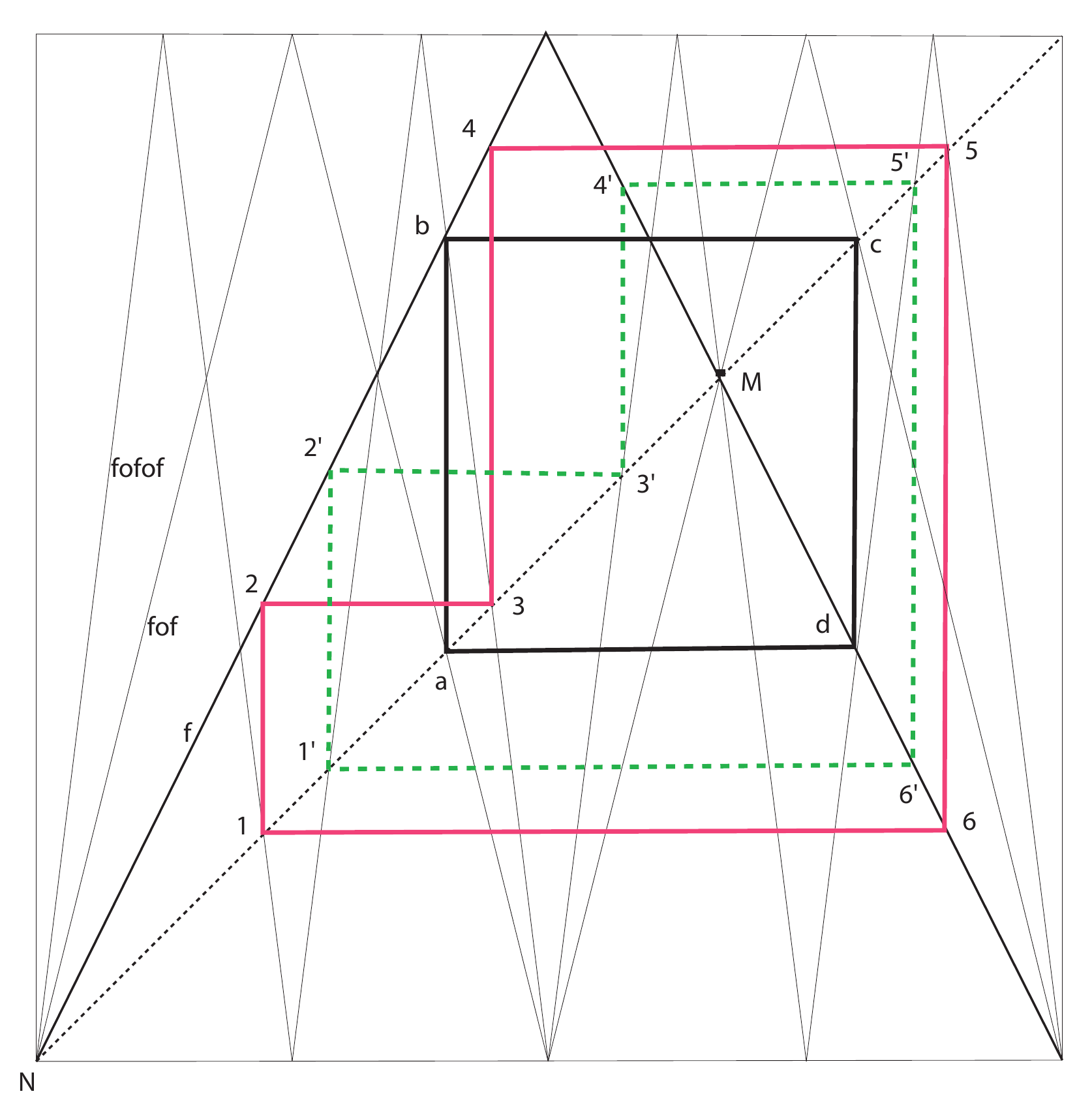}{0.3}{Cycles of tent transformation.}{tent}

The solutions are $\lambda=0$ and $y=x_{1}/x_{2}$ satisfying the equation $$y=y^2\oplus 2/y^2,$$ 
which has for solutions $y=0$ and $y=2/3$. These two solutions are unstable fixed points of the transformation $f(y)=y^2\oplus 2/y^2$.
But the system $y_{n+1}= f(y_{n})$ is a chaotic system 
since $f$ is the tent transform (see~\cite{BERG} for example for a clear
discussion of this dynamics).
In Figure~\ref{tent} we show the graph of $f$, $f\circ f$ , $f\circ f \circ f$, their fixed points and the corresponding periodic trajectories.

In Figure~\ref{tentsimu} we show a trajectory for an initial
condition chosen randomly with the uniform law on the set $\{(i-1)/10^5, i=1,\cdots,10^5\}$. The diagonal line in the picture is
a decreasing sort applied to the trajectory. It shows that the invariant
empirical density is uniform.
\dessin{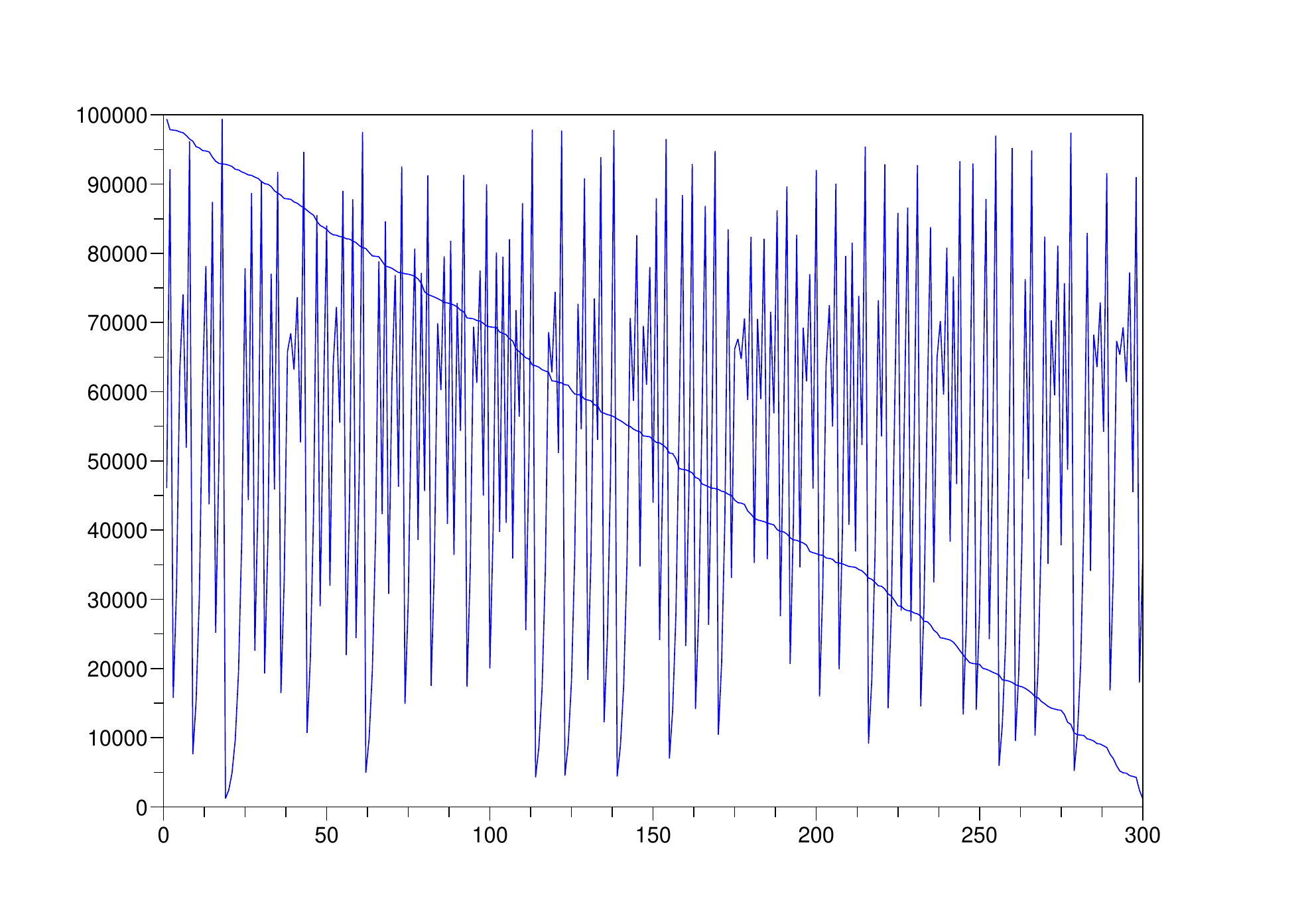}{0.3}{A tent iteration trajectory.}{tentsimu}
We can prove that the tent iteration has a unique invariant
measure absolutely continuous with respect to the Lebesgue
measure~: the uniform law on $[0,1]$.

More generally a chaotic 1-homogeneous minplus system
will grow linearly with a value $\lambda$ given by~:
$$\lambda=\int f_{1}(y)d\mu(y)\;,$$
where $\mu$ is the invariant probability measure of $y$ 
depending on the initial value $y^0$.
For example, according to the initial value $y^0$, 
the tent iterations $y^{k}$ stay in circuits or follow trajectories 
without circuit (possibly dense in $[0,1]$).

\setcounter{section}{0}
\setcounter{footnote}{0}
\setcounter{figure}{0}
\nachaloe{N. Farhi, M. Goursat, and J.-P. Quadrat}{Degree one homogeneous \mbox{minplus} \mbox{dynamic} 
\mbox{systems} and \mbox{traffic applications : Part II}}%
{Partially supported by RFBR/CNRS grant 05-01-02807.}
\label{far-abs2}
\markboth{N. Farhi, M. Goursat, J.-P. Quadrat}{Degree one homogeneous dynamic systems: Part II}

In this second part we 
discuss the phases appearing in the 
fundamental diagrams of traffic systems modeled by 
1-homogeneous minplus dynamics and show the 
improvement obtained by traffic light control.  

We have shown in the first part 
that 1-homogeneous systems may have a 
chaotic behavior. Here we give a new subclass 
of 1-homogeneous dynamics having periodic 
trajectories. It generalizes the standard cases 
which need a monotony property. Moreover we 
show that this new, but still restrictive class, 
has applications to regular town traffic with 
crossings but without turning possibilities.

\section{The traffic fundamental diagram phases.}
The fundamental diagrams of quite 
different systems are similar to
the one given in part I. We have 
studied the cases of two circular
roads with one crossing and two crossings 
and the cases of regular towns with various 
number of roads on a torus. In all these cases we 
suppose the existence of right priority.

The fundamental diagrams have always three phases 
corresponding respectively to low, average and high densities. 
We see on the
fundamental diagram of Part I that~: -- for low densities the flow 
increases linearly with the density, -- for average densities 
the flow is constant, -- for high density there are 
deadlocks and the flows are null.

On Figures (\ref{un-phases}), (\ref{deux-phases}) and (\ref{ville-phases}) we show the typic asymptotic distribution of vehicles in the three phases \cite{FGQ1,FGQ2}.
  \begin{figure}[h]
    \begin{center}
      \includegraphics[width=3.2cm]{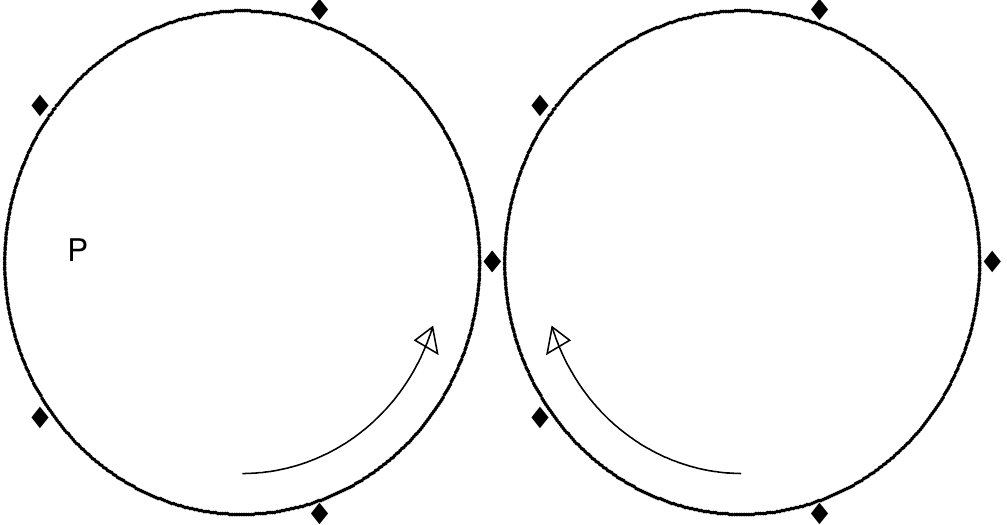}\hspace{0.5cm}
      \includegraphics[width=3.2cm]{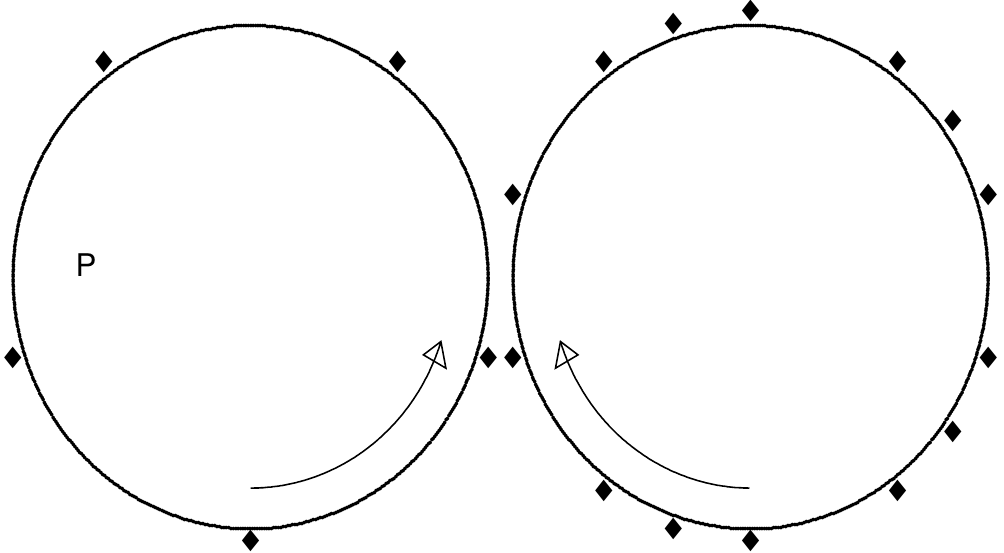}\hspace{0.5cm}
      \includegraphics[width=3.2cm]{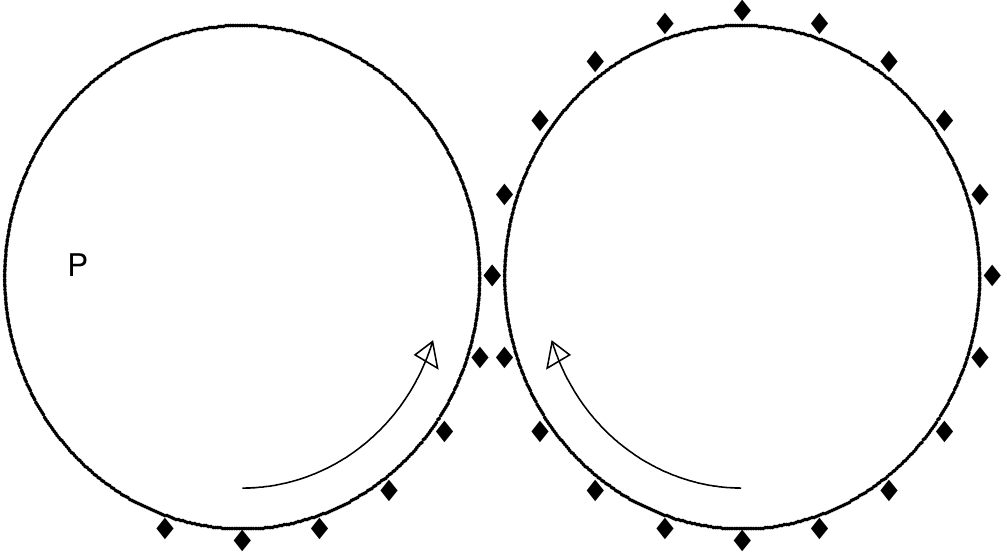}
    \end{center}
    \caption{Two circular roads with one crossing case. Car distributions in the low average and high density phases.}
    \label{un-phases}
  \end{figure}
  \begin{figure}[h]
    \begin{center}
      \includegraphics [width=3.1cm]{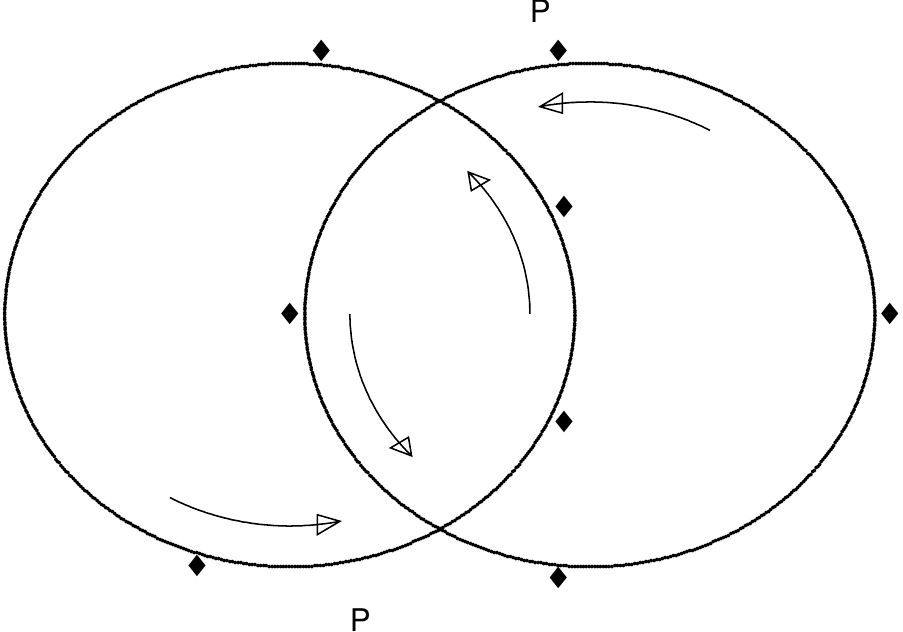}\hspace{0.5cm}
      \includegraphics[width=3.1cm]{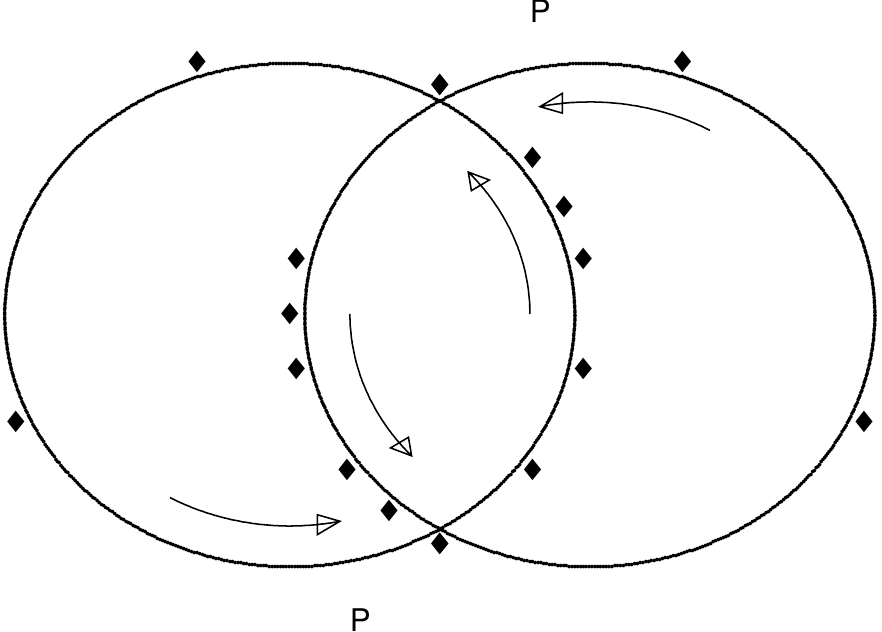}\hspace{0.5cm}
      \includegraphics[width=3.1cm]{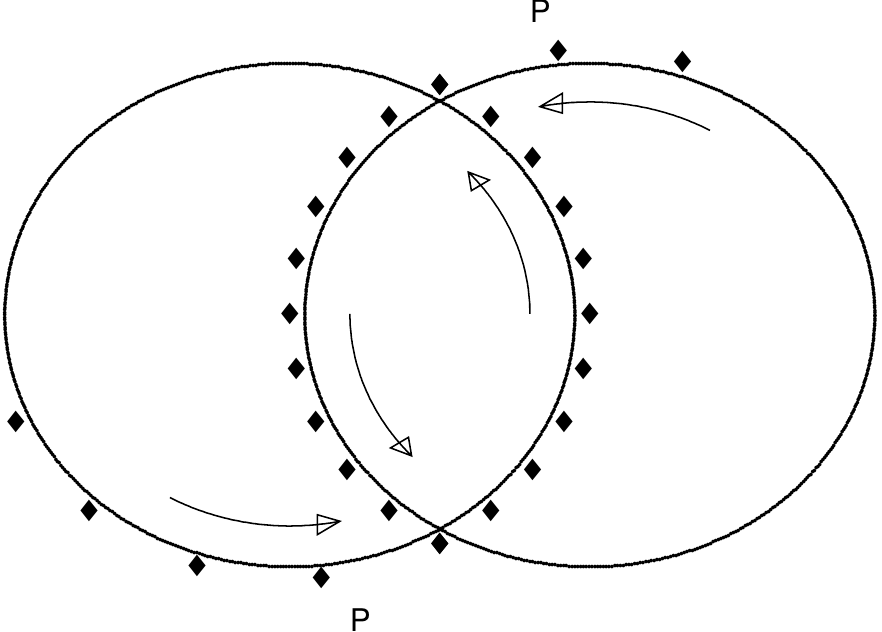}
    \end{center}
    \caption{Four roads with two crossings. 
Car distributions in the low average and high density phases.}
    \label{deux-phases}
  \end{figure}
  \begin{figure}[h]
    \begin{center}
      \includegraphics [width=4cm]{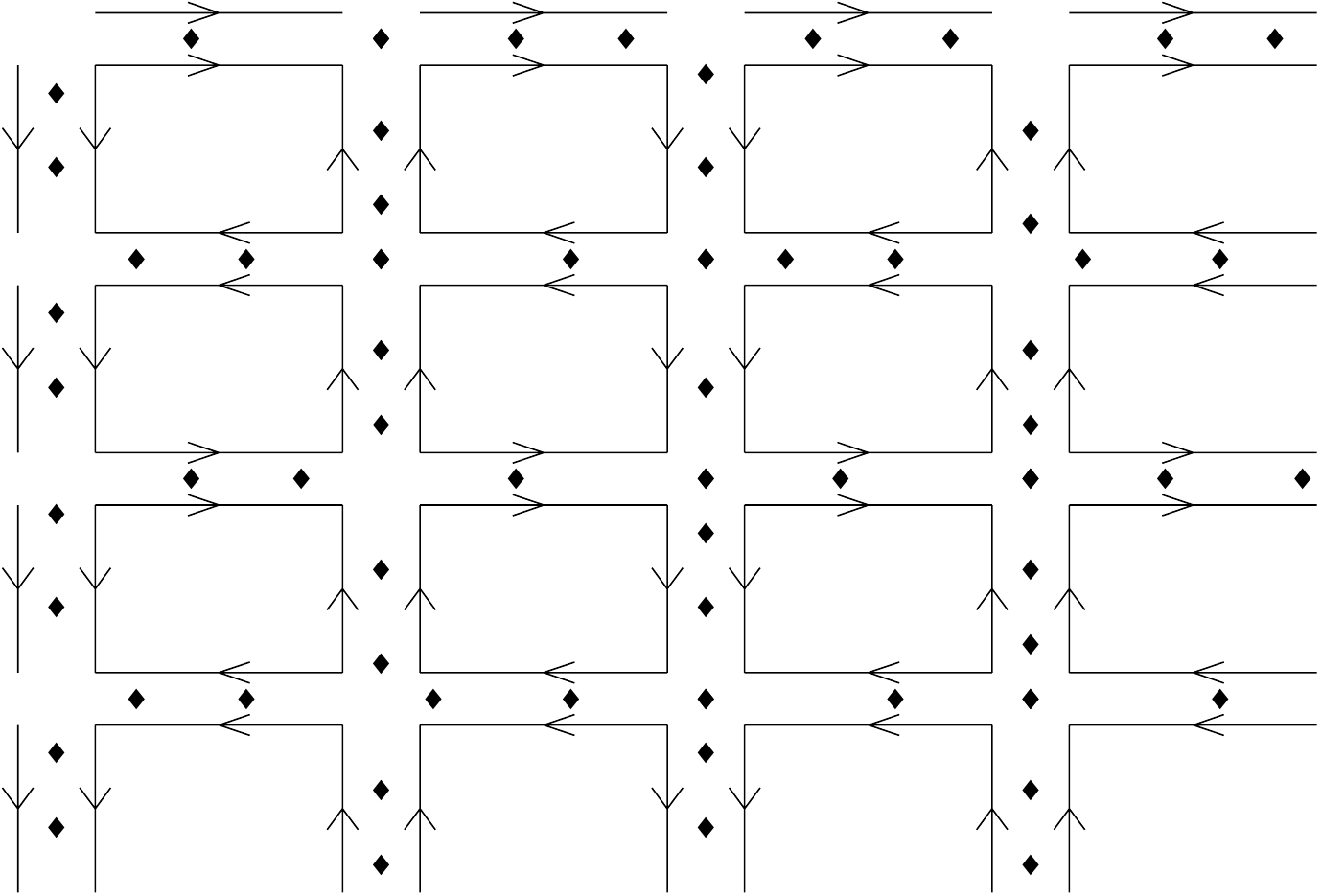}\hspace{0.2cm}
      \includegraphics[width=4cm]{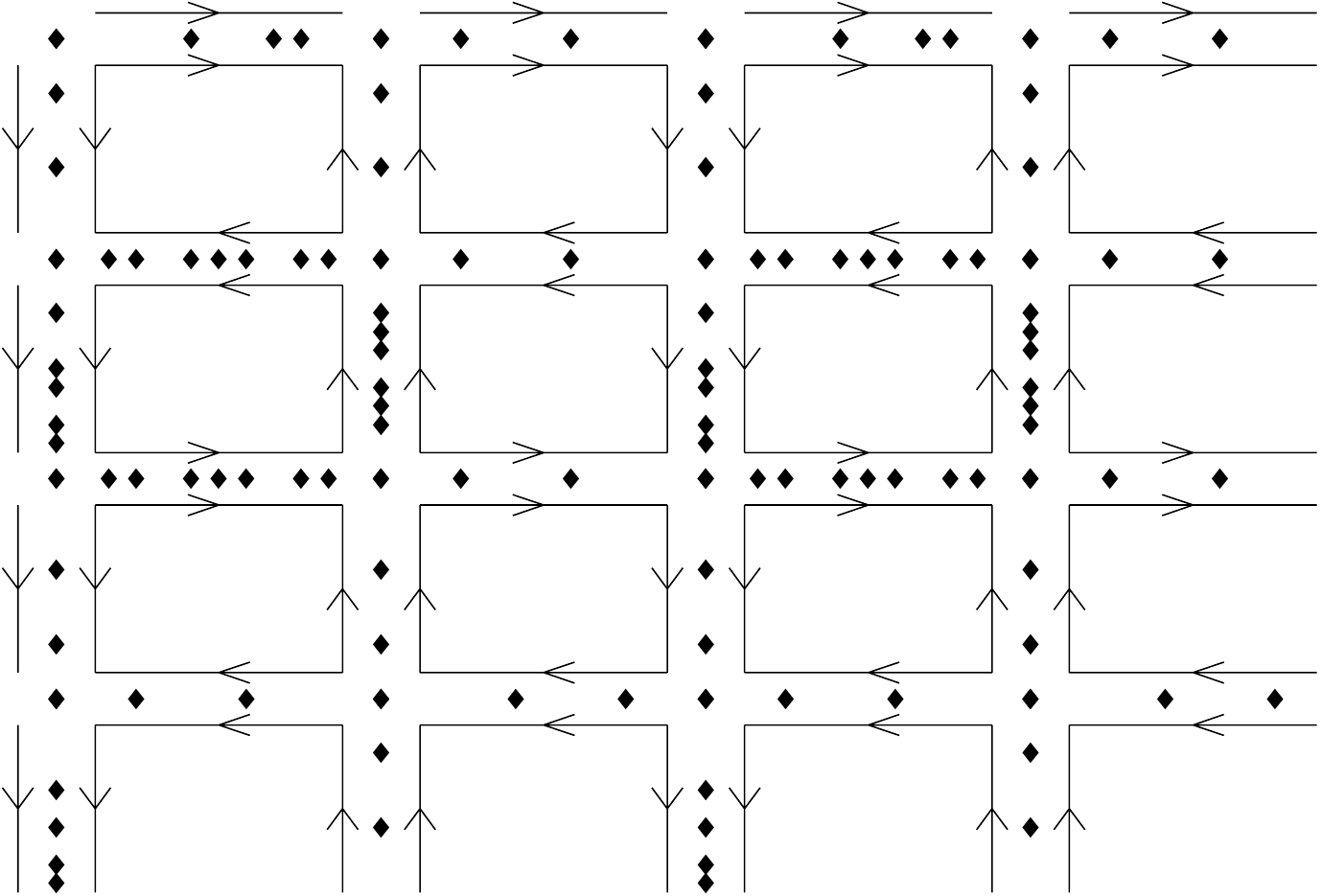}\hspace{0.2cm}
      \includegraphics[width=4cm]{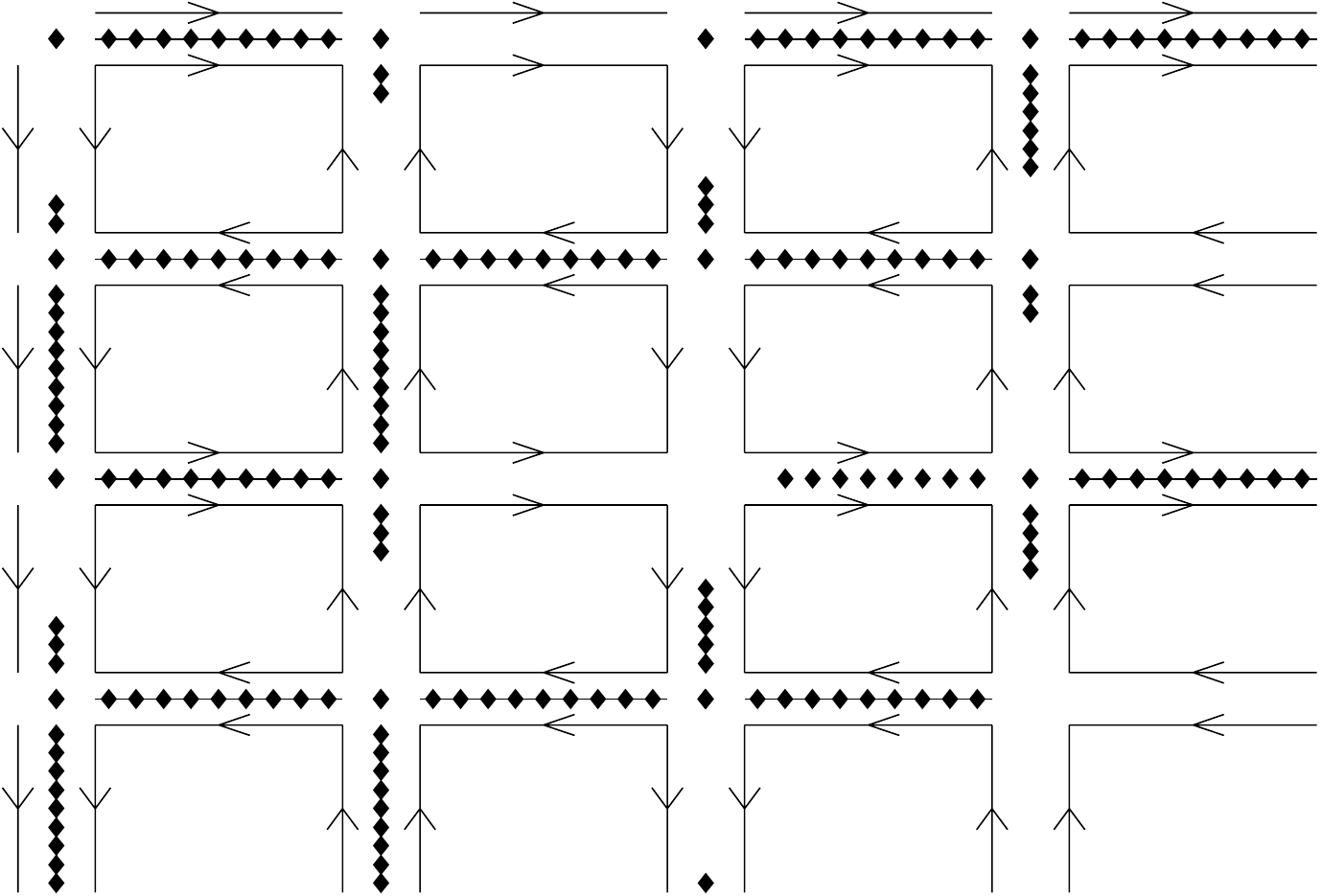}
    \end{center}
    \caption{A regular town. Car distributions in the low average and high density phases.}
    \label{ville-phases}
  \end{figure}\\
 We see that~:
\begin{itemize} 
\item \emph{Low density phase.} There are so few vehicles in the network that after a transient regime, they  move without obstructing each other on the roads and in the crossings. Thus, the ``priority to the right'' is not used, the vehicles moves as on a unique circular road and the average flow is equal to the vehicle density in the network. This phase corresponds to densities less than $1/4$.
\item \emph{Average density phase.} When the density is between 1/4 and 1/2 (in the symmetric road cases), the vehicles can neither move freely on the roads, nor avoid each other on the crossings. Therefore ``priority to the right'' happens. The car on the priority road move freely and the waiting cars are all in the non priority road. The flow reaches the maximum value 1/4 corresponding to the full use of the crossings.
\item \emph{High density phase.} When the density exceeds a quantity equal to 1/2 (in the symmetric case), at asymptotic regime, a closed circuit of vehicles on some nonpriority roads appears which creates a complete deadlock of the system. 
\end{itemize}

\section{Traffic light control.} 
To avoid the deadlock due to right priority we can use traffic light controls.
A Petri net describing the junction with the traffic light control is shown on Figure~\ref{feu}. The negative weight extension of Petri net is necessary to model
the light phases in a time invariant way. The part of the Petri modeling the light
control corresponds to the places $a_{g},a_{c},\bar{a}_{g},\bar{a}_{c}$. As long as $a_{c}$ contains
a token $a_{c}=1$ the green light is for the North street, when $\bar{a}_{c}=1$ the green light is for the East street. As long as $a_{c}=1$ we have $a_{g}=1$ and $q_{v}$ is authorized to fire (since thanks
to the loop $q_{v},a_{g},q_{v}$ as soon as a token is consumed another one is generated in
the place $a_{g}$).
The main point is that when the token in $a_{c}$ goes in $\bar{a}_{c}$ (phase change) the tokens in
$a_{g}$ must be removed (this is done by the input arc with weight -1 of the place $a_{g}$). More generally without negative weight we cannot model tokens staying less then a prescribed time.

  \begin{figure}[h]
    \begin{center}
      \includegraphics[width=6cm]{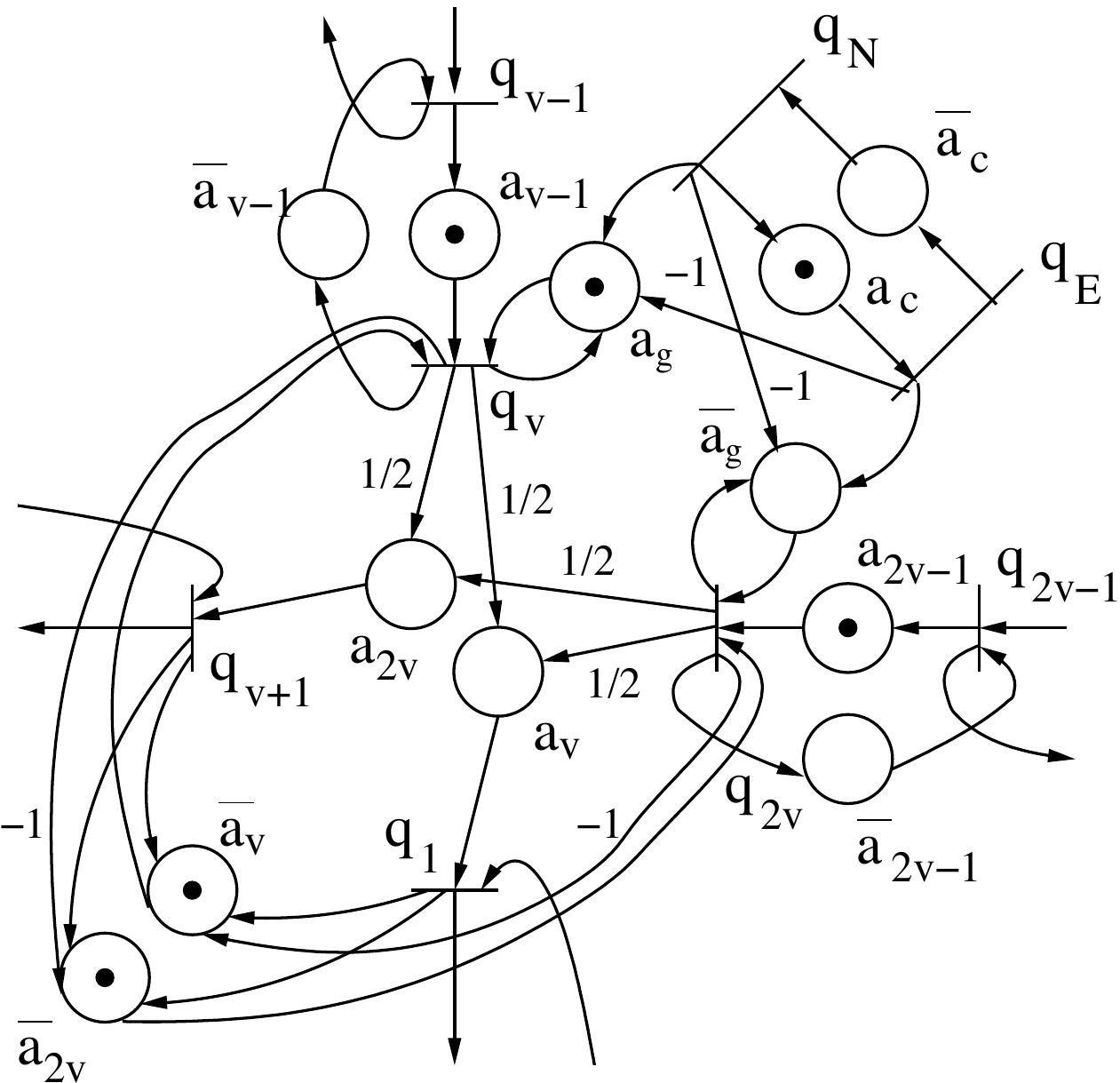}
    \end{center}
    \caption{Traffic lights modeling.}
    \label{feu}
  \end{figure}
  
In Figure~\ref{compar}, we compare 
the fundamental diagrams of three
crossing policies for a system 
composed of two circular roads of same size with two
junctions. The three 
policies are~: -- right priority, -- standard given phase duration,
-- feedback controlled duration 
(based on the road congestion) computed by LQG method.

The control improves the average and the high density
phases, without spoiling the low density one. The improvement
given by the feedback control achieves the throughput obtained
on a unique circular road without 
crossing but doubling the time spent
in a place representing the crossing place.

Furthermore, the feedback control dissolves more
efficiently the jams (that can appear locally in 
transient regimes) than the other policies would do.
 \begin{figure}[h]
    \begin{center}
      \includegraphics[width=5cm]{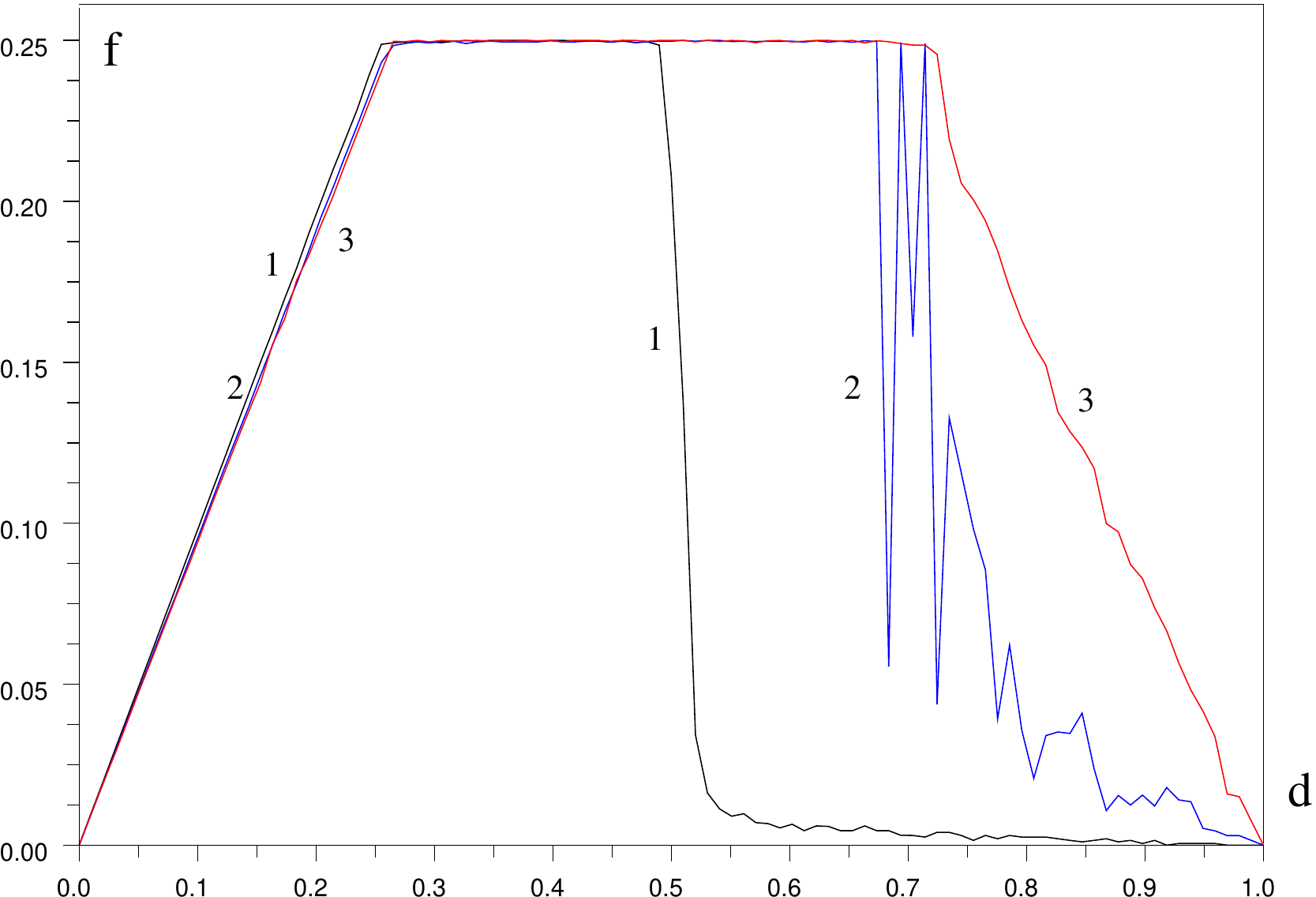}
    \end{center}
\caption{Comparison of three 
policies of managing the crossings~: --right priority 
to the right (1), --open loop 
light control (2), --feedback light control (3).}
    \label{compar}
  \end{figure}
\section{A subclass of triangular homogeneous dynamics}
In this section we study a subclass 
of 1-homogeneous minplus linear systems for
which we can prove the periodicity. 
Their dynamics belongs to a subclass of 1-homogeneous
triangular systems~:
\begin{equation}\label{trihom}
  \left\{ \begin{array}{l}
    u_{k+1}=C\otimes u_k,\\
    x_{k+1}=A(u_k)\otimes x_k\oplus B(u_k)\otimes u_k.
  \end{array}\right.
  \end{equation}
where $\{u_k\}_{k\in\mathbb N}$ and $\{x_k\}_{k\in\mathbb N}$ are
minplus column vectors, $C$ is a minplus square matrix, $A(u_k)$ and
$B(u_k)$ are two minplus 0-homegeneous matrices depending of $u_k$. 
We call this type of systems Triangular 1-Homogeneous (T1H).
 
We call linear periodic dynamic (LP) a dynamic given by~:
  $$x_{k+1}=A_k\otimes x_k, \quad x_0 \text{ given},$$
where $A_k$ are minplus matrices periodic in $k$.

We can prove the following theorems (see the proofs in~\cite{FAR456}).

\begin{theorem}\label{principal}
  Every T1H dynamics behaves asymptotically as a LP dynamics.
\end{theorem}
\begin{theorem}
  A T1H system with $A(u)$ irreducible for every $u\in\mathbb R_{min}$ satisfies~:
  $$\max_{u_0\in\mathbb R_{\min}}\mu_x(u_0)=\max_{\bar{u}\in \mathcal V}\mu_x(\bar{u}),$$
  where~: -- $u_{0}$ denotes the initial condition of $u$, -- $\mu_x(u_{0})=\lim_{k\rightarrow\infty}x_k/k$, -- $\mathcal V$ is the set of the minplus eigen vectors of the matrix $C$.
\end{theorem}
\begin{theorem}
  Every LP dynamic $y_{k+1}=E_{k}\otimes y_k$, such that the matrices $E_k$ have the
  same support, is realizable by a T1H dynamics.
\end{theorem}
\section{Application to traffic}
We show that the traffic of regular towns with traffic light, buffered junction but without turning possibilities can be modeled with a T1H dynamics.
\begin{figure}[h]
  \begin{center}
    \includegraphics[width=7cm]{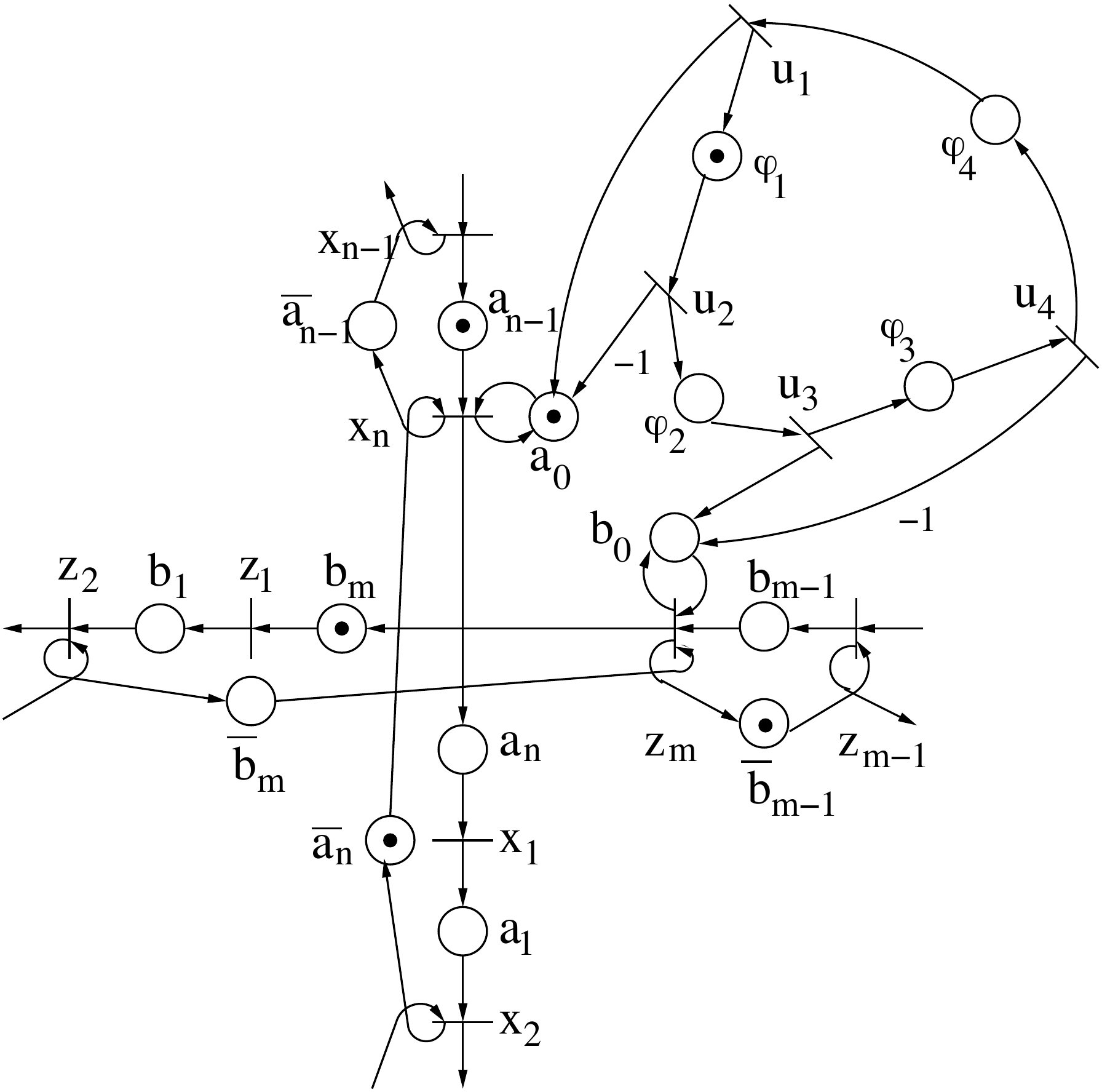}
    \caption{Traffic light intersection without possibility of turning.}
    \label{lights}
  \end{center}
\end{figure}

On the Petri net of Figure~\ref{lights} the traffic light is modeled by the subsystem 
corresponding to the transitions $u_1,u_2,u_3,u_4$, which has no input coming from the
rest of the system. The dynamic of this subsystem is minplus linear.
If the initial condition $u_{0}=(0,0,0,0)$ the
number of tokens in the places $a_0$ and $b_0$ is boolean and periodic.
To a cycle corresponds the four phases given in the Table~\ref{table1}.
\begin{table}[h] 
\begin{center}
  \begin{tabular}{||c|c|c|c|c||}
      \hline
      Phase & $a_0$ & $b_0$ & Vertical light color & Horizontal light color\\ \hline \hline
      1 & 1 & 0 & green & red\\ \hline
      2 & 0 & 0 & red & red\\ \hline
      3 & 0 & 1 & red & green\\ \hline
      4 & 0 & 0 & red & red\\
      \hline
    \end{tabular}\vspace{0.5cm}
    \caption{The phases of the traffic light.}
    \label{table1}
\end{center}
\end{table}
The junction has a buffer place in each direction ($a_{1}, b_{1}$) to avoid blocking.
The phases 2 and 4 gives the time, for car entering in the junction, to go in the buffer and then to free the crossing. Indeed, a vehicle entering in the crossing (represented by the two places
$a_n$ and $b_m$) leaves it surely in one unit of time. 

The green duration of phase 1 and 3
is the sojourn time of tokens in the place $\varphi_i$. The phases 2~and~4 have a duration of one unit.


\begin{proposition}
 The dynamics of the Petri net of Figure \ref{lights} has the T1H dynamics~:
   $$ u^{k+1}=\begin{bmatrix}
              \cdot & \cdot & \cdot & \varphi_4\\
              \varphi_1 & \cdot & \cdot & \cdot\\
              \cdot & \varphi_2 & \cdot & \cdot\\
              \cdot & \cdot & \varphi_3 & \cdot
           \end{bmatrix}\otimes u^k,\quad\quad
     \begin{bmatrix} x^{k+1}\\z^{k+1}  \end{bmatrix}=
     \begin{bmatrix}
       A_1(u^k) & \cdot\\
   \cdot & A_2(u^k)
     \end{bmatrix}\otimes
     \begin{bmatrix}  x^k\\z^k  \end{bmatrix},$$
 where $\cdot$ denotes $\infty$, with 
 $$A_1(u)_{i,j}=\begin{cases}
         a_0u_1/u_2\ \text{if } (i,j)=(n,n),\\ 
         \text{independent of} u\ \text{elsewhere.}
\end{cases} $$
 and  
 $$ A_2(u)_{i,j}=\begin{cases}
            b_0u_3/u_4\  \text{if } (i,j)=(m,m),\\ 
            \text{independent of } u\ \text{elsewhere.}
            \end{cases} $$
\end{proposition}

We are able to explicit the asymptotic flows which are different according the direction
followed by the vehicles.
\begin{theorem}
 The average flow on the horizontal (resp. vertical) road is given by $\lambda/4$
 where $\lambda$ is the unique eigenvalue of the irreducible matrix $\bigotimes_{k=0}^3A_1(u^k)$
 [resp. $\bigotimes_{k=0}^3A_2(u^k)$].
\end{theorem}

\bibliographystyle{amsalpha}

\newpage
\renewcommand{\thetheorem}{\arabic{theorem}}
\renewcommand{\theproposition}{\arabic{proposition}}
\setcounter{theorem}{0}
\setcounter{proposition}{0}
\setcounter{section}{0}
\setcounter{footnote}{0}
\nachaloe{F.~ Faye, M.~Thiam, L.~Truffet, 
and E.~Wagneur}{Max-plus cones and semimodules}%
{Research supported by NRC grant RGPIN-143068-05.}
\label{wag-abs}
The concept of modulo\"\i d over a dio\"\i d has been 
introduced in \cite{kn: GM77}.
These algebraic structures have been considered in the context of 
production systems \cite{kn:CDQV85}, 
computer systems \cite{kn:LS83}, network systems \cite{kn:Carre71},
or more generally for the modeling and analysis of discrete event systems \cite{kn:Bac92}, \cite{kn:Kol-Mas97}.
For G. Gondran and M. Minoux (\cite{kn: GM77}), 
a modulo\"\i d over a dio\"\i d is the algebraic structure  
left invariant under 
the action of a matrix $\,A\,$ with entries in a dio\"\i d (the ``space'' 
of proper ``vectors'' of $\,A$).  This structure is also very similar to 
that of band-space over a belt of R.A. Cunninghame-Green \cite{kn:CG79}.

 The problem of solving linear equations  of the type $Ax=Bx$ in 
the max-algebra has been considered 
by many authors (cf  \cite{kn:Carre71}, and \cite{kn:BH84}, 
where additional references may be found). 
In \cite{kn:BH84} the authors show how to compute all 
solutions to a system of linear equations 
over a totally ordered idempotent semifield.
\markboth{F.~Faye, M.~Thiam, L.~Truffet, E.~Wagneur}{Max-plus cones and
semimodules}

In \cite{kn:BSS07} the authors consider subsets of the 
positive cone  $\R_{+}^n$ endowed 
with the $\max$ operator as the first 
composition law, together with $``+''$ as the second composition law. 
They show how to relate 
subsets of $\R_+^n$ to the concept of 
generating vectors and bases defined in \cite{kn:W91a}.

 Since the early years, the terminology evolved, and the
concept of idempotent semimodule over an idempotent semiring 
(or semifield) has  emerged as the 
counterpart of that of vector space or, more generally, 
of module over a ring.

The most general definition of a (finite dimensional) idempotent semimodule $M$ 
is the following one:
take  two matrices $A,B$ of finite size with coefficients in a semifield $S$, consider the inequalities $Ax\leq Bx$,
and then define $M$ as the set of solutions to this set of inequalities. 
Another way to define (or to represent) a 
(finite dimensional) semimodule over $S$ is to give
its basis,  for example as the (independent) columns of a 
matrix with coefficients in $S$, 
It is then natural to ask how to get from one representation to the other.

The aim of this paper is to study 
$n$-dimensional semimodules over a 
completely ordered and complete idempotent semifield  
defined by a pair $(A,B)$.

 Recall that an idempotent semigroup 
$(S, \vee)$ is ordered by the relation 
$s,t\in S\,,\ s\leq t\ \iff s\vee t=t$.   
If $(S,\vee, \0)$ is an idempotent 
semigroup with the neutral element $\0$, 
then $\0$ is the least element of $S$, 
since for every $s\in S, s\vee \0 = s$.
We will assume here that the set of scalars $S$ 
is an idempotent completely ordered semifield, 
which is complete, i.e. $S$ is totally
ordered and complete, with least element $\0$, 
endowed with two composition laws: $\vee$, and $\cdot$ such  that :\\
i) $(S,\vee)$ is an idempotent commutative mono\"{i}d, with neutral element $\0$. \\
ii) $(S\setminus\{  \ \0\},   \cdot)$  is an abelian group --- hence $(S\setminus\{\0\},\cdot ,\leq )$ is an $\ell$-group --, with neutral element written $1$.\\
 iii) $\cdot$ is distributive over $\vee$,\\
iv) $ \0$ is absorbent (for every $s\in S, \0\cdot s= \0$).

 An idempotent semifield is also called a dio\"{i}d \cite{kn:Bac92}.

Next we give a brief summary of our talk.
In Part\ 1, we state some general results on semimodules defined by a pair $(A,B)$. In particular, we 
give an explicit formula for the semimodule of solutions to a single 
equation $a_i x\leq b_i x$ in terms of the
coefficients $a_{ij}$, and $b_{ij}$ of $a_i$ and $b_i$ (here and in the sequel, scalar product is denoted
by concatenation). 
Take a permutation $\sigma_i$ from
the symmetrical group $S_n$
such that, for $j=1,\,\dots ,k\,,\  a_{i,\sigma_i(j)}\leq b_{i,\sigma_i(j)}$, 
while $b_{i,\sigma_i(j)} < a_{i,\sigma_i(j)}\,,\ j=k+1,\dots n$.
Let $J_k(i)=\{ j\, (1\le j\le n) \vert \ a_{i,\sigma_i(j)} \le b_{i,\sigma_i(j)}\}$.
Clearly for every $j\in J_k(i)$, $e_{\sigma_i(j)}$ is a solution to $a_i x \le b_i x$,  where
the $e_j$'s are the elements of  the canonical basis of $S^n$.\\
Also, for every $j\in J_k(i)\,,\ \ell\notin J_k(i)$, we have:
\begin{equation*}
\begin{split}
&a_{i,\sigma_i(j)} \vee b_{i,\sigma_i(j)}= 
a_{i,\sigma_i(j)}\vee a_{i,\sigma_i(\ell)}(a_{i,\sigma_i(\ell)}^{-1}b_{i,\sigma_i(j)})=\\
=&b_{i,\sigma_i(j)} = b_{i,\sigma_i(j)}(\1 \vee b_{i,\sigma_i(\ell)}a_{i,\sigma_i(\ell)}^{-1})=
b_{i,\sigma_i(j)}\vee b_{i,\sigma_i(\ell)}(a_{i,\sigma_i(\ell)}^{-1}b_{i,\sigma_i(j)}).
\end{split}
\end{equation*}
Hence
$e_{\sigma_i(j)}\vee (a_{i,\sigma_i(\ell)}^{-1}b_{i,\sigma_i(j)})e_{\sigma_i(\ell)}$  also satisfies 
$a_i x \le b_i x$.


We have proved the following statement.
\begin{proposition}
If $b_i\not< a_i$, then the set of solutions to $a_ix\le b_ix$ is a semimodule $M_i$ 
generated by the $k\,(n+1-k)$ vectors given, for every $j\in J_k(j)$, by:\\
$e_{\sigma_i(j)}$, and $e_{\sigma_i(j)}\vee (a_{i,\sigma_i(\ell)}^{-1}b_{i,\sigma_i(j)})e_{\sigma_i(\ell)}\,,\ \ell\not\in J_k(j)$.
\end{proposition}
The semimodule $M_i$ is generated by the columns of $V_i$ given by the concatenation over $J_k(i)$ of the matrices  
\begin{equation*}
\begin{split}
&V_j(i) = [\ e_{\sigma_i(j)} \vert 
e_{\sigma_i(j)} \vee (a_{i,\sigma_i(\ell_1)}^{-1}b_{i,\sigma_i(j)})e_{\sigma_i(\ell_1)}\vert\\ 
&e_{\sigma_i(j)} \vee (a_{i,\sigma_i(\ell_2)}^{-1}b_{i,\sigma_i(j)})e_{\sigma_i(\ell_2)}\vert\cdots 
\vert e_{\sigma_i(j)} \vee (a_{i,\sigma_i(\ell_{n-k})}^{-1}b_{i,\sigma_i(j)})e_{\sigma_i(\ell_{n-k)}}\  ],
\end{split} 
\end{equation*}
where $\{ \ell_1,\ell_2,\dots ,\ell_{n-k}\} = \{\ell\vert\,\ell \notin J_k(i)\}$.

In Part 2, we give a geometric interpretation of the results  of Part 1. 
Let $M_p$ stand for the semimodule generated by the solutions to 
 $a_p x\ \le \ b_p x$, with $\sigma_p\in S_n$, and 
$J_q(p) = \{ r\ (1\le r\le n)\vert a_{\sigma_p(r)}\le b_{\sigma_p(r)}\}$. We have the following statement. 

\begin{theorem}
For $1\le k\le n-1$, we have $M_i\cap M_p\ne\{ 0\}$  iff one of the following conditions holds\\
i) $J_k(i)\bigcap J_q(p)\ne \emptyset$\\
ii) $ \bigvee\limits_{\ell \in J_q(p)} a_{i\sigma_p(\ell)}b_{p,\sigma_p(\ell)}^{-1}\ge 
\bigwedge\limits_{j\in J_k(i)} a_{p,\sigma_i(j)}^{-1}b_{i,\sigma_i(j)}$.
\end{theorem}

In Part 3, we solve the system of inequalities $Ax\leq Bx$ 
using combinatorial decomposition of the
inequalities $a_i x\leq b_i x$.\\
In particular, we show that $a_i x\leq b_i x$ is equivalent to 
a series of inequalities :\\
$b_{i1}^{-1}\bigl( a_{i1}x_{i1}\vee( a_{i2}\vee b_{i2})x_2\vee \ldots \vee(a_{in}\vee b_{in})x_n\bigr)\leq x_{1}$ or\\
$b_{i2}^{-1}\bigl( (a_{i1}\vee b_{i1})x_{1}\vee a_{i2}x_2 \vee (a_{i3}\vee b_{i3})x_3\vee \ldots \vee(a_{in}\vee b_{in})x_n \bigr)\leq x_{2}$ or\\
\hglue1cm $\ldots\ldots$\qquad  or \\
$b_{in}^{-1}\bigl( (a_{i1}\vee b_{i1})x_{1}\vee \ldots \vee a_{in}x_n \bigr)\leq x_{n}$.

The first inequality may be written in matrix form (add the trivial inequalities $x_2\leq x_2\,,\ 
\ldots ,x_n\leq x_n$) as:
$P_{i1}x\leq x$, with 
$$
P_{i1}= 
\left[ \begin{array}{ccccc} b_{i1}^{-1}a_{i1} & b_{i1}^{-1}(a_{i2}\vee b_{i2}) & 
\cdot & \cdot & b_{i1}^{-1}(a_{in}\vee b_{in})\\
\0 & \1 &\0 & \cdot & \0\\
\0 & \0 &\1  & \cdot & \0\\
\cdot & \cdots &\cdot &\cdot &\cdot \\
\0 & \cdot  &\cdot  & \cdot & \1
 \end{array}\right]
$$
It is well-known that $P_{i1}x\leq x\ \iff \ P_{i1}^*x = x$. 
This equation  has a nontrivial solution iff 
$a_{i1}\leq b_{i1}$, and in this  case the solutions are given 
by the columns of $P^*_{i1}$.
\markboth{V.V.~Fock, A.B.~Goncharov}{Max-plus cones and
semimodules}

Proceeding similarly for each line, we get $P^*_2,\dots ,P^*_n$, 
and all solutions lie in the concatenation 
$[P_{i1}\vert P_{i2}\vert \ldots \vert P_{in}]$.

Then we look at all  the intersections of the type 
$P_{ik}\cap P_{j\ell}$, etc.  in a combinatorial way, 
and devise an algorithm for the solution to these 
systems of inequalities.

Finally, we give a complete description, both algebraic and geometric, 
for cases $n=2,3$.

\renewcommand{\theproposition}{\thesection.\arabic{proposition}}
\renewcommand{\thetheorem}{\thesection.\arabic{theorem}}
\renewcommand{\theequation}{\arabic{equation}}
\setcounter{section}{0}
\setcounter{footnote}{0}
\nachalo{V.V.~Fock and A.B.~Goncharov}{Duality of cluster varieties}
\label{fok-abs}

Cluster variety is an algebraic variety (strictly speaking, a scheme)
defined by combinatorial data by explicit set of
coordinate charts and transition functions.
More precisely, for any collection of combinatorial data,
called \textit{seed} one associates three 
varieties $\mathcal A_{|\mathbb I|}$,  
$\mathcal X_{|\mathbb I|}$, and $\mathcal D_{|\mathbb I|}$. 
These varieties possess canonical pre-symplectic, Poisson
and symplectic structures, respectively. One defines
also a discrete group $\mathfrak D_{|\mathbb I|}$ acting on
all the three types of varieties and preserving the respective structures.
The manifolds $\mathcal X_{|\mathbb I|}$ and $\mathcal D_{|\mathbb I|}$
admit a quantisation (noncommutative deformation
of the algebra of functions) which is
also $\mathfrak D_{|\mathbb I|}$-invariant.

Varieties admitting cluster descriptions are simple Lie groups,
moduli spaces of Stokes parameters, moduli of flat connections
on Riemann surfaces, configuration spaces of flags,
Teichm\"uller spaces and their generalisations,
the spaces of measured laminations and some others.  
One of the important features of cluster varieties is that 
they are defined not only over a field but also over semifields 
(semigroups w.r.t. addition and groups w.r.t. the multiplication). 
For example, one can consider Teichm\"uller space, 
space of measured laminations and the space of flat 
$PSL(2,\mathbb F)$-connections over a surface $\Sigma$ as the 
same cluster manifold but defined over the semifield $\mathbb R_{>0}$ 
of positive real numbers, tropical semifield $\mathbb R^t$ (which is ordinary $R$ as a set with maximum for the addition operation and ordinary addition for the multiplication), and a field $\mathbb F$, respectively.
\markboth{V.V.~Fock, A.B.~Goncharov}{Duality of cluster varieties}

Let us give the precise definitions:

 A {\em cluster seed}, or just {\em seed}, ${\mathbf I}$ is a quadruple $(I, I_0, \varepsilon, d)$, where

i) $I$ is a finite set;

ii)  $I_0 \subset I$ is its subset; 

iii) $\varepsilon$ is a matrix $\varepsilon_{ij}$, where $i,j \in I$, such that $\varepsilon_{ij} \in {\mathbb Z}$ unless $i,j \in I_0$.

iv) $d = \{d_i\}$, where $i \in I$, is a set of positive integers, such that the matrix $\widehat{\varepsilon}_{ij}=\varepsilon_{ij}d_j$ is skew-symmetric. 

The elements of the set $I$ are called {\em vertices}, the elements of $I_0$ are called {\em frozen vertices}.  
The matrix $\varepsilon$ is called {\em exchange matrix},  the numbers $\{d_i\}$ are called {\em multipliers}, and the function $d$ on $I$ whose value at $i$ is $d_i$ is called {\it multiplier function}. We omit $\{d_i\}$ if all of them are equal to one, and therefore the matrix $\varepsilon$ is skew-symmetric, and we omit the set $I_0$ if it is empty. 

An isomorphism $\sigma$ between two seeds is a map ${\mathbf I}=(I, I_0,\varepsilon, d)$ and ${\mathbf I}'=(I', I'_0,\varepsilon', d')$ is an isomorphism of finite sets $\sigma:I\to I'$ such that $\sigma(I_0)=I'_0$, $d_{\sigma(i)}=d_i$ and $\varepsilon_{\sigma_i,\sigma_j}=\varepsilon_{ij}$. Observe that the automorphism group of a seed may be nontrivial. 

For a seed ${\mathbf I}$ we associate a torus ${\mathcal X}_{\mathbf I} = (\mathbb F^\times)^I$, called {\em ${\mathcal X}$-torus}, another torus ${\mathcal A}_{\mathbf I} = (\mathbb F^\times)^I$, called {\em ${\mathcal X}$-torus} and the third one ${\mathcal D}_{\mathbf I} = (\mathbb F^\times)^{I\times I}$ called {\em ${\mathcal D}$-torus} or a {\em double torus}. We denote the standard coordinates on these tori by $\{x_i|i\in I\}$, $\{a_i|i\in I\}$ and $\{y_i, b_i|i\in I\}$, respectively.

The $\mathcal X$-torus is equipped with the Poisson structure
\begin{equation}
\{x_i,x_j\}=\widehat{\varepsilon}_{ij}x_ix_j \label{Poisson-X}
\end{equation}

The $\mathcal A$-torus is equipped with the pre-symplectic structure (closed 2-form $\omega$ possibly degenerate)
\begin{equation}
\omega=\frac{1}{2}\sum_{i,j}\widehat{\varepsilon}_{ij}\frac{da_i\wedge da_j}{a_ia_j} \label{pre-symplectic}
\end{equation}

The $\mathcal D$-torus is equipped with the symplectic form
\begin{equation}
\omega_{\mathcal D}=\frac{1}{2}\sum_{i,j}\widehat{\varepsilon}_{ij}\frac{db_i\wedge db_j}{b_ib_j} + \sum_i d_i^{-1}\frac{db_i\wedge dy_i}{b_iy_i}\label{symplectic}
\end{equation} 
The inverse of this form is a nondegenerate Poisson structure which can be written as
\begin{equation}
\{y_i,y_j\}=\widehat{\varepsilon}_{ij}y_iy_j,\quad \{y_i,b_j\}=\delta^i_jd^iy_ib_j,\quad  \{b_i,b_j\}=0\label{Poisson-D}
\end{equation}

Observe that these sructures are constant in logarithmic coordinates.

Isomorphism between two $\mathcal X$-tori $\mathcal X_{\mathbb I}$ and  $\mathcal X_{\mathbb I'}$ is a map given in coordinates by $x_{\sigma(i)}=x_i$, where $\sigma$ is an isomorphism of the seeds. Observe that there are much less isomorphisms of $\mathcal X$-tori then just isomorphisms of the corresponding Poisson manifolds. Isomorphisms of $\mathcal A$- and $\mathcal D$-tori are defined analogously. 

There exist the following maps between the tori:

\begin{equation}\label{AX}
\mathcal A_{\mathbf I} \rightarrow \mathcal X_{\mathbf I}, \quad x_i =\prod_j a_j^{\varepsilon_{ij}};
\end{equation}
\begin{equation} \label{AAD}
\mathcal A_{\mathbf I}\times \mathcal A_{\mathbf I}  \rightarrow \mathcal D_{\mathbf I},\quad y_i =\prod_j a_j^{\varepsilon_{ij}}, \quad b_i =a_i/\tilde{a}_i,
\end{equation}
Here $\tilde{a}_i$ are coordinates on the second $\mathcal A_{\mathbf I}$-factor.
\begin{equation} \label{DX}
\mathcal D_{\mathbf I}\rightarrow \mathcal X_{\mathbf I}, \quad x_i=y_i,
\end{equation}
and
\begin{equation} \label{DXX}\mathcal D_{\mathbf I}\rightarrow \mathcal X_{\mathbf I}, \quad x_i=y_i\prod_jb_j^{\varepsilon_{ij}}.
\end{equation}

All the maps are compatible with the respective symplectic, 
pre-sym\-plec\-tic and Poisson structures.
 Namely the map (\ref{AX}) is a 
composition of the quotient by the kernel of the  pre-sym\-plec\-tic form 
and a symplectic map to a symplectic leaf. 
The map (\ref{AAD}) maps the symplectic form 
to the pre-symplectic one. The map (\ref{DX}) is Poisson, 
the map (\ref{DXX}) is anti-Poisson (Poisson with the opposite 
Poisson structure on the $\mathcal X$-torus). The maps (\ref{DX}) and (\ref{DXX}) 
are dual to each other in the sense on Poisson pairs.

Let ${\mathbf I}=(I, I_0,\varepsilon, d)$ and ${\mathbf I}'=(I', I'_0,\varepsilon', d')$  be two seeds, and $k\in
I-I_0$. A {\em mutation in the vertex $k$} is an isomorphism $\mu_k: I\rightarrow I'$ satisfying the following conditions:
\begin{itemize}
\item[$\bullet$] $\mu_k(I_0)=I'_0$,
\item[$\bullet$] $d'_{\mu_k(i)}=d_i$,
\item[$\bullet$] $\varepsilon'_{\mu_k(i)\mu_k(j)}=
\left\{ \begin{array}{lll} -\varepsilon_{ij} & \mbox{ if }  i=k \mbox{ or } j=k \mbox{ otherwise}\\
\varepsilon_{ij} & \mbox{ if } \varepsilon_{ik}\varepsilon_{kj}< 0\\
\varepsilon_{ij} + \varepsilon_{ik}|\varepsilon_{kj}| & \mbox{ if } \varepsilon_{ik}\varepsilon_{kj}\geq 0
\end{array}\right.$
\end{itemize}

Two seeds related by a sequence of mutations are called equivalent.

Mutations induce rational maps between the corresponding  
seed tori, which are  denoted by the same symbol $\mu_k$  and are given by the formulae

$$
x_{\mu_k(i)} = \left\{\begin{array}{lll} x_k^{-1}& \mbox{ if } & i=k \\
    x_i(1+x_k)^{\varepsilon_{ik}} & \mbox{ if } & \varepsilon_{ik}\geq 0\\
 x_i(1+(x_k)^{-1})^{\varepsilon_{ik}} & \mbox{ if } & \varepsilon_{ik}\leq 0
\end{array} \right..
$$
for the $\mathcal X$-torus,

$$
a_{\mu_k(i)} = 
\left\{\begin{array}{cl}
	\dfrac{\prod\limits_{j|\varepsilon_{jk}>0}a_j^{\varepsilon_{jk}}+\prod\limits_{j|\varepsilon_{jk}<0}a_j^{-\varepsilon_{jk}}}{a_k}& \mbox{ if } i=k\\ & \\
	a_i & \mbox{ if }  i\neq k
         \end{array}\right.
$$
for the $\mathcal A$-torus and
\begin{equation*}
\begin{split}
b_{\mu_k(i)} &= 
\left\{\begin{array}{cl}
	\dfrac{(1+x_k)^{-1}\prod\limits_{j|\varepsilon_{jk}>0}b_j^{\varepsilon_{jk}}+(1+(x_k)^{-1})^{-1}\prod\limits_{j|\varepsilon_{jk}<0}b_j^{-\varepsilon_{jk}}}{b_k}& \mbox{ if } i=k\\ & \\
	b_i & \mbox{ if }  i\neq k
         \end{array}\right.\\
y_{\mu_k(i)} &= \left\{\begin{array}{lll} 
y_k^{-1}& \mbox{ if } & i=k \\
y_i(1+y_k)^{\varepsilon_{ik}} & \mbox{ if } & \varepsilon_{ik}\geq 0\\
 y_i(1+(y_k)^{-1})^{\varepsilon_{ik}} & \mbox{ if } & \varepsilon_{ik}\leq 0
\end{array} \right..
\end{split}
\end{equation*}
for the $\mathcal D$-torus. 

Since in the sequel we shall extensively use compositions of mutations called also {\em cluster transformations} we would like to introduce a shorthand notation for them. Namely, we denote an expression $\mu_{\mu_{i}(j)}\mu_{i}$ by $\mu_j\mu_k$, $\mu_{\mu_{\mu_{i}(j)}\mu_{i}(k)}\mu_{\mu_{i}(j)}\mu_{i}$ by $\mu_k\mu_j\mu_i$, and so on.

Mutations have the following properties (valid for mutation of seeds as well as for mutations of respective tori):

\begin{itemize}
\item Every seed $\mathbf I=(I,I_0,\varepsilon, d)$ seed is related to other seeds by exactly $\sharp (I- I_0)$ mutations.
\item[$A_1$:] $\mu_i\mu_i=id$
\item[$A_1\times A_1$] If $\varepsilon_{ij}=\varepsilon_{ji}=0$ then $\mu_i\mu_j\mu_j\mu_i=id$.
\item[$A_2$:] If $\varepsilon_{ij}=-\varepsilon_{ji}=-1$ then $\mu_i\mu_j\mu_i\mu_j\mu_i=id$. (This is called the {\em pentagon relation}.)
\item[$B_2$:] If  $\varepsilon_{ij}=-2\varepsilon_{ji}=-2$ then $\mu_i\mu_j\mu_i\mu_j\mu_i\mu_j=id$.
\item[$G_2$:] If $\varepsilon_{ij}=-3\varepsilon_{ji}=-3$ then $\mu_i\mu_j\mu_i\mu_j\mu_i\mu_j\mu_i=id$.
\end{itemize}
By $id$ we mean here an isomorphism of the seeds or tori. Conjecturally all relations between mutation follow from these ones. 

Given a seed one can produce a $\sharp(I-I_0)$ seeds by mutations. Continuing this procedure one obtains a $\sharp(I-I_0)$-valent tree whose vertices are seeds (or seed tori) and edges are pairs of mutually inverse mutations. Obviously if we start from any other seed from the tree we obtain the same tree. Every two tori of the tree are related by exactly one composition of mutations. Call two points of two different tori equivalent if they are related by the composition of mutations. The cluster manifold (denoted by $\mathcal X_{|\mathbf I|}$, $\mathcal A_{|\mathbf I|}$ or $\mathcal D_{|\mathbf I|}$ depending on which kind of tori are used) is the affine closure of disjoint union of the tori quotiented by the equivalence relation.

Each particular seed tori can be considered as a coordinate chart of the corresponding cluster manifolds and compositions of mutations can be considered as transition functions between the charts.

Mutations respect the Poisson structure when acting on $\mathcal X$ tori,
pre-symplectic structure when acting on $\mathcal A$-tori and symplectic 
when acting on $\mathcal D$-tori. Thus the 
cluster manifolds $\mathcal X_{|\mathbf I|}$, $\mathcal A_{|\mathbf I|}$ 
and $\mathcal D_{|\mathbf I|}$ acquire the respective structures. 
(In fact the formula for mutation of the matrix $\varepsilon$ 
can be considered as a corollary of this property 
and the mutation formulae for, say, $\mathcal X$-tori). 

Mutations commute with the maps (\ref{AX}),(\ref{AAD}),(\ref{DX}) and (\ref{DXX}) thus these maps are defined between the respective cluster varieties compatible with pre-symplectic, symplectic and Poisson structures thereof.

Mutations are rational maps with positive integral coefficients and thus the cluster manifold can be defined not only over a field but over any semifield as well. For semifields without -1 (like the semifields of positive real numbers or the tropical semifields) the mutations are isomorphisms and thus the whole manifold is isomorphic to every coordinate torus.

The symmetry group $\mathfrak D_{|\mathbf I|}$ of a cluster manifold permuting the seed tori is called the (generalised) {\em mapping class group} of the cluster manifold. The name comes from the case of Teichm\"uller space, when this group is the actual mapping class group. The group depends on the equivalence class of a seed only and is common for cluster manifolds of types $\mathcal X$, $\mathcal A$ and $\mathcal D$. Every sequence of mutations together with an isomorphism of the initial and the final seed gives an element of the mapping class group. Conversely, given a seed, every mapping class group element can be presented by a sequence of mutations starting from the given seed together with the isomorphism between the final seed and the initial one. Two sequences of mutations different by the relations $A_1$--$G_2$ correspond to the same mapping class group elements.

Consider the ring of algebraic functions on a cluster manifold in more details. The ring of algebraic functions $\mathcal O(\mathcal X_{\mathbf I})$ (resp. $\mathcal O(\mathcal D_{\mathbf I})$, $\mathcal O(\mathcal A_{\mathbf I})$) on every torus is the ring of Laurent polynomials of cluster variables.  This ring contains a subring of Laurent polynomials with integral coefficients $\mathcal O^\mathbb Z$ and a semiring of Laurent polynomials with positive integral coefficients $\mathcal O^{\mathbb Z}_{>0}$ also depending of course of the seed and of the type of the torus. A cluster transformation in general does not presereve the ring $\mathcal O$ since it is birational. The ring of algebraic functions on the whole cluster manifold is the intersection of inverse images of the rings $\mathcal O$ under all possible cluster transformations of a seed tori. In other words the ring $\mathcal O$ consists of Laurent polynomials which stay Laurent under all possible cluster transformations. The celebrated result of Fomin and Zelevinsky called {\em Laurent phenomenon} claims that for any cluster variety $\mathcal A_{|\mathbf I|}$ of type $\mathcal A$ the coordinate functions belong to the ring $\mathcal O^{\mathbb Z}$. The ring $\mathcal O$ contains a subring $\mathcal O^\mathbb Z$ and a subsemiring $\mathcal O^{\mathbb Z}_{>0}$. The latter is additively generated by Laurent polynomials from $\mathcal O$ with positive integral coefficients indecomposable into a sum of two such polynomials. Such elements of the semiring $\mathcal O^{\mathbb Z}_{>0}$ are called {\em irreducibles}. The main conjecture, proven for a sufficiently wide class of cluster manifolds describes the structure of the set of irreducible Laurent polynomials:
\begin{itemize}
\item[$\bullet$] The set of irreducible Laurent polynomials is a basis in the ring $\mathcal O$.
\item[$\bullet$] The set of irreducible Laurent polynomials for the cluster variety $\mathcal X_{|\mathbf J|}$ (resp. $\mathcal A_{|\mathbf I|}$, $\mathcal D_{|\mathbf I|}$) is canonically isomorphic to the set of points of the cluster variety $\mathcal A_{|\mathbf I|}(\mathbb Z^t)$ (resp. $\mathcal X_{|\mathbf I|}(\mathbb Z^t)$, $\mathcal D_{|\mathbf I|}(\mathbb Z^t)$).  
\end{itemize}
One can consider this property as a duality between cluster varieties of type $\mathcal X$ (resp. $\mathcal A$, $\mathcal D$) and the tropical cluster varieties of type $\mathcal A$, $\mathcal X$ and $\mathcal D$, respectively.

The correspondence between irreducible Laurent polynomials 
and points of the dual tropical variety is especially simple for the variety of type $\mathcal X$.  
In this case the coordinates 
of the corresponding point of the 
tropical variety are given by multidegree of the 
highest term of the corresponding Laurent polynomial.
{\tolerance=1000

}
\vspace{2mm}

{\scshape Example} 

Let us consider the simplest nontrivial example: the seed $\mathbf I=\{I,\varepsilon\}$ with $I=\{1,2\}$ and $\varepsilon_{12}=1$. There are exactly 5 isomorphism classes of seed tori equivalent to a given one, however all the five seeds are isomorphic, thus the mapping class group is $\mathbb Z/5\mathbb Z$. 

The simplest geometric meaning has the space $\mathcal X$. It is the space of 5-tuples of points $(p_1,\ldots,p_5)$ on the projective line $P^1$ such that $p_i\neq p_{i+1 \pmod 5}$ and modulo the automorphisms of $P^1$. The 5-tuple of coordinate systems on this space is numerated by triangulations of the pentagon with vertices $1,\ldots,5$. For every internal diagonal one associates the cross-ratio of the four points of the quadrilateral which this diagonal cuts into halves. Mutations correspond to removing a diagonal and replacing it by another one of the quadrilateral. The same variety over $\mathbb R_{>0}$ is the configuration space of 5-tuples of points on $\mathbb RP^1$ with prescribed cyclic order.

The $\mathcal A$-space is the space of collections of 10 nonvanishing vectors $v_1,\ldots,v_{10}$ in $\mathbb F^2$ equipped with a nonzero bivector $Vol$. The collections are considered up to the action of the group $SL(2,\mathbb F)$ of linear transformations preserving $Vol$ and subject to the relations $v_i=-v_{i+5\pmod{10}}$ and $v_i\wedge v_{i+1\pmod{10}}=Vol$. The map $\mathcal X_{|\mathbf I|}\to \mathcal X_{|\mathbf I|}$ is given by the obvious projection of $\mathbb F^2 - \{0\}\to P^1$. For the internal diagonal of the pentagon with ends $i$ and $j$ one associates the coordinate $v_i\wedge v_j/Vol$.
\markboth{St\'ephane Gaubert}{Duality of cluster varieties}
The $\mathcal D$ variety is the space of flat $SL(2,\mathbb F)$ connections on a sphere with 5 different points on the equator removed with parabolic monodromy around these points. Consider the associated vector bundle and choose a monodromy invariant section about each singular points. Then trivialise the bundle over the northern hemisphere. The five chosen sections give five vectors $v_1,\ldots,v_5$ in $\mathbb F^2$. The same procedure over the southern hemisphere gives five vectors $w_1,\ldots,w_5$ in another copy of $\mathbb F^2$. Given a triangulation of the pentagon we associate to every internal diagonal two coordinates $x$ and $b$. The coordinate $x$ is just the cross ratio of four points in $P^1$ defined by the vectors $v_i$ standing at the corners of the quadrilateral cut by the diagonal (just like for the $\mathcal X$-space). The coordinate $b$ is given by $b=(v_i\wedge v_j)/(w_i\wedge w_j)$, where $i$ and $j$ are the ends of our diagonal. The two projections to the $\mathcal X$ variety are obviously given by projectivising the collections of vectors $\{v_i\}$ and $\{w_i\}$, respectively. The same manifold over $\mathbb R_{>0}$ can be identified with the space of complex structures on a sphere with five punctures on the equator. 

Given a triangulation of the pentagon one can describe the basis of the ring $\mathcal O^\mathbb Z$ of the 
corresponding $\mathcal X$-variety explicitly as a set of Laurent polynomials 
$P_{\mathsf a,\mathsf b}(x,y)$ of two variables $x,y$ parameterised by two integers  $\mathsf a,\mathsf b$ 
as follows:
\newpage
\begin{equation*}
 P_{\mathsf a,\mathsf b}(x,y)=\left\{\begin{array}{lll}
x^{\mathsf a}y^{\mathsf b}&\mbox{ if }& \mathsf a\leq 0, \mathsf b\leq 0\\
x^{\mathsf a}y^{\mathsf b}(1+x^{-1})^{-\mathsf b}&\mbox{ if }& \mathsf a\leq 0, \mathsf b\geq0\\
x^{\mathsf a}y^{\mathsf b}(1+x^{-1})^{-\mathsf b}(1+y^{-1}+x^{-1}y^{-1})^{\mathsf a}&\mbox{ if }& \mathsf a\geq 0, \mathsf b\leq 0\\
x^{\mathsf a}y^{\mathsf b}(1+y^{-1})^{\mathbb b}(1+y^{-1}+x^{-1}y^{-1})^{\mathsf a - \mathsf b}&\mbox{ if }& \mathsf a\geq \mathsf b\geq 0\\
x^{\mathsf a}y^{\mathsf b}(1+y^{-1})^{\mathsf a}&\mbox{ if }& \mathsf b\geq \mathsf a\geq 0\\
\end{array}\right.
\end{equation*}
One can easily check that this set of Laurent polynomials is invariant under 
simultaneous mutation of the variables $x,y$ and of the variables $\mathsf a, \mathsf b$.

\renewcommand{\theequation}{\thesection.\arabic{equation}}
\setcounter{section}{0}
\setcounter{footnote}{0}
\nachaloe{St\'ephane Gaubert}{From max-plus algebra to 
\mbox{non-linear} \mbox{Perron-Frobenius} \mbox{theory:} an 
\mbox{approach} to \mbox{zero-sum repeated games}}%
{This work was partially supported by the joint RFBR/CNRS grant 
05-01-02807}
\label{gau-abs}
This talk is based essentially on two joint works, with Akian and
Nussbaum~\cite{agn07},
on the one hand, and with Akian and Lemmens~\cite{agl07}, on the other
hand.
{\sloppy

}
\section{Introduction}
The analysis of zero-sum repeated games by the dynamic programming
method classically leads to studying discrete time dynamical systems of the form
\begin{align}\label{e-dyn}
v(k,\cdot)= f(v(k-1,\cdot)) 
\end{align}
where the map $f$ is order preserving. 
Here, $v(k,\cdot)$ is the {\em value function}, which associates
to any initial state the value of the corresponding
game in horizon $k$. The map $f$ is the ``one day'' dynamic programming
operator.
The case
of a finite state space is already interesting. Then, denoting
by $n$ the number of states, we may identify the value function
to a vector in $\R^n$, and the map $f$ to
a self-map of $\R^n$.

The explicit form of $f$ depends on the details
of the game. However, the map $f$ may be written abstractly
as: 
\begin{align}
f(v)=\inf_{\sigma} \sup_{\pi } r^{\sigma \pi}
+P^{\sigma \pi } v \enspace,\label{e-nonlin}
\end{align}
where the infimum is taken over the strategies $\sigma$ of the first player
and the supremum is taken over the strategies $\pi$
\markboth{St\'ephane Gaubert}{Max-plus approach to zero-sum repeated games}
of the second player, $r^{\sigma \pi}\in \R^n$
is a vector of payments, and $P^{\sigma \pi}$
is a $n\times n$ nonnegative matrix. 
In the case of games with undiscounted payoff,
the matrices $P^{\sigma \pi}$ are stochastic.
When there is a positive discount rate,
or when the game may halt with a positive
probability, the matrices $P^{\sigma \pi}$ are substochastic.
These (sub)-stochasticity properties imply that $f$ is nonexpansive
in the sup-norm, meaning that
\[
\|f(v)-f(w)\|_\infty \leq \|v-w\|_\infty \enspace .
\]
The relevance of the order and nonexpansiveness properties
to control and game problems has been brought to light by several
authors, see in particular~\cite{crandall,kol92,rosenbergsorin,neymansurv}.

The dynamic programming operators~\eqref{e-nonlin} may
be thought of as generalizations of linear positive maps
in several different ways. 
First, linear maps
of the form $x\mapsto Px$, where $P$ is a (sub)stochastic matrix,
correspond to the zero-player case, in which every
player has only one possible strategy, if we assume
in addition that the payments are zero.
Another special situation concerns the deterministic one player
case, in which one of the two players has only one possible
strategy, and the entries of the matrices $P^{\sigma\pi}$ are
only $0$ or $1$. Then, $f$ becomes an affine map
over the min-plus or max-plus semiring. 
Further connections with Perron-Frobenius theory
become apparent when using a familiar tropical instrument,
the ``logarithmic/exponential'' glasses
or ``dequantization'', as in~\cite{litvinov00,viro}.
This leads us to consider the conjugate
map:
\[
g=\exp\circ f\circ \log 
\]
where $\log$ denotes the map from the interior of
the standard positive cone $\R_+^n:=\set{x\in \R^n}{x\geq 0}$
to $\R^n$ which does $\log$ entrywise,
and $\exp:=\log^{-1}$. Then, the map $g$ is an order preserving
self-map of the interior of $\R_+^n$, and it is positively homogeneous
 or subhomogeneous of degree one, meaning that $g(tx)=tg(x)$
 or $g(tx)\leq tg(x)$ for all scalars $t\geq 1$ and
for all $x\in \operatorname{int}\R_+^n$. Such maps
belong to non-linear Perron-Frobenius theory, which
deals with the nonlinear extensions of the spectral theory
of positive linear maps. We refer the reader to~\cite{nuss88}
for a general account of this topic and for references.

I will present some results concerning zero-sum games,
which have been obtained by exploiting methods from non-linear Perron-Frobenius
theory with a max-plus or tropical point of view. The main results are taken
from the two joint works~\cite{agn07,agl07}.

\section{Nonlinear spectral radius of dynamic programming operators}
The classical notion of spectral radius has been extended to nonlinear
maps in several ways~\cite{Nuss-Mallet}. We assume here that $g$
is a continuous positively homogeneous of degree one map leaving
invariant a (closed, convex, pointed) cone $C$ in a Banach space $X$,
and that $g$ preserves the order induced by $C$, which is such that
$x\leq y$ if $y-x\in C$. {\em Bonsall's cone spectral radius} of $g$
is defined by:
\[ \tilde{r}_C(g) = \lim_k \|g^k\|_C^{1/k}
\]
where, for all continuous, positively homogeneous of degree
one self-maps $h$ of $C$,
\[
\|h\|_C:=\sup_{x\in C\setminus \{0\} } \frac{\|h(x)\|}{\|x\|} \enspace .
\]
Another natural definition of the spectral radius
arises when considering the nonlinear eigenproblem:
\[
g(u)=\lambda u
\]
where the {\em nonlinear eigenvector} $u$ belongs to $C\setminus \{0\}$,
and the {\em nonlinear eigenvalue} $\lambda$ is a nonnegative number.
The {\em cone eigenvalue spectral radius}, $\hat{r}_C(g)$, 
is by definition the maximal nonlinear eigenvalue
$\lambda$. Under some assumptions involving measures of non-compactness,
it has been shown in~\cite{Nuss-Mallet}
that $\tilde{r}_C(g)=\hat{r}_C(g)$. 
Other useful notions of spectral radius, which
coincide with the previous ones under reasonable assumptions,
are studied in~\cite{Nuss-Mallet}.

We shall discuss here the related notion of {\em Collatz-Wielandt number},
which is obtained by considering super-eigenvectors in
the interior of the cone instead of eigenvectors in the closed cone:
\[ \bar{r}_C(g):=\inf\{\lambda>0\mid \exists u\in \operatorname{int} C,\; g(u)\leq \lambda u\}\enspace.
\] 
The term ``Collatz-Wielandt number'' arises from Wielandt's
proof of the finite dimensional Perron-Frobenius theorem, in which the 
same formula is seen to characterize the Perron root of an irreducible
nonnegative matrix.

The main result of~\cite{agn07} shows that $\tilde{r}_C(g)=\bar{r}_C(g)$,
when the cone $C$ is normal, and when $g$ satisfies some compactness
assumptions. 

We apply these tools to dynamic programming
maps of the form~\eqref{e-nonlin},
when the payments $r^{\sigma\pi}$ are $0$, so that $f$
is positively homogeneous of degree one.
Under some standard assumptions (compactness of the action
spaces, continuous dependence of the reward and transition probabilities
in the actions), which imply that the infimum and supremum
are attained in~\eqref{e-nonlin} for all $v$,
it is shown in~\cite{agn07} that
\begin{align}
\tilde{r}_C(f)=\hat{r}_C(f)=\bar{r}_C(f) =
\inf_\sigma\sup_\pi r(P^{\sigma\pi})\label{e2}
\end{align}
where $C=\R_+^n$, and $r(P^{\sigma\pi})$ denotes the Perron root
of $P^{\sigma\pi}$.

We derive from the previous result an explicit formula
for the geometrical convergence rate of the iterates of the
dynamic programming map $f$, this time with
nonzero payments $r^{\sigma\pi}$.
To this end, we use the notion of {\em subdifferential}.
Maps of the form~\eqref{e-nonlin} may not be differentiable,
in particular, if the action spaces are finite, $f$ is piecewise affine.
However, $f$ may often be assumed to be semidifferentiable, meaning
that for all $v$ and $h$, we can write $f(v+h)=f(v)+f'_v(h)+o(\|h\|)$,
where $f'_v$, the {\em semidifferential} of $f$ at point $v$,
is a continuous positively homogeneous of degree one map,
which is defined uniquely by the latter property.

When $f$ is of the form~\eqref{e-nonlin}, it can be shown that
under fairly general assumptions, the semidifferential $f'_v$
at point $v$ exists, and is given by:
\[
f'_v(h)= \inf_{\sigma\in \Sigma^*(v)} \sup_{\pi\in \Pi^*(v,\sigma) } 
P^{\sigma \pi } h \enspace,\label{e-nonlin2}
\]
where $\Sigma^*(v)$ denote the set of policies $\sigma$
which attain the infimum in~\eqref{e-nonlin}, 
and for all $\sigma$, $\Pi^*(v,\sigma)$ denotes the
set of strategies $\pi$ which attain the supremum
in the internal term in~\eqref{e-nonlin}. (We need not
assume that the inf and sup commute.)

We show that if $f$ has a fixed point $v\in \R^n$, 
and if 
\[ \rho:= \max(r_C(f'_v),r_{-C}(f'_v))<1
\enspace ,
\]
then any orbit of $f$ converges to $v$ at a geometric 
rate which is bounded from above by $\rho$ (this bound
is tight). We eventually get the following explicit convergence rate:
\[
\rho= 
\max(\inf_{\sigma\in \Sigma^*(v)} \sup_{\pi\in \Pi^*(v,\sigma) } 
r(P^{\sigma \pi }),
\sup_{\sigma\in \Sigma^*(v)} \inf_{\pi\in \Pi^*(v,\sigma) } 
r(P^{\sigma \pi })) \enspace .
\]
\section{Order preserving convex functions}
The techniques of the previous section are mostly appropriate
when $f$ has a unique fixed point,
perhaps up to an additive or multiplicative constant.

Therefore, a basic problem is to give a complete 
description of the fixed point set of the map~\eqref{e-nonlin}.
As a partial answer, a precise description of the set of stable fixed points
is given in~\cite{agl07}, when the map $f$ is convex (this corresponds
to the one player case). This extends our earlier results~\cite{AG-Gau}
which concerned the undiscounted case. Here, we do not require any more
the matrices $P^{\sigma\pi}$
in~\eqref{e-nonlin} to be (sub)stochastic. In other words,
we allow the possibility of a negative
discount rate. Despite its apparently
unphysical nature, negative discount
is of practical interest: for instance, the study
of static analysis problems by abstract interpretation~\cite{goubault2}
leads to fixed point problems involving maps which are always
order preserving but not necessarily nonexpansive in some norm.
Another motivation may come for fixed point problems for polynomials
with positive coefficients, leading to maps
like:
\begin{align*}
f_i(v)=\log(\sum_{j\in \mathbb{N}^n}a_{ij}\exp( j\cdot v)),
\end{align*}
where for all $1\leq i\leq n$, $(a_{ij})_{j\in \mathbb{N}^n}$ 
is an almost zero family of real nonnegative numbers.
\markboth{S.~Gaubert, S.~Sergeev}{Max-plus approach to zero-sum repeated games}

A convenient notion of stability, in the present setting,
is the following one: we say that a fixed
point $v$ is \emph{$\star$-stable} if every orbit of the
semidifferential $f'_v$ is bounded from above. It can be checked
that a Lyapunov stable fixed point is $\star$-stable.

Recall that a (communication) {\em class}
of a nonnegative matrix $P$ is by definition a strongly
connected component of the digraph of $P$.
We say that a class is {\em critical}
if the corresponding principal submatrix of $P$ has Perron root $1$.
The {\em critical graph} of $P$ is the union of the subgraphs
of the graph of $P$ induced by the critical classes.
If $v$ is a $\star$-stable fixed point of $f$, we define the {\em critical
graph} of $f$, $G^c(f)$ to be the union of the critical graphs of the matrices
in the subdifferential
\[
\partial f(v):= \{P\mid f(w)-f(v)\geq P(w-v) ,\forall w\} \enspace .
\]
Of course, $\partial f(v)$ depends on $v$,
but $G^c(f)$ is independent of the choice of the $\star$-stable fixed
point $v$. The critical nodes of $f$ are defined to be the nodes
of $G^c(f)$. 

We show in~\cite{agl07} that a $\star$-stable fixed point is uniquely
determined by its restriction to the set of critical nodes.
Moreover, the restriction to the critical nodes
allows us to identify the set of $\star$-stable fixed points of $f$ to
a convex inf-subsemilattice of $\R^p$, where $p$ is bounded by the number
of critical nodes. Some dynamical information, including
a characterization of the possible lengths of ``$\star$-stable'' periodic
orbits of $f$, is also derived in~\cite{agl07}.

The results of~\cite{AG-Gau,agl07} concern the one player
case but have applications to the two player case.
Indeed, the representation of the fixed point set
has been used in~\cite{cochet-cras} to design a policy iteration
algorithm for zero-sum two player stochastic games. It allows
one to handle ``degenerate'' iterations, in which 
the policies which are selected yield dynamic programming
maps with several fixed points.
Some other applications of these ideas, to static analysis of programs,
are presented in~\cite{goubault,goubault2}.

\newcommand{\etalchar}[1]{$^{#1}$}

\setcounter{footnote}{0}
\setcounter{section}{0}
\nachaloe{St\'{e}phane Gaubert and Serge{\u{\i}} Sergeev}%
{Cyclic projectors and separation \mbox{theorems} in \mbox{idempotent
semimodules}}%
{Supported by the RFBR grant 05-01-00824 and
the joint RFBR/CNRS grant 05-01-02807.}
\label{gau-ser-abs}

\section{Introduction}
\label{s:intro}
In an idempotent semiring, there is a canonical order
relation, for which every element is ``nonnegative''.
Therefore, idempotent semimodules
have much in common with the semimodules over the semiring of nonnegative 
numbers, that is, with {\em convex cones} \cite{Roc:70}. 
One of the first results 
based on this idea is the separation theorem for
convex sets over ``extremal algebras'' proved
by K.~Zimmermann in \cite{KZim-77}. 
Generalizations of this result were obtained in a work by 
S.N.~Samborski{\u{\i}}
and G.B.~Shpiz \cite{SS-92} and in works by G.~Cohen, J.-P.~Quadrat, 
I.~Singer, and
the first author \cite{CGQ-04}, \cite{CGQS-05}. 

The main result of this paper, Theorem~\ref{maxsep2}, 
shows that in the setting of finite-dimensional
semimodules over max-plus semiring, {\em several}
closed subsemimodules which do not have common nonzero
points can be separated from each other.
This means that for each of these subsemimodules, we can
select an idempotent halfspace containing it, in such a way
that these halfspaces also do not have common nonzero points.

Even in the case of two semimodules, 
this statement has not been proved in the idempotent literature. Indeed, 
the earlier separation theorems deal with the separation
of a point from an (idempotent) convex set or semimodule,
rather than with the separation of two convex sets or semimodules.
\markboth{S.~Gaubert, S.~Sergeev}{Cyclic projectors and separation theorems}

In order to prove the main result, Theorem~\ref{maxsep2}, 
we investigate the spectral
properties of idempotent cyclic
projectors. By idempotent cyclic projectors
we mean finite compositions of certain nonlinear projectors on
idempotent semimodules.
The continuity and homogeneity of these nonlinear projectors
enables us to apply to their compositions, i.e.\ 
to the cyclic projectors, a result of R.D. ~Nussbaum\cite{Nus-86}  (non-linear
Perron-Frobenius theory).  We also show  
that the orbit of an eigenvector of
a cyclic projector maximizes a certain
objective function. We call this maximum
the Hilbert value of semimodules, as it is a natural
generalization of Hilbert's projective metric, and 
characterize the spectrum of cyclic projectors 
in terms of these Hilbert values
(Theorem \ref{C-G}).

Our main results apply to the finite-dimensional semimodules over 
max-plus semiring.
Some of our results still hold in a 
more general setting, see Sect.~\ref{s:gensep}. However, the separation of 
several semimodules in such a generality remains an open question.

The results of this paper are presented as follows. 
Sect.\ \ref{s:pointsep} describes the main assumptions, and some preliminary notions and facts
that will be used in the paper. 
Sect.\ \ref{s:gensep} is devoted to the results obtained in the
most general setting, with respect to the assumptions of 
Sect.\ \ref{s:pointsep}. The main results are obtained 
in Sect.\ \ref{s:maxsep}. They include separation of several 
semimodules and characterization of the spectrum of cyclic projectors.

The proofs of our results are contained in \cite{GS-07}, which is an extended version of this text.

\section{Preliminaries}
\label{s:pointsep}
We recall that a semiring (essentially, a ring without subtraction) is called {\em idempotent}, if 
its addition $\oplus$ is idempotent:$a\oplus a=a$. The order relation mentioned above is given by
$a\oplus b=b\Leftrightarrow a\leq b$. An example of idempotent semiring that will be important to us is 
$\Rmaxm$, it is the set of nonnegative numbers $\R^+$ equipped with operations $a\oplus b:=\max(a,b)$ and $a\odot b=a\times b$.
It is isomorphic to the max-plus semiring (the set $\R\cup-\infty$ equipped with $a\oplus b:=\max(a,b)$ and
$a\odot b=a+b$). The ``spaces'' over semirings are called semimodules.

An idempotent semiring or an idempotent 
semimodule will be called {\em $b$-complete}, 
following \cite{LMSz-01<IFA>}, if it is
closed under the sum (i.e.\ the supremum) of any subset
bounded from above, and  if the multiplication  
distributes over such sums. 
We shall consider
semirings $\cK$
and semimodules $\cV$ over $\cK$ 
that satisfy the following assumptions:

$(A0)$: the semiring $\cK$ is a $b$-complete idempotent 
semifield, 
and the semimodule $\cV$ is a $b$-complete semimodule
over $\cK$;

$(A1)$: for all elements $x$ and $y\ne\0$ from $\cV$,
the set $\{\lambda\in\cK\mid \lambda y\leq x\}$ is bounded from above.

Note that both assumptions are true for the semimodules $\cK^I$
 of $\cK$-valued
functions on a set $I$, where $\cK$ is a $b$-complete semifield. 

Assumptions $(A0,A1)$ imply that the operation
\begin{equation}
\label{def:div}
x/y=\max\{\lambda\in\cK\mid \lambda y\leq x\}.
\end{equation}
is defined for all elements $x$ and $y\ne\0$ from $\cV$.

\begin{definition}
\label{d:subsem}
A subsemimodule $V$ of $\cV$ is a {\em $b$-(sub)semimodule}, 
if $V$ is closed under the sum of any of its subsets bounded
from above in $\cV$.
\end{definition}

Let $V$ be a $b$-subsemimodule of
the semimodule $\cV$. Consider the operator $P_V$ defined by
\begin{equation}
\label{projector}
P_V(x)=\max\{u\in V\mid u\leq x\},
\end{equation}
for every element $x\in\cV$. Here we use ``$\max$'' to
indicate that the least upper bound belongs to the set. The operator $P_V$
is a {\em projector} onto the subsemimodule $V$, as $P_V(x)\in V$ for
any $x\in\cV$ and $P_V(v)\in V$ for any $v\in V$. 

In idempotent geometry, the role of halfspace is played
by the following object.

\begin{definition}
\label{d:halfs}
A set $H$ given by
\begin{equation}
\label{halfs0}
H=\{x\mid u/x \geq v/x\}\cup\{\0\}
\end{equation}
with $u,v\in \Rmaxmn$, $u\leq v$, 
will be called {\em (idempotent) halfspace}.
\end{definition}

Any halfspace is a semimodule.
If $\cV=\cK^n$, an $n$-dimensional semimodule over $\cK$, and all coordinates of $u$ and $v$ are nonzero,  
then we have that
\begin{equation}
\label{reghalfspace}
H=\{x\mid\bigoplus_{\ldotsn} 
x_i u_i^{-1}\leq
\bigoplus_{\ldotsn} x_i v_i^{-1}\}.
\end{equation}
 
The following theorem is a version of idempotent separation
theorems \cite{CGQ-04,CGQS-05}, see also
\cite{LMSz-01<IFA>}. 

\begin{theorem}
\label{separation} Let $V$ be a $b$-complete
subsemimodule of $\cV$ and let $u\in\cV$ be not in $V$. 
Then the set
\begin{equation*}
H=\{x\mid P_V(u)\odiv x\geq u\odiv x\}\cup\{\textbf{0}\}
\end{equation*}
contains $V$ but not $u$.
\end{theorem}

For any subsemimodule $V$ and $y\in\cV$, we denote 
\begin{equation}
\label{vy}
V^y=\{x\in V\mid y\odiv x>\textbf{0}\}.
\end{equation}
It is a subsemimodule of $V$.

\begin{definition}
\label{def:arch} 
{\rm A vector $x$ is called {\em archi\-med\-ean}, if $x/ y>\textbf{0}$ for
all $y\in\cV$. A subsemimodule of $\cV$
is called {\em archi\-med\-ean}, if it contains archi\-med\-ean vectors.
A halfspace $H$ defined by (\ref{halfs0}) will be called {\em archi\-med\-ean}
if both $u$ and $v$ are archi\-med\-ean.}
\end{definition}

Obviously, Def.~\ref{def:arch}
makes sense only under

$(A2)$: The semimodule $\cV$ has an archimedean vector.

This assumption is true in particular 
for the semimodules of type $\cK^n$.
In these semimodules we have that $y\odiv x>\textbf{0}$ 
if and only if the support of
$x$, i.e. the set $\text{supp}(x)=\{i\mid x_i\ne\textbf{0}\}$,
is a subset of $\text{supp}(y)$ (the support of $y$). 
In this case $V^y$ has the form
\begin{equation}
\label{vm}
V^M=\{x\in V\mid \text{supp}(x)\subseteq M\},
\end{equation}
for some index set $M$.
A vector in $\cK^n$ is archi\-med\-ean if and only if it
is positive.  Regular halfspaces in this case are given by \eqref{reghalfspace}.

\section{General results}
\label{s:gensep}

We shall study cyclic projectors, that is,
compositions of projectors
\begin{equation*}
P_{V_k}\cdots P_{V_1},
\end{equation*}
where $V_1,\ldots,V_k$ are $b$-subsemimodules of $\cV$. We assume $(A0,A1)$,
which means in particular that $\cK$ is an idempotent semifield.
For the notational convenience, 
we will write $P_t$ instead of
$P_{V_t}$. We will also adopt a convention of cyclic numbering
of indices of projectors
and semimodules, so that
$P_{l+k}=P_l$ and $V_{l+k}=V_l$ for all $l$.

\begin{definition}
\label{d:hvalue}
Let $x^1,\ldots,x^k$ be nonzero elements of $\cV$. 
The value
\begin{equation*}
d_{\Hilb}(x^1,\ldots,x^k)=(x^1\odiv x^2)\ (x^2\odiv x^3)\ldots
(x^k\odiv x^1).
\end{equation*}
will be called {\em the Hilbert value} of $x^1,\ldots,x^k$.
\end{definition}

\msn The Hilbert value of
two vectors $x^1,x^2$
was studied in \cite{CGQ-04}. For two comparable
vectors in $\Rmaxmn$, that is, for two vectors with common
support $M$ it is given by
\begin{equation*}
d_{\Hilb}(x^1,x^2)=\min_{i,j\in M}
(x^1_i(x^2_i)^{-1}x^2_j (x^1_j)^{-1}),
\end{equation*}
so that $-\log(d_{\Hilb}(x^1,x^2))$ coincides with 
Hilbert's projective metric
\begin{equation*}
\delta_{\Hilb}(x^1,x^2)=\log(\max_{i,j\in M}
(x^1_i(x^2_i)^{-1}x^2_j (x^1_j)^{-1}))=-\log(d_{\Hilb}(x^1,x^2)).
\end{equation*}

\begin{definition}
\label{d:hvalue-sem}
The {\em Hilbert value} of $k$ subsemimodules 
$V_1,\ldots,V_k$ of $\cV$ is defined by 
\begin{equation*}
d_{\Hilb}(V_1,\ldots,V_k)=\sup_{x^1\in V_1,\ldots,x^k\in V_k}
d_{\Hilb}(x^1,\ldots,x^k)
\end{equation*}
\end{definition}

We establish two results on the spectrum of cyclic projectors and on their
iterations.
\begin{theorem}
\label{maximum}
Suppose that
the operator $P_k\circ\ldots\circ P_1$ has an 
eigenvector $y$ with eigenvalue $\lambda$, and define 
$\Bar{x}^i=P_i\circ\ldots\circ P_1 y$. Then
\begin{equation*}
\lambda=d_{\Hilb}(V^y_1,\ldots,V^y_k)=d_{\Hilb}(\Bar{x}^1,\ldots,\Bar{x}^k).
\end{equation*}
\end{theorem}

\begin{theorem}
\label{dhincr}
For any sequence of nonzero vectors
$\{x^i,\ i=1,\ldots\}$ such that $x^1\in V_1$ and
$x^i=P_i x^{i-1}$ for $i=2,\ldots$, the Hilbert value
$d_{\Hilb}(x^{l+1},\ldots,x^{l+k})$ is nondecreasing with $l$.
\end{theorem}
The following is an extension of Theorem \ref{separation}, under 
assumptions $(A0-A2)$.
\begin{theorem}
\label{gensep}
Suppose that $V_1,\ldots, V_k$ are $b$-closed semimodules
and that $P_k\circ\ldots\circ P_1$ has an archimedean eigenvector $y$
with nonzero eigenvalue $\lambda$. The following are equivalent:
\begin{itemize}
\item[(1)] there exists an archimedean vector $x$ and a scalar
$\mu<\textbf{1}$ such that
\begin{equation*}
P_k\circ\ldots\circ P_1 x\leq\mu x;
\end{equation*}
\item[(2)] for all $i=1,\ldots,k$ there exist regular halfspaces
$H_i$ such that $V_i\subseteq H_i$ and $\cap_i H_i=\{\textbf{0}\}$;
\item[(3)] $\cap_i V_i=\{\textbf{0}\}$;
\item[(4)] $\lambda<\textbf{1}$.
\end{itemize}
\end{theorem}

\section{Projectors and separation in max algebra} 
\label{s:maxsep}
 In $\Rmaxmn$, it is natural to consider semimodules that are closed
in the Euclidean topology. One can easily show that such semimodules are 
$b$-semimodules.
Theorem 3.11 of \cite{CGQS-05} implies that projectors onto
closed subsemimodules of $\Rmaxmn$ are continuous.

In order to relax the assumption concerning archimedean
vectors in Theorem \ref{gensep}, we use some results from
nonlinear spectral theory, that we next recall. 
By Brouwer's fixed point theorem, a continuous
homogeneous operator $x\mapsto Fx$ that maps $\Rpn$ to itself
has a nonzero eigenvector. This allows us to define
the nonlinear spectral radius of $F$, 
\begin{equation}
\label{specrad}
\rho(F)=\max\{\lambda\in \R_+\mid \exists x\in 
(\R_+^n)\setminus 0,\; Fx=\lambda x\} \enspace .
\end{equation}
Suppose in addition that $F$ is isotone, then the maximum in (\ref{specrad}) is attained
and we can use the following nonlinear
generalization of the Collatz-Wielandt formula for the spectral
radius of a nonnegative matrix.

\begin{theorem}
\label{t:Nussbaum}
{\rm (R.D.~Nussbaum, Theorem~3.1 of \cite{Nus-86})}
For any isotone, homogeneous, and continuous map $F$ from
$\Rpn$ to itself, we have:
\begin{equation*}
\rho(F)=\inf_{x\in(\Rp\backslash\{0\})^n}\max_{1\leq i\leq n} [F(x)]_i x_i^{-1}.
\end{equation*}
\end{theorem}
This result implies that 
the spectral radius of such operators is isotone:
if $F(x)\leq G(x)$ for any $x\in\Rpn$, then 
$\rho(F)\leq\rho(G)$.

\msn As the projectors on subsemimodules of $\Rmaxmn$ are isotone, 
homogeneous and continuous, so are their compositions, i.e.\ cyclic
projectors. Consequently, we can apply Theorem~\ref{t:Nussbaum} 
to them. This allows us to refine the general results from the previous section.
The following result refines Theorem \ref{maximum} (the spectrum of 
cyclic projections).\markboth{T.~Grbi\'c, E.~Pap}{Cyclic projectors and separation theorems}
\begin{theorem}
\label{C-G}
Let $V_1,\ldots,V_k$ be closed semimodules
in $\Rmaxmn$. Then the Hilbert value of $V_1,\ldots,V_k$ is the
spectral radius of $P_k\circ\ldots\circ P_1$.
Every eigenvalue of $P_k\circ\ldots\circ P_1$ 
is equal to $d_{\Hilb}(V_1^M,\ldots,V_k^M)$ for some $M$. 
Conversely, every such
Hilbert value is an eigenvalue of $P_k\circ\ldots\circ P_1$.
\end{theorem}
The following result refines Theorem \ref{gensep} (separation).
\begin{theorem}
\label{maxsep2}
Suppose that $V_i,\ i=1,\ldots,k$ are closed
semimodules of $\Rmaxmn$, and that $\cap_i V_i=\{\textbf{0}\}$.
Then there exist archimedean halfspaces $H_i,\ i=1,\ldots,k$ such that
$V_i\subseteq H_i$, for $i=1,\ldots,k$, and $\cap_i H_i=\{\textbf{0}\}$.
\end{theorem}
A particular case of Theorem \ref{maxsep2} is the following
separation theorem for two semimodules.
\begin{theorem}
\label{2sep}
Suppose that $U$ and $V$ are two closed
max cones, and that $U\cap V=\textbf{0}$. 
Then there exists an archimedean halfspace $H_U$, which contains
$U$ and does not intersect with $V$,
and there exists an archimedean halfspace $H_V$, which contains $V$ and does not
intersect with $U$.
\end{theorem}

\setcounter{footnote}{0}
\setcounter{section}{0}
\nachaloe{T.~Grbi\'c and E.~Pap}{Pseudo-weak convergence of the
\mbox{random} sets \mbox{defined} by a \mbox{pseudo integral}
based on \mbox{non-additive measure}}{Partially supported by  
the Project MNZ\v ZSS $144012,$ grant of MTA HTMT,
French-Serbian project "Pavle Savi\'c", and by the project
''Mathematical Models for Decision Making under Uncertain Conditions
and Their Applications'' supported by Vojvodina Provincial
Secretariat for Science and Technological Development.}
\label{grb-abs}

\section{Introduction}
The weak convergence of sequence of probability measures is the
main subject for a large class of limit theorems in the
probability theory.  In the classical
probability theory, it works with $\sigma$-additive measures and the
Lebesgue integral (\cite{Bill68}).  Several conditions equivalent
to the weak convergence are provided by the theorem of Portmanteau
(\cite{Bill68}). The main aim of this paper is to prove a
Portmanteau-type theorem, with capacity functionals instead of
probability measures, and with the general pseudo integral instead
of the Lebesgue integral.

 Since the convergence in
distribution of sequence of random closed sets on $\R$  can
be tricky,  it is often more appropriate to study the convergence
of the corresponding sequence of capacity functionals. In this
paper we study the convergence of sequences of random closed sets
on $\R$ by looking at the convergence of the corresponding
sequence of capacity functionals. Theoretical foundations of the
theory of random sets, as generalization of random variables, were
layed down by Kendall (\cite{Ken74}) and Matheron (\cite{Mat75}).
Recall that random closed sets are random elements on the space of closed
subsets of $\R$.
\markboth{T.~Grbi\'c, E.~Pap}{Pseudo-weak convergence of random sets}

Our paper is organized as follows. Sect. 2 contains some
preliminary notions, such as pseudo-operations and general pseudo
integral \cite{BeMeVi02,Pa95,WaKl92}. In  Sect. 3, we recall some
basic notions and definitions from the theory of random sets
(\cite{Gout90,Ken74,Mat75,Molc94,Molc05}). The main results of
this paper, also contained in Sect. 3, are concerned with  the
weak convergence of sequence of random closed sets, i.e., of the
corresponding sequence of capacity functionals with respect to the
general pseudo integral.
\section{Preliminary notions}\label{preliminary}
Following {\cite{BeMeVi02,Pa95,Pa02}}, we recall the notions of
pseudo operations and general pseudo integral. Let $\leq$ be 
the
total order on $[0, \infty].$
\begin{definition} A binary operation $\oplus : [0,
\infty]^2 \rightarrow [0, \infty] $ is called {\it
pseudo-addition} if  the following properties are satisfied:

\textrm{(A1)} $a \oplus b = b \oplus a\;\;\;$ {\em(commutativity)}

\vspace{-0.5mm}

 \textrm{(A2)} $a \leq a' \wedge b \leq b'
\Rightarrow a \oplus b \leq a' \oplus b'\;\;\;$
{\em(monotonicity)}

\vspace{-0.5mm}

 \textrm{(A3)} $(a \oplus b) \oplus c = a \oplus (b
\oplus c)\;\;\;$ { \em(associativity)}

\vspace{-0.5mm}

\textrm{(A4)} $a \oplus 0 = a\;\;\;$  {\em(neutral element)}

\vspace{-0.5mm}

\textrm{(A5)} $a_n \rightarrow a \wedge b_n \rightarrow b
\Rightarrow (a_n \oplus b_n) \rightarrow a \oplus b \;\;\;$
{\em(continuity)}
\end{definition}
\begin{example}\label{ex1-pseudo-addition} The following
operations are pseudo-additions (\cite{BeMeVi02,Pa95,Pa02}): \quad
\textrm{(i)} $x \oplus y = g^{-1} (g(x)) + g(y),$ where
$g:[0,\infty]^2 \rightarrow [0,\infty]$ is an increasing
bijection;

\textrm{(ii)} $x \oplus y = \max (x,y)$ (note that this operation
is idempotent).
\end{example}
\begin{definition} For a given pseudo-addition $\oplus$ {\em pseudo-difference} is the binary operation $\;\ominus :
[0,\infty]^2 \rightarrow [0,\infty]$ given by

$\hfill a \ominus b = \inf \{x \in [0,\infty] : b \oplus x \geq a
\}.\hfill$
\end{definition}
\begin{example}\label{ex1-pseudo-difference} Obviously,  $a \ominus b =
0$ for $a \leq b$ and $a \ominus b> 0$ for $a > b$, see
\cite{BeMeVi02,KMP}. For pseudo-additions from Example
\ref{ex1-pseudo-addition} and $a>b$ corresponding
pseudo-differences are

 \textrm{(i)}  $a \ominus b = g^{-1} (g(a) -
g(b));$ \quad \quad \quad \quad \textrm{(ii)} $a \ominus b = a.$
\end{example}
\begin{definition} For a given pseudo-addition $\oplus$ the {\em
pseudo-mul\-ti\-pli\-ca\-tion} is a binary operation $\; \odot :
[0,\infty]^2 \rightarrow [0, \infty]$ such that the following
conditions are satisfied

 \textrm{(M1)} $a \odot 0 = 0 \odot b =
0\;\;\;$  {\em(zero element)}

\vspace{-0.5mm}

\textrm{(M2)} $a \leq a' \wedge b \leq b' \Rightarrow a \odot b
\leq a' \odot b'\;\;\;$  {\em(monotonicity)}

\vspace{-0.5mm}

\textrm{(M3)} $(a \oplus b) \odot c = (a \odot c) \oplus (a \odot
c)\;\;\;$ {\em(right distributivity)}

\vspace{-0.5mm}

\textrm{(M4)} $a \odot \mathbf{1} = \mathbf{1} \odot a = a\;\;\;$
{\em(unit element)}

\vspace{-0.5mm}

\textrm{(M5)} $a \odot (b \odot c) = (a \odot b) \odot c \;\;\;$
{\em(associativity)}

\vspace{-0.5mm}

\textrm{(M6)} $a_n \rightarrow a \wedge b_n \rightarrow b
\Rightarrow (a_n \odot b_n) \rightarrow a \odot b \;\;\;$
{\em(continuity)}
\end{definition}
\begin{example}\label{ex1-pseudo-multiplication} \textrm{(i)} For the pseudo-addition from Example
\ref{ex1-pseudo-addition} \textrm{(i)}, define
pseudo-mul\-ti\-pli\-ca\-ti\-on by $a \odot b   = g^{-1}
(g(a)g(b)).$

 \textrm{(ii)} For the pseudo-addition from Example
\ref{ex1-pseudo-addition} \textrm{(ii)}, one of  the possible
pseudo-mul\-ti\-pli\-ca\-ti\-ons is $a \odot b = a + b,$ see
\cite{Lit05,MaSa92}.
\end{example}
\vspace{-3mm}
 The algebraic structure $([0, \infty], \oplus,
\odot)$ is a {\em semiring}.
\vspace{2mm}

 Let $\Omega$ be an abstract space, $\mathcal{A}$ a
$\sigma$-algebra of subsets of $\Omega$ and $m:\mathcal{A}
\rightarrow \R$ a non-decreasing set function with
$m(\emptyset) = 0.$ We consider the space $(\Omega, \mathcal{A}, m)$
and a family of $\mathcal{A}$-measurable functions $f:\Omega
\rightarrow [0, \infty],$ denoted by $\mathcal{F}.$ A {\em
simple function} is a measurable function $s:\Omega \rightarrow
[0,\infty]$ whose range is finite. Let $Rang(s) =
\{a_1,a_2,\ldots,a_k\}$ such that $0< a_1 < a_2 < \ldots < a_k,$
and $A_i \cap A_j = \emptyset$ for $i \neq j$. The {\em standard
$\oplus$-step representation} of a simple function $s$ is given by
$s = \bigoplus\limits_{i=1}^k b (c_i^*, C_i^*),$ where $c_1^*=a_1,
\; c_2^* = a_2 \ominus a_1,\; \ldots , c_m^* = a_m \ominus
a_{m-1},$ $C_i^* = \bigcup\limits_{j=i}^m A_i$ and
$b:\Omega\rightarrow [0,\infty]$ is a {\em basic function} of the
form $ b(c_i^*, C_i^*) (\omega) = \left\{
\begin{array}{ll}   c_i^*, & \omega \in C_i^*, \\[3mm]
0, & \omega \notin C_i^*.
\end{array}
\right.$
\begin{definition}\label{GeneralFuzzyIntegral}
\textrm{(i)} {\em The general pseudo integral of a simple
function} $s$ with the standard $\oplus$-step representation
 is given by

 $\hfill \int\nolimits^\oplus s \odot \; d m  = \bigoplus\limits_{i=1}^m
c_i^* \odot m ( C_i^*).\hfill$

\textrm{(ii)} The {\em general pseudo integral} of a measurable
function $f \in \mathcal{F}$ is given by

$\hfill \int\nolimits^\oplus f \odot \; d m = \sup \{
\int\nolimits^\oplus s \odot \; d m : s \in \mathcal{S}_f
\},\hfill$

\noindent where $\mathcal{S}_f$ is the family of all simple
function $s$ such that $s \leq f.$
\end{definition}
The general pseudo integral has the following properties:

\textrm{(i)}  $ \int\nolimits^\oplus b(c,C) \odot \; d m = c \odot
m (C). $

\textrm{(ii)} $f \leq g \Rightarrow \int\nolimits^\oplus f \odot
\; d m \leq \int\nolimits^\oplus g \odot\; d m.$

\textrm{(iii)} For the pseudo characteristic function $\chi_A :
\Omega \rightarrow [0, \infty]$ of a set $A \subset \Omega,$
defined by $ \chi_A (x) = \left\{
\begin{array}{ll}
\mathbf{1}, & x \in A, \\ 0, & x \notin A,
\end{array}
\right.$ we have

 $\hfill \int\nolimits^\oplus \chi_A \odot \; d m = m
(A).\hfill$
\newpage
\section{Weak convergence of the sequence of capacity
functionals}\label{randomsets}
\subsection{Random closed sets and capacity
functionals}\label{UvodniPojmoviRS} We start with a short overview
of the theory of random closed sets
(\cite{Gout90,Ken74,Mat75,Molc94,Molc05,Ngu04a,NguBouch04}).
 Denote collections of closed, open and compact subsets of $\R$ by
$\mathcal{F}$, $\mathcal{O}$ and $\mathcal{K}$, respectively. A
very important role in the theory of random closed sets is played
by collections of closed sets $\mathcal{F},$ and its
sub-collections $ {\mathcal{F}}_G =\{ F \in \mathcal{F} : F \cap G
\neq \emptyset\},\; G \in {\mathcal O},\;$ and ${\mathcal{F}}^K
=\{ F \in {\mathcal F} : F \cap K = \emptyset\},\; K \in {\mathcal
K}.$ Collections $\{\mathcal{F}_G : G \in \mathcal{O}\}$ and
$\{\mathcal{F}^K : K \in \mathcal{K}\}$ generate a topology $\tau
(\mathcal{F})$ on $\mathcal{F}$. This topology is known as {\em
hit-or-miss-topology}. The collection $\mathcal{F}$ endowed with
the hit-or-miss topology is a compact, separable and Hausdorff
space (\cite{Mat75}). Taking countable unions and intersections of
 open sets of the topological space $(\mathcal{F},
\tau(\mathcal{F}))$, we obtain a $\sigma$-field $\Sigma
(\mathcal{F})$.
\begin{definition}\label{randomset}A {\em random closed set} ${\mathrm S}$ is a measurable
mapping from the probability space $(\Omega, \mathcal{A},
\mathrm{P})$ into the measurable space $(\mathcal{F}, \Sigma
(\mathcal{F}))$.\end{definition}
 A random closed set ${\mathrm S}$ generates a probability distribution $\mathbf{P}_{\mathrm S}$  in
the following way

$\hfill \mathbf{P}_{\mathrm S} (A) = {\mathrm P} (\{\omega \in
\Omega: {\mathrm S } (\omega) \in A \}) = \mathbf{P}_{\mathrm S}
({\mathrm S} \in A),\;\; \mbox{for all}\;\; A \in \Sigma({\mathcal
F}).\hfill$
\begin{definition}\label{capacity} For a random closed set $\mathrm{S}$ its {\em capacity functional}
${\mathrm T}_{\mathrm S} : \mathcal{K} \rightarrow [0,1]$ for $K
\in \mathcal{K}$ is defined by

$\hfill {\mathrm T}_{\mathrm S} (K) = \mathbf{P}_{\mathrm S}
({\mathrm S} \in \mathcal{F}_K) = \mathbf{P}_{\mathrm S} ({\mathrm
S} \cap K \neq \emptyset).\hfill$
\end{definition}
 The capacity functional ${\mathrm T_\mathrm{S}}$ is defined on $\mathcal{K},$ and it can
be extended onto the family $\mathcal{P}$ of all subsets of
$\R$. A subset $M \subset \R$ is called {\em
capacitable} if the following equality ${\mathrm T_\mathrm{S}} (M) =
\sup \{{\mathrm T_\mathrm{S}}(K) : K \in \mathcal{K}, K \subset M
\}$ is true. All Borel sets $B$ are capacitable (\cite{Molc05}). For
a given random closed set ${\mathrm{S}}$, and a sequence of random
closed sets $\{{\mathrm{S}}_n\}$ corresponding capacity functionals
will be denoted by ${\mathrm T}$ and $\{{\mathrm T}_n\},$
respectively.
\subsection{$(\oplus,\odot)$-weak
convergence}\label{Konvergencija}
\begin{definition}\label{Definicija} A sequence of capacity functionals $\{{\mathrm T}_n\}$ {\em
$(\oplus,\odot)$-weak converges} to a capacity functional $\mathrm
T$ (shortly, {\em pseudo-weak converges}) if and only if for each
continuous, bounded
 function $f: \R \rightarrow [0, \infty]$ we have that $\lim\limits_{n \rightarrow \infty} \int\nolimits^\oplus f \odot
d {\mathrm T}_n = \int\nolimits^\oplus f \odot d {\mathrm T}.$
\end{definition}
We have proved in \cite{GrPa07}the following three theorems.
\begin{theorem}\label{Teorema1}  If a sequence of capacity
 functionals $\{{\mathrm T}_n \}$ pseudo-weak converges to
 capacity functional ${\mathrm T}$, then $
\limsup\limits_n {\mathrm T}_n (F) \leq {\mathrm T} (F)$ for all
closed sets $F \subseteq \R$.
\end{theorem}
\begin{theorem}\label{Teorema2} If a sequence of capacity
 functionals $\{{\mathrm T}_n \}$ pseudo-weak converges to
 capacity functional ${\mathrm T}$, then $\;\; \liminf\limits_n {\mathrm T}_n (G) \geq {\mathrm T}
(G)\;\;$ for all open sets $G \subset \R.$
\end{theorem}
\begin{theorem}\label{Teorema3} If for a sequence of capacity
 functionals $\{{\mathrm T}_n \}$ and for all closed sets $F$ holds $\rm{(A)}\;\limsup\limits_n
{\mathrm T}_n (F) \leq {\mathrm T} (F)$ and for all open sets $G$
holds $\rm{(B)}\; \liminf\limits_n {\mathrm T}_n (G) \geq {\mathrm
T} (G),$ then $\{{\mathrm T}_n \}$ pseudo-weak converges to
capacity functional ${\mathrm T}.$
\end{theorem}
\begin{corollary}\label{Primer1} For a random closed set $\mathrm S$ and a sequence of random closed sets $\{{\mathrm
S}_n\},$ which are defined in the following way: ${\mathrm S} = \{
{\mathrm X}\}$ and ${\mathrm S}_n = \{ {\mathrm X}_n \}$, where
$\mathrm{X}$ is a random variable and $\{{\mathrm X}_n\}$ is a
sequence of random variables, the $(\oplus,\odot)$-weak
convergence is equivalent to the weak convergence (with respect
to continuous, bounded function $f: \R \rightarrow [0,
\infty]$).
\end{corollary}
\begin{proof}
For ${\mathrm S} = \{ {\mathrm X}\}$ and ${\mathrm S}_n = \{
{\mathrm X}_n \}$, we have that ${\mathrm T}(K) = {\mathrm P} ({\mathrm
X} \in K)$ and ${\mathrm T}_n(K) = {\mathrm P} ({\mathrm X}_n \in
K)$ (\cite{Gout90}),  where $\mathrm T$ and  ${\mathrm T}_n$ are
capacity functionals of random
 sets  $\mathrm S$ and  $\mathrm{S}_n$, respectively. Since for each Borel set $B$ we have that
${\mathrm T}(B) =\sup \{{\mathrm T} (K): K \in
\mathcal{K}, K \subset B \}$, it follows that  ${\mathrm T} (B) = {\mathrm
P} ({\mathrm X} \in B)$. For all  $n \in \mathbb{N}$ we have that
${\mathrm T}_n (B) = {\mathrm P} ({\mathrm X}_n \in B)$. Suppose
that $\{{\mathrm S}_n\}$ $(\oplus,\odot)$-weak converges to $\mathrm
S$. Then by Theorem \ref{Teorema1}, $\limsup\limits_n {\mathrm P} ({\mathrm X}_n \in F) \leq
{\mathrm P} ({\mathrm X} \in F)$ for all closed sets $F$.  From the classical theorem of
Portmanteau (\cite{Bill68})  we obtain that $\int f d {\mathrm P}_n
\rightarrow \int f d {\mathrm P },$ i.e., that the sequence of random
closed sets $\{{\mathrm S}_n\}$ weak converges to $\mathrm S.$
\markboth{Oleg V.~Gulinsky}{Pseudo-weak convergence of random sets}

\noindent The weak convergence of the sequence of probability
measures for any open set $G$ implies that $ \liminf\limits_n {\mathrm
P} ({\mathrm X}_n \in G) \geq {\mathrm P} ({\mathrm X} \in G)$, and
for any closed set $F$ it implies that $ \liminf\limits_n {\mathrm P}
({\mathrm X}_n \in F) \leq {\mathrm P} ({\mathrm X} \in F).$ Then,
by Theorem \ref{Teorema3},  the sequence of random closed sets
$\{{\mathrm S}_n\}$ $(\oplus,\odot)$-weak converges to $\mathrm
S.$
\end{proof}
\begin{remark}
\textrm{(i)} For the special case described by corollary \ref{Primer1},
the capacity functional reduces to  the probability measure and
then Theorem \ref{Teorema3} can be proved by taking into  the
account only one of the assumptions, $\textrm{(A)}$ or
$\textrm{(B)}.$

\textrm{(ii)} Weak convergence of the sequence of capacity
functionals with respect to Choquet integral is investigated in
\cite{Ngu04a}. Some equivalent conditions for the weak convergence
of the sequence of probability measures, induced by sequence of
random closed sets, are obtained in \cite{NguBouch04,Pa05}.

\textrm{(iii)} The results obtained in this paper will serve for
the investigation of further convergence properties of a sequence
of capacity functionals of the sequence of random closed sets,
based on the idempotent $ \sup $-measure and related integrals,
see (\cite{GrPa07,Pu01}.
\end{remark}

\setcounter{section}{0}
\setcounter{footnote}{0}
\nachalo{Oleg V.~Gulinsky}{The stationary phase method and large deviations}
\label{gul-abs}

Let $\{P_{\lambda}\}$  be
a family of probability measures on a measurable space
$(X,\mathcal{F})$ and let $I$ be a nonnegative function on $X$
with compact level sets. $\{P_{\lambda}\}$ obeys the large
deviation principle with a rate function $I$ if and only if
\begin{eqnarray*}
\lim_{\lambda\rightarrow \infty} \big[\int_{X} (g(x))^{\lambda}
P_{\lambda}(dx)\big]^{1/\lambda} = \sup_{x\in X} g(x)e^{-I},
\end{eqnarray*}
for all bounded continuous nonnegative functions $g$ on $X$
\cite{Bryc}, \cite{Puh1}.

 We  say that in this sense
$\{P_{\lambda}\}$ converges to an idempotent measure $\exp\{-I\}.$
The r.h.s. of the last display is called a {\bf sup - integral} or
{\bf idempotent integral} with respect to the idempotent measure
and defines rough logarithmic asymptotics of the Laplace method.
\markboth{Oleg V.~Gulinsky}{The stationary phase method and large deviations}

   In this report we discuss logarithmic asymptotics of the
   integral

\begin{eqnarray*}
J(\lambda)=\int_{X} \exp\{\imath\lambda u(x)\}(g(x))^{\lambda}
P_{\lambda}(dx),
\end{eqnarray*}
where $X = R^d$ (in what follows $R^1$ for the simplicity) and
$u$,$g$ $(g\geq 0)$ are smooth enough real-valued functions.

We consider this problem as a natural generalization of the
stationary phase method which imbeds the classical one in the
context of large deviations. The interest in the problem is
motivated by the slicing approximation approach to infinite
dimension oscillatory integrals as well (see, for example
\cite{Kumano}).

Our approach is based on the technique of an almost analytic
extension and follows the ideas of \cite{Melin} where the
classical method of stationary phase was extended to the case of
complex-valued phase function. The new difficulty in our problem
is the following. The function which plays the role corresponding
to the imaginary part of the phase function in \cite{Melin}, is
just the rate function $I$ defined asymptotically by the large
deviation principle.

   Nevertheless, we consider $ f(x)= u(x) + \imath I(x)$ as
complex-valued "phase function" and assume that $f(x)$ is
$C^{\infty}$ function in a neighborhood of the origin, which in
turn is a non-degenerate stationary point of $f$ with $I(0)=0$.

We introduce an almost analytic extension of $f$ as follows:
\begin{eqnarray*}
 f(z)=[u(x)\chi(y)
-I^{'}(x)y\chi(t_{1}y)- \frac{1}{2!}u''(x)y^{2}\chi(t_{2}y)+...]
\\ +\imath[I(x)\chi(y)+u'(x)y\chi(t_{1}y)-\frac{1}{2!}I''(x)y^{2}\chi(t_{2}y)+...]=
\\ u(x,y)+\imath v(x,y),
\end{eqnarray*}
where $\chi(y)\in C^{\infty}_0$ is equal to one in a neighborhood
of the origin and vanishes for $\mid y\mid \geq 1$. The numbers
$t_{k} \geq1$ are chosen sufficiently large so that the series
converges.

To examine the asymptotic behavior of the integral $J(\lambda)$, we
replace the integration along $R$ by the integration along a
suitable chain in the complex domain passing through the critical
point of $f(z)$. We show that on this chain the problem is reduced
to the standard variational principle of large deviations.

To fulfil the program following \cite{Melin}, we first find new
coordinates $\tilde{z}$ in $C$ for which $ f(z)-f(0)$ is a
quadratic form in $\tilde{z}$ . To this end, using Taylor's formula
we write
$$f(z)-f(0)=<z,R(z) z>/2 ,$$
where by definition $R(z)$ is an almost analytic function of  $z$.
Since all non-degenerate quadratic forms on $C$ are equivalent,
there is a linear transformation $A$ such that
$$A^{T}R(0)A=\imath\textbf{I}.$$
In turn, the equation
\begin{eqnarray*}
 \imath Q^{T}(z)Q^{T}(z)= R(z), \\ Q(0)=A^{-1},
\end{eqnarray*}
has a $C^{\infty}$ solution $Q(z)$ defined near the origin, since
the map $Q\rightarrow\imath Q^{T}Q$ is analytic with surjective
differential at $Q=A^{-1}$. Moreover, $Q$ is an analytic function
of $R(z)$ and therefore an almost analytic function of $z$.

The map $z\rightarrow \tilde{z}=\tilde{z}(z)=Q(z)z$ defines new
coordinates in the neighborhood of the origin and in this
coordinates we have
$$f(z)=f(0)+\imath <\tilde{z},\tilde{z}>/2 , $$
where $\tilde{z}=\tilde{x}+\imath\tilde{y}$ and
$<\tilde{z},\tilde{z}>=\tilde{x}^2 -\tilde{y}^2 +
2\imath<\tilde{x},\tilde{y}>$.

Since $I(x)\geq 0$ with $I(0)=0$, it follows that $\tilde{x}^2
-\tilde{y}^2 \geq 0$ on the tangent space at the origin and so
there is a $C^{\infty}$ function $\varphi$, defined in a
neighborhood of $0$, such that in the new coordinates $R$ is given
by the equation $\tilde{y}=\varphi(\tilde{x})$.

We are now in a position to define a family of chains $\Gamma_{s}$
in $C$ and examine the behavior of $f$ on them. Let
$\tilde{z}\rightarrow z=z(\tilde{z})$ be the inverse of the map
$z\rightarrow\tilde{z}(z)$. For $0\leq s\leq 1$, putting
$$\Gamma_{s} =\{z: z=z(\tilde{z}_s),  \tilde{z}_s = \tilde{z}_s
(\tilde{x}) =\tilde{x}+\imath \varphi(\tilde{x})s ,  \tilde{x}\in
R\}$$ one gets the estimate
$$ \texttt{Im} f(z(\tilde{z}_s))\geq C\mid \texttt{Im} z(\tilde{z}_s)\mid^2 .$$
To replace the integration, we first consider the chain
$\Gamma_{1}$ and note that in a small enough neighborhood of $z=0$
the integrals $J(\lambda)$ and
\begin{eqnarray*}
\int_{\Gamma_{1}} \exp\{\imath\lambda f(z)\}(g(z))^{\lambda}
\exp\{\lambda I(x)\chi(y)\} P_{\lambda}(dx)dy,
\end{eqnarray*}
where $g(z)$ is almost analytic extension of $g$, are equivalent(
we may assume that the support of $g$ w.r.t. $z$ belongs to a small
fixed neighborhood of the origin).

Finally we have to show that $\int_{\Gamma_{1}}$ differs from
$\int_{\Gamma_{0}}$ with a very small error. We are able to do
that with the help of Stokes's formula by the following arguments:
(1)$f(z)$ and $g(z)$ are almost analytical functions, $
(2)\texttt{Im} f(z(\tilde{z}_s))\geq C\mid \texttt{Im}
z(\tilde{z}_s)\mid^2 $.

Thus, it suffices compute the logarithmic asymptotics of the
integral\markboth{Dmitry Gurevich}{The stationary phase method and large deviations}
\begin{eqnarray*}
\int_{\Gamma_{0}} \exp\{\imath\lambda f(z)\}(g(z))^{\lambda}
\exp\{\lambda I(x)\chi(y)\} P_{\lambda}(dx)dy= \exp\{\imath\lambda
u(0)\}\times \\
 \int \exp\{-\lambda [I(0)+ \mid
\tilde{x}\mid^2/2]\}(g(z(\tilde{x}))^{\lambda} \exp\{\lambda
I(x(\tilde{x}))\} G(\tilde{x}) P_{\lambda}(d\tilde{x}),
\end{eqnarray*}
where $G(\tilde{x})=\texttt{det}(\frac{dz}{d\tilde{x}})$. One can
easily recognize that the asymptotics of the  last integral
coincides with the asymptotics of
\begin{eqnarray*}
\exp\{\imath\lambda u(0)\}\times
 \int (g(z(\tilde{x}))^{\lambda}  G(\tilde{x}) P_{\lambda}(d\tilde{x}).
\end{eqnarray*}
Thus we reduced the initial problem to the problem of large
deviations.

\newpage

\setcounter{footnote}{0}
\setcounter{section}{0}
\nachalo{Dmitry Gurevich}{Quantization with a deformed trace}
\label{gur-abs}


The standard quantization  scheme of a Poisson structure on a  variety $M$ consists
in the following. First,  one looks for an associative $\star$-product satisfying
 the so-called correspondence principle.
Existence of such a product is shown by Kontsevich. Second, one  represents the
constructed associative algebra $\ah$ in a linear (hopefully, Hilbert) space.
However, if the initial Poisson structure is not symplectic, such a
representation is usually   associated  to each symplectic leaf of
the bracket.

In the  80's the author considered some Poisson pencils whose quantization leads to
 "braided" algebras (cf. \cite{G1, G2} and the references therein). This means that in a sense
they are related  to a braiding, i.e.  a solution to the Quantum Yang-Baxter
equation (YBE) \markboth{Dmitry Gurevich}{Quantization with a deformed trace}
$$R_{12}R_{23}R_{12}=R_{23}R_{12}R_{23}, \,\,{\rm where} \,\,R_{12}=R\otimes I,\,\,  R_{23}=I\otimes R,$$
$V$ is a vector space
over the ground field $\K$ ($\R$ or $\C$), and
$R:V^{\otimes 2}\to V^{\otimes 2}$ is a linear operator. Such a braiding plays the role of the usual flip in all related
constructions and operations. In particular,  generalized Lie algebras and
their enveloping algebras were defined in this way. However, braidings entering their definitions were assumed to be
 involutive $(R^2=I)$.

Semiclassical counterpart of such braiding  is a classical r-matrix. Given
a Lie algebra $\gggg$. By a classical r-matrix we mean an element $r\in \bigwedge^2 \gggg$ satisfying
the classical analog of the YBE
$$[r_{12},r_{13}]+[r_{12},r_{23}]+[r_{13},r_{23}]=0,\,\,{\rm where}\,\,r_{12}=r\otimes 1, r_{23}=1\otimes r.$$
Let $\rho:\gggg\to \Vect(M)$ be a representation of the Lie algebra $\gggg$ into the vector fields space on a variety $M$.
It is clear that the operator
$$ f\otimes g\in \K[M] \rightarrow \{f,g\}_r=\circ\, \rho^{\otimes 2}(r)(f\otimes g)$$
where $\circ$ stands for the usual (commutative) product in
the  coordinate ring $\K[M]$ of the variety $M$ defines a Poisson bracket on it.

A typical example is $M=\gggg^*,\,\K[\gggg^*]=\Sym(\gggg)$. Given a classical r-matrix $r\in \bigwedge^2( \gggg)$ then
the bracket $\{\,,\,\}_r$ is compatible with the linear Poisson-Lie bracket $\{\,,\,\}_{PL}$, i.e. these brackets
generate a Poisson pencil
$$\{\,,\,\}_{a,b}=a\{\,,\,\}_{PL}+b \{\,,\,\}_r.$$
Moreover, each of them (and consequently, the whole Poisson pencil) can be restricted to any $G$-orbit $\O\subset \gggg^*$.
(The restriction of the PL bracket to the orbit $\O$ is called Kirillov-Kostant-Souriau bracket.)

A quantization of the Poisson-Lie bracket can be  realized in  different ways. We consider the enveloping algebra
$U(\gh)$ to be quantum counterpart of the bracket $\{\,,\,\}_{PL}$. Hereafter by $\gh$ we mean the Lie algebra with
the bracket ${\hbar}[\,,\,]$ where $[\,,\,]$ is the bracket of the Lie algebra $\gggg$ and ${\hbar}$ is a deformation
(quantization) parameter. As for the the KKS bracket its quantization can be realized as an appropriate quotient
of the algebra $U(\gh)$ and represented in a vector space in the spirit of the Kirillov orbit method.

In order to quantize the whole pencil $\{\,,\,\}_{a,b}$ or its restriction to an orbit $\O$ we apply
the following   result of Drinfeld. There exists an element $F_\nu\in U(\gggg)\hat{\otimes}U(\gggg)$ such that
\begin{align*}
&F_\nu=1+\nu\, r+\ldots, \\
 &F_\nu(X+Y,Z)\,F_\nu(X,Y)=F_\nu(X,Y+Z)\,F_\nu(Y,Z),\\
\intertext{and}
&(\epsilon\otimes 1)F_\nu=(1\otimes \epsilon)F_\nu=1,
\end{align*}
where $\epsilon$ is the counit in $U(\gggg)$.

By using this element ("quantor" according to Lychagin's terminology) it is possible to quantize the above Poisson
pencil and all other operators. Say, by equipping the algebra $\ah=U(\gh)$ with a new product
$$f\star_{{\hbar},\nu} g=\star_{\hbar}\, \rho^{\otimes 2}(F_\nu)(f\otimes g)$$
where the representation $\rho={\rm ad\, }$ is naturally extended to the algebra $U(\gh)$ and $\star_{\hbar}$ stands for the product
in this algebra we get a new associative   algebra denoted $\ahn$.

The aforementioned  braiding can be introduced via the element $F_\nu$. Namely, we put $R=F^{-1}_\nu \,P\,F_\nu$ where $P$
is the usual flip. It is clear that $R$ is involutive. Moreover, it is subject to the quantum YBE, i.e. it is a braiding.

By means of the quantor $F_\nu$ the category of finite dimensional modules of the algebra $\ah$ can be converted
into that  of $\ahn$-ones. This category is monoidal tensor rigid. Let us consider an object $V$
of this category and the corresponding object $\End(V)\cong V\otimes V^*$ of internal endomorphisms. There exists a map
$$\Tr_R: \End(V)\to\K$$
which is a deformation of the usual trace and is morphism in this category. Moreover, it is R-symmetric, i.e.
$$\Tr_R\,(X\circ Y)=\Tr_R\,\circ R(X\otimes Y)$$
where $\circ$ stands for the usual product in the algebra $\End(V)$.
In a sense it looks like a super-trace for which the role of $R$ is played by a super-flip.

So, by quantizing the above Poisson pencil and by considering representations of the quantum algebra we are forced to
replace the usual trace by its braided version. According to \cite{G1} the linear term of the deformation of
of the trace can be treated as a cocycle on the Lie algebra $\gggg$. So, the deformation procedure itself
can be regarded as a quantization of this cocycle.
(Note that the involution operator $A\to A^*$ must be also deformed.)

Recently it was understood what is an analog of the
above algebra $\ahn$ corresponding to a non-involutive braiding (of Hecke type) and what is its semiclassical counterpart.

Let $R:V^{\otimes 2}\to V^{\otimes 2}$ be a Hecke symmetry, i.e. a braiding which meets the Hecke relation
$$(q\,I-R)({q^{-1}} I+R)=0,\,\,q\in\K\,\,{\rm is\,\, generic}. $$
The algebra generated by the unit and elements $l_i^j,\,\,1\leq i, j\leq n=\dim \, V$
subject to the equation
$$RL_1 RL_1-L_1 R L_1 R={\hbar}(RL_1-L_1 R),$$
where $L=(l_i^j)$ is the matrix with entries $l_i^j$ and $L_1=L\otimes 1$ is call modified Reflection Equation Algebra (mREA).

If the Hecke symmetry $R$ comes from the quantum group $\Uq$ it is a one parameter deformation of the usual flip. In this
 case the mREA is two parameter deformation of the commutative algebra $\Sym(gl(n))$ (which is a
specialization of the $\ahq$ at ${\hbar}=0,\,q=1$). Its semiclassical counterpart is a Poisson pencil similar to that above but
 with the bracket $\{\,,\,\}_r$ defined in another way. Namely, it is an extension to the ambient vector
 space of the so-called Semenov-Tian-Shansky bracket defined on the group $SL(n)$. Similarly to the pencil above
 the latter one can be also restricted to any $GL(n)$-orbit in $gl(n)^*$.

The algebra $\ahq$ possesses a braided bi-algebra structure and  has the same category of finite dimensional
 representations as the quantum group $\Uq$ has. However, in contrast with the above monoidal tensor category
 this one is quasitensor one. Nevertheless, an intrinsic trace $\Tr_R$
 which is a categorical morphism and a deformation of the usual trace is well defined
on any object $\End(V)$ of internal endomorphisms. (For simple objects $V$ it is unique up to a factor.)
\markboth{Alexander E.~Guterman}{Quantization with a deformed trace}

For instance, if $V$ is the basic space then $\Tr_R\, l_i^j=\delta_i^j$.
 The defining relations of the algebra $\ahq$ can be rewritten as follows
$$X\otimes Y-Q(X\otimes Y)=[X,Y],\,\,X,Y\in {\bf L}=\spann(l_i^j)$$
where $Q:{\bf L}^{\otimes 2}\to {\bf L}^{\otimes 2}$ is a braiding and $[\,,\,]:{\bf L}^{\otimes 2}\to {\bf L}$ is a "braided
 Lie bracket". We would like to emphasize that the latter form of the mREA makes it
  more similar to  an enveloping algebra.

 It is easy to check that
 $$\Tr_R(X\circ Y)=\Tr_R\circ Q(X\otimes Y)$$
 where $\circ$ stands for the usual product in the algebra $\End(V)$ (note that the space ${\bf L}$ can be naturally
 identified with $\End(V)$).
 So, we can see that such a trace is Q-symmetric.

A more detailed presentation of the topic can be found in the paper \cite{GPS}.

In my talk I shall exhibit the role of a deformed (quantum) trace in
"braided geometry".

\renewcommand{\thedefinition}{\arabic{definition}}
\renewcommand{\theexample}{\arabic{example}}
\renewcommand{\thetheorem}{\arabic{theorem}}
\setcounter{theorem}{0}
\setcounter{section}{0}
\setcounter{footnote}{0}
\nachaloe{Alexander E.~Guterman}{Transformations preserving matrix invariants over semirings}%
{Partially supported by the RFBR grant 05-01-01048 and the grant MK-2718.2007.1.}
\label{gut-abs}

The investigations of matrix transformations which leave fixed different matrix properties and invariants is an actively developing part of matrix theory. This research was started in the works by Frobenius, see~\cite[Theorem 1.1]{Fr,LP}, and Dieudonn\'e, see~\cite[Theorem 1.2]{Di2,LP}, where bijective linear transformations on matrices over fields which preserve the determinant and the set of singular matrices, correspondingly, were characterized. 

During the last three decades many authors investigated linear transformations on more general algebraic structures, such as matrices over rings and semirings. 
In this talk we are going to discuss the corresponding problems on max-algebras and related classes of semirings. 

\begin{definition}
A {\em semiring\/} is a set $\SS$ with two binary operations, addition
and multiplication, such that:
\begin{itemize}
\item $\SS$ is an abelian monoid under addition (identity denoted by 0);
\item $\SS$ is a semigroup under multiplication (identity, if any, denoted
by 1);
\item multiplication is distributive over addition on both sides;
\item $s0=0s=0$ for all $s\in \SS$.
\end{itemize}
In this paper we will always assume that there is a multiplicative identity
1 in~$\SS$ which is different from 0.
\end{definition}
\markboth{Alexander E.~Guterman}{Transformations preserving matrix invariants over semirings}

\begin{definition}
A semiring $\SS$ is called {\em commutative\/} if the multiplication in $\SS$ is commutative.
\end{definition}

\begin{definition}
A semiring $\SS$ is called {\em antinegative\/} (or {\em zero-sum-free\/}) if $a+b=0$ implies that $a=b=0$.
\end{definition}
This means that the zero element is the only element with an
additive inverse

\begin{definition} We say that a semiring $\SS$  {\em has no zero divisors\/} if from $ab=0$ in $\SS$ it follows that either $a=0$ or $b=0$.
\end{definition}

\begin{definition}
A semiring is called a {\em max-algebra\/} if the set $\SS$ is an ordered group with the multiplication $*$ and the order relation $\leq$, and operations in $\SS$ are defined as follows: $a+b =\max \{a,b\}$, $a\cdot b= a*b$ for any $a,b\in \SS$. 
\end{definition}

It is straightforward to see that max-algebra is antinegative. Also it does not contain zero divisors, moreover any non-zero element of a max-algebra has a multiplicative inverse.

Let $\MM_{m,n}(\SS)$ denote the set of $m \times n$ matrices
with entries from the semiring $\SS$, $\M=\MM_{n,n}(\SS)$. Under natural definitions of matrix
addition and multiplication $\MM_n(\SS)$ is obviously a semiring. Matrix theory over semirings has been an object
of intensive study during the last decades, see for example the monograph~\cite{Gl} and
references therein. The development of linear algebra over semirings certainly requires such an
important matrix invariant as the determinant function.
However it turns out that even over commutative semirings without zero divisors
the classical determinant can not be
defined as over fields and commutative rings. The main problem lies in the
fact that in
semirings which are not rings not all elements possess an additive
inverse. A natural replacement of the determinant
function for matrices over commutative semirings is the bideterminant known for many years, see~\cite{Gl}.

\begin{definition} 
A  {\em bideterminant} of a matrix $A=[a_{i,j}]\in \MM_n({\mathcal S})$ is the
pair $  (\|A\|^+,\|A\|^-)$, where
$$\|A\|^+ =\sum\limits_{\sigma\in A_n} a_{1,\sigma(1)}\cdots a_{n,\sigma(n)},\
\|A\|^- =\sum\limits_{\sigma\in S_n\setminus  A_n} a_{1,\sigma(1)}\cdots
a_{n,\sigma(n)},$$
here $S_n$ denotes the symmetric group on the set $\{1,
\ldots, n\}$, $A_n$ denotes its subgroup of even permutations.
\end{definition}

It is known that the bideterminant function possesses some natural
properties. Namely it is invariant under transposition, and for any scalar
$\alpha\in\SS$,  $(\|\alpha A\|^+,\|\alpha A\|^-)= (\alpha^n
\|A\|^+,\alpha^n \| A\|^-)$. However, some basic properties of the
determinant are no longer true for the bideterminant. 
For example, if $A$
is invertible then $\|A\|^+\ne \|A\|^-$ but the
 converse
is not always true. 
\begin{example} 
Let us consider $A=E_{1,1}+2E_{1,2}+3E_{2,1}+4E_{2,2}\in M_2(\SS)$,
where $\SS=({\mathbb Q}_+, \max , \,\cdot\,)$, namely the set of non-negative
rationals with the standard multiplication and the addition defined by $a+b=\max\{a,b\}$. Then $(\|A\|^+,\|A\|^-)=(4,6)$ but $A$ is not invertible.
\end{example}

Note that the bideterminant is not multiplicative in general. However, some weaker versions of this property are true, in particular,
$$
\|AB\|^+ + \|A\|^+\|B\|^-+ \|A\|^-\|B\|^+= \|AB\|^-+\|A\|^+\|B\|^++\|A\|^- \|B\|^-.
$$

We prove the following theorem which is a semiring analog of famous Frobenius theorem, see \cite{Fr}, on linear transformations preserving the determinant of complex matrices. 

\begin{theorem} \label{T1} {\rm \cite{BGdet}} Let $\SS$ be a commutative antinegative semiring without zero divisors and
$T: \M \to\M $ be a surjective linear transformation. Then $(\| T(X)\|^+,\| T(X)\|^-)  = (\|X\|^+,\|X\|^-)$ for all $X\in \M$ if and only if there exists permutation matrices $P,Q$ and invertible diagonal matrices $D,E$, satisfying $(\|PQ\|^+,\|PQ\|^-)=( \|DE\|^+, \|DE\|^-)=(1,0)$, such that either  $T(X)=PDXEQ$ for all $X\in \M$ or $T(X)=PDX^tEQ$ for all $X\in \M$. Here the matrices $P,Q$ are defined
uniquely and the matrices $D,E$ are defined uniquely up to an invertible scalar factor.
\end{theorem}

\begin{definition}
We say that a transformation $T: \M \to\M $ is {\em standard\/} if it is defined by $T(X)=PDXEQ$ for all $X\in \M$ or $T(X)=PDX^tEQ$ for all $X\in \M$ for certain permutational matrices $P,Q$ and diagonal matrices $D,E$.
\end{definition}

The following similar function is widely considered in combinatorial matrix theory:
\begin{definition}
A  {\em permanent} of a matrix $A=[a_{i,j}]\in \MM_n({\mathcal S})$ is 
$$\per( A ) =\sum\limits_{\sigma\in S_n} a_{1,\sigma(1)}\cdots a_{n,\sigma(n)}.$$
\end{definition}

Also the following polynomial is related to this function:

\begin{definition} The {\em rook polynomial\/}
of a matrix $A\in M_{m,n}(\SS)$ is $R_A(x)
=\sum\limits_{j\ge 0} p_j x^j$, where $p_0=1$,
$p_j$ is the sum of the permanents of all
$j\times j$ submatrices of $A$.
\end{definition}

Linear transformations preserving the rook polynomial and permanent itself were characterized by Beasley and Pullman. 

Here we provide more general result, namely, we prove that in order to characterize a transformation, it is enough to know that it preserves any single coefficient of a rook polynomial, namely we prove the following:
\begin{theorem} {\rm \cite{GUsp}} Let $T:M_n(\SS)\to M_n(\SS)$ be a surjective linear transformation and $j$, $2\le j\le n$, be fixed.
Then $T$ preserves the $j$-th coefficient
of the rook polynomial iff $T$ is standard with $p_j(DE)=1$.
\end{theorem}

The following notions of singularity are in use while dealing with matrices over semirings, as usual, we separate left and right singularity.

\begin{definition}
A matrix $A
\in \Mnm$ is said to be $\SS$-{\em right singular\/} if $A {\mathbf x} = {\mathbf 0}$ for
some nonzero ${\mathbf x} \in \SS^n$.
$A
\in \Mnm$ is $\SS$-{\em left singular\/} if  ${\mathbf x}^tA =
{\mathbf 0}^t$ for some nonzero ${\mathbf x} \in \SS^m$.
A matrix $A
\in \Mnm$ is $\SS$-{\em singular\/} if $A$ is either $\SS$-left singular or
$\SS$-right singular.
\end{definition}

The next example shows that even over antinegative commutative semirings without zero divisors there exist
matrices that are $\SS$-left singular and are not
$\SS$-right
singular or vice versa.
\begin{example} 
Let $({\mathbb R},+,\max)$ be a max-algebra,
$$A=\left[ \begin{array}{cc} 0 & 0\\ 1&
1\end{array}\right ] , B=\left[ \begin{array}{cc} 1 & 0\\ 1&
0\end{array}\right ] \in\MM_2({\mathbb R},+,\max).$$ We have that $A{\mathbf  x}={\mathbf 0}$
 forces ${\mathbf
x}={\mathbf 0}$ since $({\mathbb R},+,\max)$ is antinegative, but $[1,0]A=[0,0]$.
Similar ${\mathbf x}^tB={\mathbf 0}$ forces ${\mathbf  x}={\mathbf 0}$
while $B[0,1]^t=[0,0]^t$.
\end{example}

\begin{definition}
A matrix $A\in \Mnm$ is
$\SS$-{\em nonsingular\/} if $A$ is not $\SS$-singular.
\end{definition}

Note that
if $\SS$ is commutative and $A$ is an $\SS$-singular square matrix then $(\|A\|^+,\|A\|^-)=(0,0)$. However the following example shows that there are $\SS$-nonsingular matrices
with the bideterminant equal to $(0,0)$. 
\begin{example}  
Over any commutative antinegative semiring,
$$
\left\| \begin{array}{ccc} 0 & 0 & 1 \\ 1& 1& 0\\ 0 & 0 & 1\end{array}\right\|^+
= 0 =\left\| \begin{array}{ccc} 0 & 0 & 1 \\ 1& 1& 0\\ 0 & 0 &
1\end{array}\right\|^- \enspace .
$$
\end{example}

We obtain the following analog of Dieudonn\'e theorem on singularity preservers, see~\cite{Di2}, for matrices over semirings.
\begin{theorem} \label{S-sThm}  {\rm \cite{BGdet}}  Let $\SS$ be an antinegative semiring without zero divisors and $T: \Mnm\to\Mnm$ be a surjective
linear operator. Then the following statements are equivalent
\begin{enumerate}
\item $T$ preserves the set of $\SS$-singular matrices;
\item $T$ preserves
the set of $\SS$-nonsingular matrices;
\item There are permutational matrices $P,Q\in \M$ and a matrix $B$ with all invertible entries such that either $T(X)=P(X\circ B)Q$ for all $X\in \M$ or $T(X)=P(X\circ B)^tQ$ for all $X\in \M$, here $X\circ B$ is an Hadamard product, i.e., $(X\circ B)_{i,j}=x_{i,j}b_{i,j}$.
\end{enumerate}  \end{theorem}

If the semiring $\SS$ is also a subsemiring of an
associative ring $\RR$ without zero divisors we can consider the following notion of singularity as well.

\begin{definition}
We say that a matrix $A \in \Mnm$ is $\RR$-{\em right singular\/} if $A
{\mathbf x} = {\mathbf 0}$ for some nonzero ${\mathbf x} \in
\RR^n$. $A \in \Mnm$ is $\RR$-{\em left singular\/} if ${\mathbf
x}^t A  = {\mathbf 0}$ for some nonzero ${\mathbf x} \in
\RR^m$. $A$ is $\RR$-{\em singular\/} if $A$ is either $\RR$-left singular or $\RR$-right singular, and $\RR$-{\em nonsingular\/} if it is not $\RR$-singular.
\end{definition}

It is straightforward to see that if a semiring $\SS$ is a subsemiring of a certain ring $\RR$ then $\RR$-right (left) singularity follows from $\SS$-right (left) singularity. However the following example shows that there are $\SS$-nonsingular matrices which are $\RR$-singular.
\begin{example}
For any $n$ the matrix $J_n=\sum\limits_{i,j=1}^n E_{i,j}\in \MM_n({\mathbb Z}_+)$ is ${\mathbb Z}$-left and ${\mathbb Z}$-right singular but ${\mathbb Z}_+$-nonsingular.
\end{example}
Note that similarly to the situation over fields all non-square matrices are $\RR$-singular, however as the above example shows they may not be $\SS$-singular.

An analog of Theorem~\ref{S-sThm} holds for transformations preserving $\RR$-singularity. Corresponding transformations appear to be standard. 

\begin{definition} 
Let $\SS$ be a max-algebra (operations are denoted by $\max$ and $+$). A matrix $A=[a_{ij}]\in M_n(\SS)$ is said to be {\em tropically singular\/} if the maximum in the expression for the permanent 
$$\per(A)=\max\limits_{\sigma\in S_n} \{a_{1\sigma(1)}+\ldots +a_{n\sigma(n)}\}$$ 
is achieved at least twice.
\end{definition}

\markboth{I.~Itenberg, V.~Kharlamov, E.~Shustin}%
{Transformations preserving matrix invariants over semirings}
It can be generalized to the case of an arbitrary antinegative semiring $\SS$ in the following way:
\begin{definition} \label{DTropSing}
A matrix $A=[a_{ij}]\in M_n(\SS)$ is said to be {\em tropically singular\/} if there exists a subset ${\mathcal T}\in S_n$ such that 
$$\sum\limits_{\sigma\in {\mathcal T}} a_{1\sigma(1)}\cdots a_{n\sigma(n)} = \sum\limits_{\sigma\in S_n\setminus {\mathcal T}} a_{1\sigma(1)}\cdots a_{n\sigma(n)}.$$ 
\end{definition}

Our further results include the characterization of linear transformations preserving these and related notions of singularity.
 Also we obtain several analogs of Markus and Moyls result on linear transformations preserving rank, see~\cite[Theorems 3.1, 3.2]{LP}, for several well-known semiring rank functions.

\renewcommand{\thedefinition}{\thesection.\arabic{definition}}
\renewcommand{\theexample}{\thesection.\arabic{example}}
\renewcommand{\thetheorem}{\thesection.\arabic{theorem}}

\renewcommand{\thetheorem}{\arabic{theorem}}
\setcounter{theorem}{0}
\setcounter{section}{0}
\setcounter{footnote}{0}
\nachaloe{I.~Itenberg, V.~Kharlamov, and E.~Shustin}{Tropical geometry and enumeration of real 
rational curves}{Partially supported by the joint RFBR/CNRS grant 05-01-02807.}
\label{ite-abs}

The talk is devoted to applications of
tropical geometry
in enumerative (complex and real) algebraic geometry.
We concentrate ourselves
at enumeration of real rational curves interpolating
fixed collections of real points
in a real algebraic surface $\Sigma$,
and more precisely, at the following
question: {\it given a real divisor $D$
and a generic collection $\bw$ of $c_1(\Sigma)\cdot D-1$
real points in $\Sigma$,
how many
of the complex
rational curves belonging to the linear system $|D|$
and passing through the points of~$\bw$
are real \/}? By rational curves we mean irreducible genus zero
curves and
their degenerations,
so that they form in $|D|$ a projective subvariety $S(\Sigma, D)$;
this subvariety is called
the
{\it Severi variety}.
A curve on a real surface
$\Sigma$
is called real, if the curve is invariant under the involution
$c:\Sigma\to\Sigma$ defining
the real structure of $\Sigma$.

While, under mild\markboth{I.~Itenberg, V.~Kharlamov, E.~Shustin}{Tropical geometry and enumeration of curves}
conditions on $\Sigma$ and $D$, the number of complex curves in question is
the same for all generic collections $\bw$
(it equals to the degree of $S(\Sigma, D)$),
it is no more the case for real curves (except
few very particular situations).

J.-Y.~Welschinger~\cite{W1, W2} discovered a way to attribute
weights $\pm 1$
to the real solutions in question so that
the number of real solutions counted with weights
becomes independent of the choice of a generic collection of
real points.
As an immediate consequence,
the absolute value of the Welschinger invariant
$W_{\Sigma,D}$
provides
a lower bound on the number
$R_{\Sigma,D}(\bw)$
of real solutions:
$R_{\Sigma,D}(\bw) \ge |W_{\Sigma,D}|$.

In some cases (for example, in
the case of toric Del Pezzo surfaces; recall that
there are five
toric Del Pezzo surfaces: the projective plane
$\mathbb{P}^2$, the product
$\mathbb{P}^1 \times \mathbb{P}^1$ of projective lines,
and $\mathbb{P}^2$
blown up at $k$ points in general position, where $k = 1, 2$ or 3)
Welschinger invariants can be calculated using
Mikhalkin's approach~\cite{M1, M2}
which deals with a corresponding count of
tropical curves.
In tropical geometry, complicated non-linear
algebro-geometric objects are replaced by simpler
piecewise-linear ones. For example, tropical plane curves
are piecewise-linear graphs whose edges have rational slopes.
Tropical curves
can be seen as algebraic curves over the tropical semi-field $(\max, +)$.

Using the tropical approach, we proved (see~\cite{IKS})
the logarithmic equivalence for
the Wel\-schinger and Gromov-Witten invariants of
any toric Del Pezzo surface equipped with
its tautological real structure,
{\it i.e.}, the real structure which is
provided by the toric structure.

\begin{theorem}
{\rm (see \cite{IKS})}
Let~$\Sigma$ be a toric Del Pezzo surface
equipped with its
tautological
real structure,
and $D$ an ample divisor on $\Sigma$.
The sequences $\log W_{\Sigma, nD}$ and $\log GW_{\Sigma, nD}$,
$n\in\mathbb{N}$, of the Welschinger invariants
and the corresponding Gromov-Witten invariants
are asymptotically equivalent. More precisely,
$\log W_{\Sigma, nD} = \log GW_{\Sigma, nD} + O(n)$
and $\log GW_{\Sigma, nD} = (c_1(\Sigma) \cdot D)\cdot n\log n + O(n)$.
\end{theorem}

We also defined (see~\cite{IKS-ch}) a series of
\markboth{Semen S.~Kutateladze}{Tropical geometry and enumeration of curves}
relative tropical Welschinger-type invariants of real toric
surfaces. In the Del Pezzo case, these invariants can be seen as
real tropical analogs of relative Gromov-Witten invariants, and
are subject to recursive formulas of Caporaso-Harris type.

In the present talk, we consider generic collections of real points on
the projective plane blown up at 4 real points in general position
and prove  that the logarithmic equivalence of
the Welschinger and Gromov-Witten invariants
holds in this situation as well. 

\begin{theorem}
Let~$\Sigma$ be the projective plane $\mathbb{P}^2$
blown up at $4$ real points in general position,
and $D$ an ample divisor on $\Sigma$.
The sequences $\log W_{\Sigma, nD}$ and $\log GW_{\Sigma, nD}$,
$n\in\mathbb{N}$, of the Welschinger invariants
and the corresponding Gromov-Witten invariants
are asymptotically equivalent.
\end{theorem}

The proof is based on a new version of the correspondence theorem,
whose proof in turn uses
an appropriate tropical Caporaso-Harris type formulas.
In particular, we get recursive formulas that allow one
to calculate Welschinger invariants
of $\mathbb{P}^2$
blown up at $4$ real points in general position.

\renewcommand{\thetheorem}{\thesection.\arabic{theorem}}
\setcounter{footnote}{0}
\setcounter{section}{0}
\nachalo{Semen S.~Kutateladze}{Abstract convexity and \hbox{cone-vexing abstractions}}
\label{kut-abs}
This talk is devoted to some origins of abstract convexity and
a~few vexing limitations on the range of abstraction in convexity.
Convexity is a relatively recent subject. Although the noble
objects of Euclidean geometry are mostly convex, the abstract
notion of a convex set appears only after the Cantor paradise was
founded. The idea of convexity feeds generation, separation, calculus, and
approximation.  Generation appears as duality; separation, as optimality;
calculus,  as representation; and approximation, as stability.

\section{Generation} 

Let $\overline{E}$  be a~complete lattice
$E$ with the adjoint top $\top:=+\infty$ and bottom $\bot:=-\infty$.
Unless otherwise stated, $Y$ is usually a~{\it Kantorovich space\/}
which is a Dedekind complete vector lattice in another terminology.
Assume further that $H$  is some subset of $E$ which is by implication a~(convex)
cone in $E$, and so the bottom of $E$
lies beyond~$H$. A subset $U$  of~ $H$ is {\it convex relative to~}
$H$ or $H$-{\it convex\/}, in symbols $U\in\mathfrak{V}(H,\overline{E})$,
provided that $U$ is the $H$-{\it support set\/}
$U^H_p:=\{h\in H:h\le p\}$ of some element $p$ of $\overline{E}$.
\markboth{Semen S.~Kutateladze}{Abstract convexity and \hbox{cone-vexing abstractions}}

Alongside the $H$-convex sets we consider
the so-called $H$-convex elements. An element   $p\in \overline{E}$
is  $H$-{\it convex} provided that $p=\sup U^H_p$;~i.e., $p$
represents the supremum of the $H$-support set of~$p$.
The $H$-convex elements comprise the cone which is denoted by
$\mathscr C(H,\overline{E}$).  We may omit  the references to $H$ when $H$ is clear
from the context. It is worth noting that
convex elements and sets are ``glued together''
by the {\it Minkowski diality\/} $ \varphi:p\mapsto U^H_p$.
This duality enables us to study convex elements and sets simultaneously.

Since the classical results by Fenchel \cite{Fenchel}
and H\"ormander
\cite{Her, Notions} it has been well known
that the most convenient and conventional classes of convex  functions
and sets are  $\mathscr C(A(X),\overline{\mathbb R^X} )$ and $\mathfrak{V}(X',\overline{\mathbb R^X})$.
Here    $X$ is a locally convex space,   $X'$ is the dual of~$X$,
and $A(X)$ is the space of affine functions on  $X$
(isomorphic with  $X'\times \mathbb R$).

In the first case the Minkowski duality is the mapping
 $f\mapsto\text{epi} (f^*)$ where
$$
f^*(y):=\sup\limits_{x\in X}(\langle y,x\rangle - f(x))
$$
is the {\it Young--Fenchel transform\/} of~$f$ or the {\it conjugate function\/} of~$f$~.
In the second case we prefer to write down the inverse of the Minkowski
duality which sends  $U$  in $\mathfrak{V}(X',\overline{\mathbb R}^X)$
to the standard {\it support function}
$$
\varphi^{-1}(U):x\mapsto\sup\limits_{y\in U}
\langle y,x\rangle.
$$
As usual, $\langle\cdot,\cdot\rangle$ stands for the canonical pairing
of~$X'$ and~$X$.

This idea of abstract convexity lies behind many current objects
of analysis and geometry. Among them we list the ``economical'' sets
with boundary points meeting the Pareto criterion: capacities, monotone
seminorms, various classes of functions convex in some generalized sense,
for instance, the Bauer convexity in Choquet  theory, etc.
It is curious that there are ordered vector spaces consisting of
the convex elements with respect to  narrow cones with finite generators.
Abstract convexity is traced and reflected, for instance, in
\cite{MD}--%
\cite{IoRu}.

\section{Separation}
Consider cones
$K_1$
and
$K_2$
in a topological vector space
$X$
and put
$\varkappa:=(K_1,K_2)$.
Given a pair
$\varkappa$
define the correspondence
$\Phi_{\varkappa}$
from
$X^2$
into
$X$
by the formula
$$
\Phi_{\varkappa}:=\{(k_1,k_2,x)\in X^3:
x=k_1-k_2\in K_\imath\ (\imath:=1,2)\}.
$$
Clearly, $\Phi_{\varkappa}$ is a cone or, in other words,
a~conic correspondence.

The pair $\varkappa$ is {\it nonoblate\/}
whenever $\Phi_{\varkappa}$
is open at the zero. Since
$\Phi_{\varkappa}(V)=V\cap K_1-V\cap K_2$
for every
$V\subset X$,
the nonoblateness of $\varkappa$
means that $$
\varkappa V:=(V\cap K_1-V\cap K_2)\cap(V\cap K_2-V\cap K_1)
$$
is a zero neighborhood  for every zero neighborhood~
$V\subset X$.
Since $\varkappa V\subset V-V$, the nonoblateness of
$\varkappa$ is equivalent to the fact that the system of sets
$\{\varkappa V\}$ serves as a filterbase of zero neighborhoods while
$V$ ranges over some base of the same filter.

Let $\Delta_n:x\mapsto(x,\dots,x)$ be
 the embedding of
$X$
into the diagonal
$\Delta_n(X)$
of
$X^n$.
A pair of cones
$\varkappa:=(K_1,K_2)$
is nonoblate if and only if
$\lambda:=(K_1\times K_2,\Delta_2(X))$
is nonoblate in~$X^2$.

Cones
$K_1$
and
$K_2$
constitute a nonoblate pair if and only if the conic correspondence~
$\Phi\subset X\times X^2$
defined as
$$
\Phi:=\{(h,x_1,x_2)\in X\times X^2 :
x_\imath+h\in K_\imath\ (\imath:=1,2)\}
$$
is open at the zero. Recall that a convex correspondence
$\Phi$ from $X$
into $Y$ is open at the zero if and only if the H\"ormander transform
of $X\times\Phi$ and the cone
$\Delta_2(X)\times\{0\}\times\mathbb R^+$
constitute a nonoblate pair in~$X^2\times Y\times\mathbb R$.

Cones~
$K_1$
and~
$K_2$
in a topological vector space~
$X$
are {\it in general position\/}
provided that

{\bf (1)}~
the algebraic span of $K_1$
and~
$K_2$
is some subspace
$X_0\subset X$;
i.e.,
$X_0=K_1-K_2=K_2-K_1$;

{\bf (2)}~the subspace
$X_0$
is complemented; i.e., there exists a continuous projection
$P:X\rightarrow X$
such that
$P(X)=X_0$;

{\bf (3)}~$K_1$
and~
$K_2$
constitute a nonoblate pair in~
$X_0$.

Let $\sigma_n$
stand for the rearrangement of coordinates
$$
\sigma_n:((x_1,y_1),\dots, (x_n,y_n))\mapsto ((x_1,\dots,x_n),
(y_1,\dots,y_n))
$$
which establishes an isomorphism between
$(X\times Y)^n$
and
$X^n\times Y^n$.

Sublinear operators  $P_1,\dots,P_n:X\rightarrow E\cup \{+\infty\}$
are {\it in general position\/} if so are
the cones $\Delta_n(X)\times E^n$
and
$\sigma_n(\epi (P_1)\times\dots\times\epi (P_n))$.
A similar terminology applies to convex operators.

Given a cone $K\subset X$, put
$$
\pi_E(K):=\{T\in\mathscr L(X,E): Tk\leq 0\ (k\in K)\}.
$$
We readily see that
$\pi_E(K)$
is a cone in
$\mathscr L(X,E)$.

{\scshape Theorem.} {\sl Let
$K_1,\dots,K_n$
be cones in a topological vector space~
$X$
and let
$E$
be a topological Kantorovich space.  If
$K_1,\dots,K_n$
are in general position then}
$$
\pi_E(K_1\cap\dots\cap K_n)=\pi_E(K_1)+\dots+\pi_E(K_n).
$$
This formula opens  a way to various separation results.

{\scshape Sandwich Theorem.} {\sl Let
$P,Q:X\rightarrow E\cup \{+\infty\} $
be sublinear operators in general position.
If
$P(x)+Q(x)\geq 0$
for all
$x\in X$
then there exists a~continuous linear operator~
$T:X\rightarrow E$
such that}
$$
-Q(x)\leq Tx\leq P(x)\quad  (x\in X).
$$

Many efforts were made to abstract these results to
a more general algebraic setting and, primarily,
to semigroups. The relevant separation results are
collected  in~\cite{Fuch}.

\section{Calculus}
Consider a~Kantorovich space
$E$
and an arbitrary nonempty set
$\mathfrak A$.
Denote by
$l_\infty (\mathfrak A,E)$
the set of all order bounded mappings from
$\mathfrak A$
into $E$; i.e.,
$f\in l_\infty (\mathfrak A,E)$
if and only if
$f:\mathfrak A \to E$
and the set
$\{f(\alpha):\alpha\in\mathfrak A\}$
is order bounded in $E$.
It is easy to verify that
$l_\infty (\mathfrak A,E)$ becomes a Kantorovich
space if endowed with the coordinatewise algebraic operations and order.
The operator
$\varepsilon_{\mathfrak A, E}$
acting from
$l_\infty (\mathfrak A,E)$
into
$E$
by the rule
$$
\varepsilon_{\mathfrak A, E}:f\mapsto\sup \{f(\alpha):\alpha\in\mathfrak A\}
\quad (f\in l_\infty (\mathfrak A,E))
$$
is called the {\it canonical sublinear operator\/}
given
$\mathfrak A$
and
$E$.
We often write
$\varepsilon_{\mathfrak A}$
instead of
$\varepsilon_{\mathfrak A, E}$
when it is clear from the context what Kantorovich space is meant.
The notation
$\varepsilon_n$
is used when the cardinality of~$\mathfrak A$
equals
$n$ and we call the operator
$\varepsilon_n$
{\it finitely-generated}.

Let
$X$
and
$E$
be ordered vector spaces. An operator
$p: X\to E$
is called {\it increasing\/} or
{\it isotonic\/}
 if for all
$x_1, x_2 \in X$
from
$x_1\leq x_2$
it follows
that
$p(x_1)\leq p(x_2)$.
An increasing linear operator is also called {\it positive}.
As usual, the collection of all positive linear operators in the space
$L(X,E)$ of all linear operators is denoted
by
$L^+ (X,E)$.
Obviously, the positivity of a~linear operator
$T$ amounts to the
inclusion
$T(X^+) \subset E^+$,
where
$X^+:=\{x\in X: x\geq 0\}$
and
$E^+:=\{e\in E: e\geq 0\}$
are the
{\it positive cones\/}
in
$X$
and
$E$
respectively.
Observe that every canonical operator is increasing and sublinear,
while every finitely-generated canonical operator is order continuous.

Recall that
$\partial p:=\partial p(0)=\{ T \in L (X,E):$ $(\forall x
\in X  )\  T x\leq p(x)\}$
is the {\it subdifferential\/}
at the zero
or  {\it support
set\/}
of a~sublinear operator~
$p$.

Consider a~set~
$\mathfrak A$
of linear operators acting from
a~vector space
$X$
into a~Kantorovich space
$E$.
The set
$\mathfrak A$
is {\it weakly order  bounded\/} if
the set
$\{\alpha x:\alpha\in\mathfrak A\}$
is order bounded for every
$x\in X$.
We denote by
$\langle\mathfrak A\rangle x$
the mapping that assigns the element
$\alpha x\in E$
to each
$\alpha\in \mathfrak A$,
i.e.
$\langle\mathfrak A\rangle x: \alpha\mapsto\alpha x$.
If
$\mathfrak A$
is weakly order bounded then
$\langle\mathfrak A\rangle x\in l_\infty (\mathfrak A,E)$
for every fixed
$x\in X$.
Consequently, we obtain the linear operator
$\langle\mathfrak A\rangle:X\rightarrow l_\infty (\mathfrak A,E)$
that acts as
$\langle\mathfrak A\rangle:x\mapsto\langle\mathfrak A\rangle x$.
Associate with
$\mathfrak A$
one more operator
$$
p_{\mathfrak A}: x\mapsto\sup \{\alpha x: \alpha\in\mathfrak A\}\quad(x\in X).
$$
The operator
$p_{\mathfrak A}$
is sublinear. The support set
$\partial p_{\mathfrak A}$
is denoted by
$\cop (\mathfrak A)$
and referred to as  the {\it support hull\/} of
$\mathfrak A$.
These definitions  entail the following

{\scshape Theorem.} {\sl If
$p$
is a~sublinear operator with
$\partial p=\cop (\mathfrak A)$
then $
P=\varepsilon_{\mathfrak A}\circ \langle\mathfrak A\rangle.
$
Assume further that
$p_1: X\to E$
is a~sublinear operator and
$p_2: E\to F$
is an increasing sublinear operator. Then
$$
\partial (p_2\circ p_1)=\left\{ T\circ\langle\partial p_1\rangle: T\in L^+
(l_{\infty}(\partial p_1, E),F)\ \wedge\ T\circ\Delta_{\partial p_1}\in
\partial p_2 \right\}.
$$
Furthermore, if
$\partial p_1=\cop (\mathfrak A_1)$
and
$\partial p_2=\cop (\mathfrak A_2)$
then}
$$
\gathered
\partial (p_2\circ p_1)
=\bigl\{T\circ\langle\mathfrak A_1\rangle : T\in L^+
(l_{\infty}(\mathfrak A_1,E),F)\
\\
\wedge\
\left(\exists\alpha\in\partial\varepsilon_{\mathfrak A_2}\bigr)\
T\circ\Delta_{\mathfrak A_1}=\alpha\circ\langle\mathfrak A_2\rangle\right\}.
\endgathered
$$

More details on subdifferential calculus and applications to optimality
are collected in~\cite{Subdiff}.

\section{Approximation} 

Study of stability in abstract convexity
is accomplished sometimes by introducing various  epsilons
in appropriate places. One of the earliest attempts in this direction
is connected with the classical  Hyers--Ulam stability theorem for
$\varepsilon$-convex functions. The most recent results are collected
in~\cite{Almost}. Exact calculations with epsilons and sharp estimates
are sometimes bulky and slightly mysterious.  Some alternatives are suggested
by actual infinities, which is illustrated with the conception
of {\it infinitesimal optimality}.

Assume given a ~convex operator
$f:X\to E\cup{+\infty}$
and a~ point
$\overline x$
in the effective domain
$\dom(f):=\{x\in X:f(x)<+\infty\}$
of
~$f$.
Given
$\varepsilon \ge 0$
in the positive cone
$E_+$
of
$E$,
by the
$\varepsilon $-{\it subdifferential\/}
of~$f$
at
~$\overline x$
we mean the set
$$
\partial\, {}^\varepsilon\!f(\overline x):=\big\{T\in L(X,E):
(\forall x\in X)(Tx-Fx\le T\overline x -f\overline x+\varepsilon) \big\},
$$
with
$L(X,E)$
standing as usual for the space of linear operators
from~
$X$
to
~$E$.

Distinguish
some downward-filtered subset
~$\mathscr E$ of
$E$
that is composed of positive elements.
Assuming
$E$ and~$\mathscr E$
 standard, define the {\it monad\/}
$\mu (\mathscr E)$ of $\mathscr E$ as
$\mu (\mathscr E):=\bigcap\{[0,\varepsilon ]:\varepsilon \in
{}^\circ\!\mathscr E\}$.
The members of $\mu(\mathscr E)$ are {\it positive
infinitesimals\/}  with respect to~$\mathscr E $.
As usual,
${}^\circ\!\mathscr E$
denotes the external set of~
all standard members of
~$E$,
the {\it standard part\/} of
~$\mathscr E$.

We will agree that the monad $\mu (\mathscr E )$
is an external cone over ${}^\circ \mathbb R $ and, moreover,
$\mu (\mathscr E)\cap{}^\circ\! E=0$.
In application, $\mathscr E $ is usually
the filter of order-units of $E$.
The relation of
{\it infinite proximity\/} or
{\it infinite closeness\/}
between the members of $E$ is introduced as follows:
$$
e_1 \approx e_2 \leftrightarrow e_1 -e_2 \in\mu
(\mathscr E )\wedge e_2 -e_1 \in\mu (\mathscr E ).
$$

Since
$$
\bigcap\limits_{\varepsilon \in{}^\circ \mathscr E }\,
\partial _\varepsilon f(\overline x)=
\bigcup\limits_{\varepsilon \in\mu (\mathscr E )}\,
\partial _\varepsilon f(\overline x);
$$
therefore, the external set on both sides  is
 the so-called {\it infinitesimal subdifferential} of
$f$ at  $\overline x$. We denote this set by
$Df(\overline x)$.
The elements of
$Df(\overline x)$ are
{\it infinitesimal subgradients}
of $f$ at
~$\overline x$.
If the zero oiperator is an infinitesimal subgradient
of $f$ at $\overline x$ then $\overline x$ is called
an {\it infinitesimal minimum point\/} of $f$.
We abstain from indicating  $\mathscr E$ explicitly
since this leads to no confusion.

\markboth{Grigory L.~Litvinov}{Abstract convexity and \hbox{cone-vexing abstractions}}

{\scshape Theorem.} {\sl
Let $f_1:X\times Y\rightarrow E\cup +\infty$ and
$f_2:Y\times Z\rightarrow E\cup +\infty$ be convex operators.
Suppose that the convolution
$f_2\vartriangle f_1$ is infinitesimally exact at some point $(x,y,z)$; i.e.,
$
(f_2\vartriangle f_1)(x,y)\approx f_1(x,y)+f_2(y,z).
$
If, moreover, the
convex sets $\epi(f_1,Z)$ and $\epi(X,f_2)$ are
in general
position then}
$$
D(f_2\vartriangle f_1)(x,y)=
Df_2(y,z)\circ Df_1(x,y).
$$


\bibliographystyle{plain}

\nachaloe{Grigory L.~Litvinov}{Interval analysis for algorithms of idempotent and tropical
mathematics}{The work has been supported 
by the joint RFBR/CNRS grant 05-01-02807 and by the 
RFBR grant 05-01-00824.}
\label{lit-abs}
The idempotent interval analysis appears to be best suited for treating problems with 
order-preserving transformations of input data [1, 2]. It gives exact interval solutions 
to optimization problems with interval uncertainties in input data without any conditions 
\markboth{Grigory L.~Litvinov}{Interval analysis for algorithms of idempotent mathematics}
of smallness on uncertainty intervals. Our aim to generalize results presented 
in [1, 2] for a very general case of arbitrary algoritms of idempotent mathematics (in 
particular, tropical mathematics) and algorithms over positive semirings (the semifield 
of all nonnegative real numbers with usual operations is a typical positive semiring). 
Algorithms of this type are generated by a collection of basic semiring/semifield operations, 
the well known star-operations $x \mapsto x^*= 1\oplus x\oplus x^2\oplus x^3\oplus\ldots,$ 
and trivial operations. 
\vskip6pt
{\sc Theorem.} {\sl Every algorithm of idempotent mathematics (and every algorithm over positive 
semirings) has an interval version. The complexity of this interval version coincides with 
the complexity of the initial algorithm. The interval version of the algorithm gives exact 
interval estimates for the corresponding output data.}
\vskip6pt
See [1, 2] for some examples. 

Note that for the traditional interval analysis the situation is opposite. For example, 
basic algorithms of the traditional linear algebra are plynomial but the corresponding 
interval versions are NP-hard and interval estimates are not exact.

\newpage
\setcounter{section}{0}
\setcounter{footnote}{0}
\nachaloe{G.L. Litvinov and G.B. Shpiz}{Dequantization procedures related to the Maslov dequantization}%
{This work has been supported by the
RFBR grant 05-01-00824 and the joint RFBR/CNRS grant 05-01-02807.}
\label{lit-shp-abs}

\section{The Maslov dequantization}

Let $\R$ and $\C$ be the fields of real and complex numbers. The 
well-known max-plus algebra $\R_{\max}= \R\cup\{-\infty\}$ is defined by
the operations 
$x\oplus y=\max\{x, y\}$ and $x\odot y= x+y$. 

The max-plus algebra can be 
treated as a result of the 
{\it Maslov dequantization} of the semifield $\R_+$ of all nonnegative numbers, 
see, 
e.g., [1,2]. The change of variables
\begin{equation}
x\mapsto u=h\log x,
\end{equation}
where $h>0$, defines a map $\Phi_h\colon \R_+\to \R\cup\{-\infty\}$. Let the 
addition 
and multiplication operations be mapped from 
\markboth{G.L. Litvinov, G.B. Shpiz}{Dequantization procedures related to the Maslov dequantization}
$\R_+$ to $\R\cup\{-\infty\}$ by 
$\Phi_h$, i.e.\ let 
\begin{eqnarray*}
u\oplus_h v = h \log({\mbox{exp}}(u/h)+{\mbox{exp}}(v/h)),\quad u\odot v= u+ v,\\ 
\mathbf{0}=-\infty = \Phi_h(0),\quad \mathbf{1}= 0 = \Phi_h(1). 
\end{eqnarray*}
It can easily be checked that 
$u\oplus_h v\to \max\{u, v\}$ as $h\to 0$. Thus we get the semifield 
$\R_{\max}$ (i.e.\ the max-plus algebra) with zero $\mathbf{0}= -\infty$ and unit 
$\mathbf{1}=0$ as a result of this deformation of the algebraic structure in 
$\R_+$. 

The semifield $\R_{\max}$ is a typical example of an {\it
 idempotent semiring}; this is a semiring with idempotent addition, i.e.,
 $x\oplus x = x$ for arbitrary element
 $x$ of this semiring, see, e.g., [3-5].

The analogy with quantization is obvious; the parameter $h$ plays the role of 
the Planck constant [2]. The map $x\mapsto|x|$ and the 
Maslov dequantization for $\R_+$ give us a natural passage from the field 
$\C$ (or $\R$) to the max-plus algebra $\R_{\max}$. Following [4], 
{\it we will also call this
passage the Maslov dequantization}. In fact the 
Maslov dequantization is the usual Schr\"odinger dequantization but for  imaginary values of the Planck 
constant (see, e.g., [4]). The passage from numerical fields to the max-plus algebra
$\R_{\max}$ (or similar semifields) in mathematical constructions and results generates the so called
{\it tropical mathematics}. The so-called {\it idempotent dequantization} is a generalization
of the Maslov dequantization; idempotent dequantization generates the so-called idempotent mathematics, see,
e.g. [4] for details.

\section{The dequantization transform}

This transform is defined in [6]. 

Let $X$ be a topological space. For functions $f(x)$ defined on $X$
we shall say that a certain property is valid {\it almost everywhere} (a.e.) if 
it is valid for all elements $x$ of an open dense subset of $X$. 
Suppose $X$ is $\C^n$ or $\R^n$; denote by $\R^n_+$ the set
$x=\{\,(x_1, \dots, x_n)\in X \mid x_i\geq 0$ for $i = 1, 2, \dots, n$.
 For $x= (x_1, \dots, x_n) \in X$ we set 
${\mbox{exp}}(x) = ({\mbox{exp}}(x_1), \dots,\linebreak {\mbox{exp}}(x_n))$;
so if $x\in\R^n$, then ${\mbox{exp}}(x)\in \R^n_+$. 

Denote by $\cF(\C^n)$ the set of all functions defined
and continuous on an open dense subset $U\subset \C^n$
such that $U\supset \R^n_+$. 
It is clear that $\cF(\C^n)$ is a ring 
(and an algebra over $\C$) with respect to the usual addition
and multiplications of functions.

For $f\in \cF(\C^n)$ let us define the function $\hat f_h$
by the following formula:
\begin{equation}
\hat f_h(x) = h \log|f({\mbox{exp}}(x/h))|,
\end{equation}
where $h$ is a (small) real positive parameter and $x\in\R^n$. Set
\begin{equation}
\label{e:hatfx}
\hat f(x) = \lim_{h\to +0} \hat f_h (x),
\end{equation}
if the right-hand part of \eqref{e:hatfx} exists almost everywhere. We shall say that the 
function $\hat f(x)$ is a {\it dequantization} of the function $f(x)$ and the 
map $f(x)\mapsto \hat f(x)$ is a {\it dequantization transform}. By 
construction, $\hat f_h(x)$ and $\hat f(x)$ can be treated as functions taking their 
values in $\R_{\max}$. Note that in fact $\hat f_h(x)$ and $\hat f(x)$
 depend on the restriction of $f$ to $\R_+^n $
only; so in fact the dequantization transform is constructed
for functions defined on $\R^n_+$ only. 
It is clear that the dequantization transform is 
generated by the Maslov dequantization and the map $x\mapsto |x|$. Of 
course, similar definitions can be given for functions defined on $\R^n$
and $\R_+^n$.

Denote by $\partial \hat f$ the subdifferential of the function $\hat f$ at the
origin.

It is well known that all the convex compact subsets in $\R^n$ form an 
idempotent 
semiring $\mathcal{S}$ with respect to the Minkowski operations: for $A, B \in \mathcal{S}$ 
the sum $A\oplus B$ is the convex hull of the union $A\cup B$; the product 
$A\odot B$ is defined in the following way: $A\odot B = \{\, x\mid x = a+b$, 
where $a\in A, b\in B$. In fact $\mathcal{S}$ is an idempotent linear space over
 $\R_{\max}$ (see, e.g., [4]). Of course, the Newton polytopes 
in $V$ form a 
subsemiring $\mathcal{N}$ in $\mathcal{S}$. If $f$, $g$ are polynomials, then 
$\partial(\widehat{fg}) = \partial\hat f\odot\partial\widehat g$; moreover, if $f$ and $g$ are 
``in general position'', then $\partial(\widehat{f+g}) = \partial\hat 
f\oplus\partial\widehat g$. For the semiring of all polynomials with 
nonnegative coefficients the dequantization transform is a homomorphism of this 
``traditional'' semiring to the idempotent semiring $\mathcal{N}$.

\begin{theorem}
If $f$ is a polynomial, then the subdifferential
$\partial\hat f$ of $\hat f$ at the origin coincides with the
Newton polytope of $f$.  For the semiring of polynomials with
nonnegative coefficients, the transform $f\mapsto\partial\hat f$
is a homomorphism of this semiring to the semiring of
convex polytopes with respect to the well-known
Minkowski operations.
\end{theorem}

Using the dequantization transform it is possible to
generalize this result to a wide class of functions and
convex sets, see [6]. Another approach based on complex analysis is due to
A.~Rashkovskii, see, e.g., [7,8]. 

\section{Dequantization of linear operators and semigroups of linear operators}

The dequantization transform can be rewritten in the following form:
$$
f\mapsto \hat{f}(x) = \lim_{h\to +0} h \log(\mid f({\mbox{exp}}(x/h)\mid)
= \lim_{s\to +\infty} (1/s) \cdot \log (\mid f ({\mbox{exp}}(sx)\mid),
$$
where $x\in \Bbb R^n$, and $h$, $s=1/h$ are real positive parameters.

Our aim is to apply the dequantization transform to matrix
elements of operator semigroups generated by linear operators.

\medskip

Suppose that $S$ is a semigroup and $s\mapsto \pi_s$ is a linear
representation of $S$ in a complete (or quasicomplete) barreled
locally convex space (by continuous operators). Denote
by $V'$ the dual space to $V$ and by $\langle v', v\rangle$ the
value of a functional $v\in V'$ on an element $v\in V$. If
$s\mapsto \pi_{v',v}(s)= \langle v', \pi_sv\rangle$ is a matrix
element of $\pi$, then its {\it dequantization} $\widehat{\pi}_{v',v}$
is defined by the formula: 
$$
\widehat{\pi}_{v',v} = 
{\overline{\lim}_{s\to\infty}}(1/s)\cdot\log (\mid\langle v',
\pi_s v\rangle\mid).
$$

We discuss the cases $S = \Bbb R_+$ or $S = \Bbb Z_+$. If
$v$ and $v'$ are fixed, then $\widehat{\pi}_{v',v}\in S\cup\{\infty\}$.

\begin{proposition} Let $A$ be a linear operator in $V$,
$\pi_s={\mbox{exp}} (sA)$, and $\dim V<\infty$. Then the set of all
dequantizations $\{ \widehat{\pi}_{v',v}\}$ coincides with the set
of real parts of all eigenvalues of $A$.
\end{proposition}

There are generalizations of this result for the case $\dim V =
\infty$.

\medskip

Suppose that for every $v'\in V'$ there exists a number $r>0$
such that the set $\{r^{-s}\cdot (\mid \langle v', \pi_s v \rangle \mid),
s\in S\}$ is bounded for every $v\in V$ and the number $r$ does
not depend on $v'\in V'$. Then the representation $\pi$
is called {\it exponential}. Note, that if $V$ is a Banach
space and $\pi$ is weakly continuous, then $\pi$ is exponential.
In the general case the {\it spectral radius} $\rho_{\pi}$ of $\pi$
is defined by the formula:
$$
\rho_{\pi} = \inf \{r\mid r^{-s}\pi_s v\to 0
\hbox{ weakly for every } v\in V \hbox{ as } s\to +\infty\}
$$

\begin{proposition} If $A$ is a bounded linear operator in a 
Banach space $V$, $S = \Bbb Z_+$, $\pi = A^s$, then $\rho_{\pi} =
\rho(A) = \lim_{s\to\infty} \Vert A^s\Vert ^{1/s}$, i.e.
$\rho_{\pi}$ is the traditional spectral radius of $A$.
\end{proposition}

\begin{theorem} If $\pi$ is exponential, then
$$
\log\rho_{\pi} = \sup\{\widehat{\pi}_{v',v}\mid v'\in V', v\in V\}.
$$
\end{theorem}

\begin{theorem} Suppose that $A$  is a compact operator and
$\pi_s = A^s$, where $s\in S = \Bbb Z_+$. Then the set 
$\{\widehat{\pi}_{v',v}\}$
of all dequantizations of $\pi$ coincides with the set of all
numbers of the form $\log(\mid\lambda\mid)$, where $\lambda$ 
runs the spectrum of $A$.
\end{theorem}

\section{Dequantization of set functions on metric spaces}

Let $M$ be a metric space, $S$ its arbitrary subset with a compact closure. It is well-known that
a Euclidean $d$-dimensional ball $B_{\rho}$ of radius $\rho$ has volume
$$
\operatorname{vol}_d(B_{\rho})=\frac{\Gamma(1/2)^d}{\Gamma(1+d/2)}\rho^d,
$$
where $d$ is a natural parameter. By means of this formula it is possible to define
a volume of $B_{\rho}$ for any {\it real} $d$ [9]. Cover $S$ by a finite number of balls
of radii $\rho_m$. Set
$$
v_d(S):=\lim_{\rho\to 0} \inf_{\rho_m<\rho} \sum_m \operatorname{vol}_d(B_{\rho_m}).
$$
Then there exists a number $D$ such that $v_d(S)=0$ for $d>D$ and $v_d(S)=\infty$ for $d<D$.
This number $D$ is called the {\it Hausdorff-Besicovich dimension} (or {\it HB-dimension}) of $S$ [9].
Note that a set of non-integral HB-dimension is called a fractal in the sense of B.~Mandelbrot.

Denote by $\cN_{\rho}(S)$ the minimal number of balls of radius $\rho$ covering $S$.
Then
$$
D(S)=\underline{\lim}_{\rho\to +0} \log_{\rho} (\cN_{\rho}(S)^{-1}),
$$
where $D(S)$ is the HB-dimension of $S$. Set $\rho=e^{-s}$, then
$$
D(S)=\underline{\lim}_{s\to +\infty} (1/s) \cdot \log \cN_{exp(-s)}(S).
$$
So the HB-dimension $D(S)$ can be treated as a result of a dequantization of the set function
$\cN_{\rho}(S)$.

Let $\mu$ be a set function on $M$ (e.g., a probability measure) and suppose that $\mu(B_{\rho})<\infty$
for every ball $B_{\rho}$. Let $B_{x,\rho}$ be a ball of radius $\rho$ having the point $x\in M$
as its center. Then define $\mu_x(\rho):=\mu(B_{x,\rho})$ and let
$$
D_{x,\mu}:=\underline{\lim}_{s\to +\infty} -(1/s)\cdot\log (|\mu_x(e^{-s})|).
$$
This number could be treated as a dimension of $M$ at the point $x$ with respect to the set function $\mu$.
There are many dequantization procedures of this type in different mathematical areas.
In particular, V.P.~Maslov's negative dimension [10] can be treated similarly.

\section{Dequantization of the Fourier-Laplace transform}

It was noticed by V.P.~Maslov (see, e.g., [1-4]) that the Legendre (or Legendre-Fenchel) 
transform can be
treated as an idempotent (or tropical) version of the Fourier-Laplace transform. It seems to be interesting to note that
the Legendre transform can be constructed from the Fourier-Legendre transform directly
by means of the Maslov dequantization.

\section{Dequantization of geometry}

An idempotent version of real
algebraic geometry was discovered in the report of O.~Viro for the
Barcelona Congress [11]. Starting from the idempotent 
correspondence principle [2], O.~Viro constructed a piecewise-linear
geometry of polyhedra of a special kind in finite dimensional 
Euclidean spaces as a result of the
Maslov dequantization of real algebraic geometry. He indicated
important applications in real algebraic geometry (e.g., 
in the framework of Hilbert's 16th problems) and 
relations to complex algebraic geometry and
amoebas in the sense of I.~M.~Gelfand, M.~M.~Kapranov, and
A.~V.~Zelevinsky.
Then complex algebraic geometry was dequantized by 
G.~Mikhalkin and the result turned out
to be the same; now the new geometry is called {\it tropical algebraic geometry}.
In particular, tropical varieties are results of a dequantization procedure 
(generated by the Maslov dequantization) applied to algebraic varieties.
There are many applications, see, e.g., [11-13,5].

\section{Remark}

It would be nice to find new dequantization procedures related to the Maslov dequantization.

\end{document}